\documentclass[a4paper,12pt,reqno,]{scrartcl}
\usepackage[latin1]{inputenc}
\usepackage{amsmath, amsthm, amssymb}
\usepackage{nicefrac}

\newtheorem{satz}{Theorem}[section]
\newtheorem{hsatz}[satz]{Proposition}
\newtheorem{lem}[satz]{Lemma}
\newtheorem{folg}[satz]{Corollary}
\newtheorem{ann}[satz]{Assumption}
\newtheorem{defn}[satz]{Definition}

\theoremstyle{remark}
\newtheorem{bem}[satz]{Remark}
\newtheorem{bsp}[satz]{Example}

\theoremstyle{definition}

\usepackage{graphicx}
\usepackage{enumitem}

\newtheorem{schritt}{Step}

\newcommand{\kz}{\mathbb C}
\newcommand{\rz}{\mathbb R}
\newcommand{\nz}{\mathbb N}
\newcommand{\gz}{\mathbb Z}

\newcommand{\fa}{\text{ } \forall \text{ }}

\newcommand{\supp}{\operatorname{supp}}
\newcommand{\dom}{\operatorname{D}}
\newcommand{\hil}{\mathcal H}

\newcommand{\re}{\operatorname{Re}} 
\newcommand{\im}{\operatorname{Im}} 
\newcommand{\dist}{\operatorname{dist}}
\newcommand{\leos}{\sigma_0}
\newcommand{\rgeom}{\mathfrak r_\mathrm{G}}
\newcommand{\rone}{\mathfrak r_1}
\newcommand{\rtwo}{\mathfrak r_2}
\newcommand{\rthree}{\mathfrak r_3}
\newcommand{\rfour}{\mathfrak r_4}
\newcommand{\rfive}{\mathfrak r_5}
\newcommand{\rsix}{\mathfrak r_6}
\newcommand{\rseven}{\mathfrak r_7}
\newcommand{\reight}{\mathfrak r_8}
\newcommand{\rnine}{\mathfrak r_9}
\newcommand{\rten}{\mathfrak r_{10}}
\newcommand{\rzero}{\mathfrak r_0}

\newcommand{\with}{\text{ with }}
\newcommand{\esgibt}{\exists \text{ }}
\newcommand{\res}{\varrho}

\newcommand{\skp}[3]{\left\langle #1,#2 \right\rangle_{#3}}

\newcommand{\nm}[1]{\hspace{#1}-\hspace{#1}}

\newcommand{\ep}{\varepsilon}
\newcommand{\cp}{c_\mathrm{p}} 
\newcommand{\pot}{\mathcal V} 
\newcommand{\dichte}{\varrho} 
\newcommand{\PP}{\mathbb P}
\newcommand{\form}{\mathfrak h} 
\newcommand{\formk}{\mathfrak k} 
\newcommand{\forms}{{\mathfrak h_\omega^\star}} 

\newcommand{\metrik}{\mathfrak d} 
\newcommand{\we}{w} 
\newcommand{\vol}{\operatorname{vol}} 
\newcommand{\Tr}{\operatorname{Tr}} 

\newcommand{\tr}[1]{\operatorname{tr}_{ #1 }} 
\newcommand{\str}[1]{\operatorname{str}_{#1}} 
\newcommand{\sgn}{\operatorname{sgn}} 
\newcommand{\lok}{_{\mathrm{loc}}}
\newcommand{\komp}{_{\mathrm{comp}}}

\usepackage{bbm}
\newcommand{\ind}{\mathbbmss 1}

\newcommand{\okl}{U}
\newcommand{\ukl}{u}

\newcommand{\ik}[1]{V_{#1,\operatorname{int}}} 
\newcommand{\rk}[1]{V_{#1,\operatorname{\partial}}} 

\newcommand{\raster}{V_{R,\frac{r}{10}}} 

\newcommand{\Lint}{\Lambda^{\operatorname{int}}}
\newcommand{\Lout}{\Lambda^{\operatorname{out}}}

\newcommand{\Cont}{\text{ \raisebox{0.2ex}{$\circ$} $\hspace{-0.965em} \Box$}} 

\newcommand{\rrv}{\ell^2(E_v;\kz)}  

\usepackage{array}
\newcolumntype{C}{>{\rule[-1.5ex]{0pt}{4ex}}c}

\usepackage{colortbl}

\newcommand{\circled}[1]{$\ooalign{\hfil\raise -0.5pt\hbox{#1}\hfil\crcr\mathhexbox20D}$}
\newcommand{\boxd}[1]{   $\ooalign{\hfil\raise 0.5pt\hbox{\small #1}\hfil\crcr\mathhexbox403}$}

\usepackage{tikz}
\usetikzlibrary{patterns}
\usetikzlibrary{decorations.pathreplacing}
\usetikzlibrary{calc}
\newcommand{\knoten}[1]{\fill  #1 circle (0.075cm);}
\newcommand{\knotenn}[1]{\fill (#1) circle (0.075cm);\node[anchor=south] at (#1) {$#1$};}  
\newcommand{\knotenf}[2]{\fill[#2]  #1 circle (0.075cm);}

\usepackage{ifthen, color}
\newboolean{PrintNotes}
\setboolean{PrintNotes}{true}
\newcommand{\note}[1]{
 \ifthenelse{\boolean{PrintNotes}}{
  \textcolor{red}{\itshape #1}
 }{}
}
\usepackage[style=alphabetic,backend=bibtex8,maxnames=5]{biblatex}
\usepackage[]{csquotes}
\bibliography{lit.bib}

\newcommand{\RBPLS}{\hyperlink{BC:PLS}{(BC:P,L,S)}}
\newcommand{\geomalles}{\hyperlink{geom:alles}{(geom:\ensuremath{\ukl},\ensuremath{\okl},poly)}}

\newcommand{\potalles}{\hyperlink{pot:alles}{(pot:char,dens,disord)}} 

\usepackage[linkbordercolor=gray,citebordercolor=black,urlbordercolor=black]{hyperref}

\hypersetup{linkbordercolor=[rgb]{0.85,0.85,0.85}}
\hypersetup{citebordercolor=[rgb]{0.85,0.85,0.85}}
\hypersetup{urlbordercolor=[rgb]{0.85,0.85,0.85}}

\begin{document} 
\newcommand{\dx}{\,dx}

\title{Localization for quantum graphs with a random potential}
\author{Carsten Schubert}

\maketitle

\begin{abstract}
We prove spectral localization for infinite metric graphs with a self-adjoint Laplace operator and a random potential. To do so we adapt
the multiscale analysis (MSA) from the $\rz^d$-case to metric graphs. In
the MSA a covering of the graph is needed which is obtained from a
uniform polynomial growth of the graph. The geometric restrictions of
the graph include a uniform bound on the edge lengths. As boundary
conditions we allow all local settings which give a lower bounded self-adjoint
operator with an associated quadratic form.

The result is spectral localization (i.e. pure point spectrum) with
polynomially decaying eigenfunctions in a small interval at the ground
state energy.
\vspace{8pt}

\noindent
MSC 2010: 82B44, 81Q10, 81Q35, 47B80
\vspace{2pt}

\noindent
Keywords: Anderson localization, quantum graphs, random Schr\"odinger operator, multiscale analysis
\end{abstract}

\section{Introduction}


The theory of transport through media plays an important role in various contexts. In the quantum mechanical treatment via Schr\"odinger operators periodic models often exhibit diffusion and  transport through the media, whereas random disturbances may lead to insulator properties, i.\,e. absence of the transport.
See e.\,g. \cite{Anderson} for one of the first descriptions of this phenomenon, which is called (Anderson) localization and has been widely studied since then. We want to present a proof of localization in the setting of quantum graphs which works in a rather general setting. This means, that we want to impose as few as possible conditions on the geometry of the graphs and the random operators describing the system. 

Quantum graphs, which are metric graphs with a differential operator, are important models in physics and mathematics, see e.\,g. the conference proceedings \cite{BCFK,AGA}. 
We will consider self-adjoint Laplace operators on metric graphs with general boundary conditions as basis of the random Schr\"odinger operator. Boundary conditions for such operators where studied in \cite{KostrykinS-99b, Harmer-08, Kuchment-04, LSV}. As random component we will add a random alloy-type potential on the edges.

Looking at the Schr\"odinger equation the RAGE theorem tells, that bound states of a system in a certain energy region correspond to pure point spectrum of the corresponding Hamiltonian.
Using this, there are two different methods in higher dimensions to prove localization. One is the multiscale analysis introduced by Fr\"ohlich and Spencer in \cite{FS83} and further developed in \cite{FMSS} and \cite{DreifusK}. The first continuous model was treated in \cite{MH}. The other method is the fractional moment method, which was introduced by Aizenman and Molchanov in \cite{AM93}. This method proved to be more elegant in the discrete case, but loses elegancy in the continuous case, which might be found in \cite{AENSS,BNSS}. 

In case of quantum graphs there is one result of delocalization in the literature: 
For a rooted tree graph with Kirchhoff boundary conditions and random edge lengths delocalization was proven under weak disorder in \cite{AizenmanSW}.
Localization was proven for:

\begin{enumerate}
\item 
Radial quantum trees with either random edge lengths or random $\delta$-couplings in \cite{HislopP-09},
\item
a metric graph over $\gz^d$ with Kirchhoff boundary conditions and a random potential in \cite{ExnerHS-07},
\item
$\gz^d$ with random $\delta$-couplings in the vertices in \cite{KloppP-08} and with random edge lengths in \cite{KloppP-09},
\item
a multi particle model over $\gz^d$ with Kirchhoff boundary conditions and a random potential in \cite{Sabri}.
\end{enumerate}
All above methods highly use the symmetry of the graph or the special boundary conditions. 
The method, which we were able to generalize with adaptions is the multiscale analysis presented in \cite{Helm-07}, which was used in \cite{ExnerHS-07} and \cite{KloppP-09}.

The method used here is based on the work \cite{Stollmann-01} and can briefly be divided as follows
\begin{enumerate}
\item[\circled{1} :] 
We have the existence of generalized eigenfunctions of the presented operator, with a known maximal growth rate.
\item[\circled{2} :] 
We can prove decay of the local resolvent depending on the size of the domain in a certain energy interval.
\item[\circled{3} :] With \circled{1} and \circled{2} we can conclude exponential or polynomial decay of the generalized eigenfunctions. Thus we get real eigenfunctions and pure point spectrum in the energy interval from \circled{2}.
\end{enumerate}

Point \circled{2} is the essential part of the multiscale analysis. 
As part of it a covering of the metric space (usually $\gz^d$ or $\rz^d$) is needed, which is commonly done by cubes of different lengths. For an arbitrary metric graph there is no embedding in such a space and thus no easy covering. 
From the conceptual point of view this is one of the main challenges in this model.
We will extract a covering with some kind of balls that will be obtained using a uniform polynomial growth from above for the metric graph.

In the end (in theorem \ref{satz:hauptsatz}) we prove spectral localization for a big class of graphs and random operators in a small energy interval at the ground state energy. Namely, we extend the group of graphs to polynomially growing graphs with uniform bounds of the edge lengths and allow all self-adjoint Laplacians with local boundary conditions, which are lower bounded. Before only $\delta$-type boundary conditions  and graphs over $\gz^d$ and special trees could be treated.

However, the multiscale analysis doesn't provide exponential localization, but only polynomial localization, i.e. polynomially decaying eigenfunctions. For an explanation see remark \ref{bem:pol:decay}.

The results of this paper are essentially included in the thesis \cite{diss_11}, where a uniform polynomial bound of the growth of the volume of the metric graph from above and from below was assumed. The additional uniform bound from below makes some results and bounds on parameters nicer and look more as in the case of $\gz^d$ or $\rz^d$. Here we also made slight improvements in the notation.

The paper is structured as follows:
In section \ref{ab:model} we give the definition and notation of metric graphs and random operators.
In the next section we define finite subgraphs and restrictions of the random operator to them. In section \ref{ab:WG} we define uniform polynomial growth of metric graphs and extract a covering of subgraphs from that.
The next sections cover all necessary estimates for the multiscale analysis, namely we prove a general Combes-Thomas estimates (section \ref{sec:CTE}), Geometric resolvent inequality (section \ref{sec:GRI}), Weyl asymptotics (\ref{sec:WEYL}) and an initial length scale estimate (sec. \ref{ab:ALAW}). Further we state the Wegner estimate from \cite{GruberHV-08} (section \ref{ab:ALAW}) and the existence of generalized eigenfunctions from \cite{LSS-08} in section \ref{ab:VEF}.

In section \ref{ab:msa} we give the adapted multiscale analysis followed by spectral localization in section \ref{ab:SL}. In the last section we discuss our results by examples.

The appendix will show that the induction parameters are well defined.

\bigskip

\textbf{Notation.} For the reader's convenience we list some notations and symbols used in this article and reference where they are defined.\\
\begin{tabular}{lll}
$\Gamma$ &  metric graph & definition \ref{def:mg}\\
$X_\Gamma$ & metric space & section \ref{subsec:mg}\\
\RBPLS & boundary conditions & def. \ref{def:bc}\\
$H^{P,L}, \form_L$ & Laplacian, form & theorem \ref{satz_H_sa}\\
$\pot_\omega$, \potalles & potential & sec. \ref{ab:pot}, assumption \ref{ann_unordnung} \\
$H^{P,L}(\omega), \form_\omega$ & random operator, form & def. \ref{def:ran_op}\\
\geomalles & geometric restrictions & ass. \ref{ass:uU}, def. \ref{def:poly}\\
$\Lambda_r(v)$ & ball with radius $r$ & def. \ref{def:ball}\\
$\Lint(v)_r$, $\Lout_r(v)$  & inner, outer part of a ball & def. \ref{def:intout}\\
$\Cont_i$ & container set & step \ref{schritt1}
\end{tabular}
\bigskip

\textbf{Acknowledgments.} Financial support by the German Research Foundation (DFG) is gratefully acknowledged. The author wants to thank Peter Stollmann, Daniel Lenz, Ivan Veseli\'c and their research groups in Chemnitz and Jena for fruitful discussions and useful hints.


\section{The model}
\label{ab:model}

\subsection{Metric graphs}
\label{subsec:mg}
\begin{defn}
\label{def:mg}
A metric graph is a tuple $\Gamma=(E,V,l,i,j)$ consisting of countable sets of edges $E$ and vertices $V$, a length function $l:E\to (0,\infty]$ giving each edge a length and functions giving each edge a starting point and each finite edge an end point $i:E\to V$, $j:\{ e\in E \with l(e)<\infty \}\to V$.
\end{defn}
With this definition we allow for loops and multiple edges. Additionally we exclude isolated vertices and treat connected graphs only. The graph is assumed to be infinite, as for finite graphs with bounded edge lengths the spectrum of the Laplacian is purely discrete, (see proposition \ref{hsatz:weyl_dir} point 2).
The interval $I_e:=(0,l(e))$ will be identified with each edge $e$. 
With these intervals we define the spaces 
\begin{equation*}
X_E:= \bigcup\limits_{e \in E}\{ e\}\times I_e, \qquad X_\Gamma:=X_E \cup V
\end{equation*}
and a mapping $\metrik:X_\Gamma \times X_\Gamma \to [0,\infty)$ by
\begin{equation*}
\metrik(x,y) = \inf \{|p(x,y)| \with p(x,y) \text{ is a path from }x \text{ to }y\},
\end{equation*}
where $|p(x,y)|$ is the length of the path from $x$ to $y$, which can be computed with the help of the Lebesgue measure on the edges.
For a metric graph with a lower bound of edge lengths the mapping $\metrik$ is a metric and $(X_\Gamma,\metrik)$ is complete.

The geometric property of a lower bound of the edge lengths will be a needed assumption:
\begin{align}
\label{geom:u}
\tag{geom:u}
\exists\  \ukl>0, \text{ s.\,t. } \fa e \in E : \qquad l(e)\geq \ukl.
\end{align}

We denote functions $f:X_E \to \kz$ by $f_e(t):=f(e,t)$.
The underlying Hilbert space is
\begin{equation*}
L^2(X_E):= \bigoplus\limits_{e\in E } L^2(I_e)=
\{f=(f_e)_{e\in E} \with f_e\in L^2(I_e), \sum\limits_{e\in E} \|f_e \|^2_{L^2(I_e)}<\infty  \}
\end{equation*}
with the corresponding Sobolev spaces
\begin{equation*}
W^{1,2} (X_E) := \bigoplus\limits_{e\in E } W^{1,2}(I_e),\qquad W^{2,2}(X_E):= \bigoplus\limits_{e\in E } W^{2,2}(I_e).
\end{equation*}
These spaces are sometimes called decoupled Sobolev spaces, as functions don't need to be continuous in the vertex---which we want to allow, to describe more general boundary conditions.

\begin{defn} 
If a vertex $v$ is a starting or end point of an edge $e$, then $v$ and $e$ are called incident. We will denote this relation by $e\sim v$.

Let $E_v:=\{(e,0) \with v=i(e) \} \cup \{(e,l(e)) \with v=j(e)  \}$ be the set of outgoing and incoming edges incident to $v$. The degree of a vertex is defined by
\begin{equation*}
d_v:= |\{ (e,0) \with v= i(e)\} \cup \{(e,l(e)) \with v=j(e) \}| =|E_v|.
\end{equation*}
\end{defn}

From the Sobolev imbedding theorem (e.g. theorem 4.12 in \cite{AdamsF}) we know that each function in $W^{j+1,2}(0,l)$ has a representative in $C^j(0,l)$ and can be continuously extended to the boundary. Thus we can define the limits
\begin{align*}
f (0) &:=\lim_{t\to 0} f(t) &      f (l) &:=\lim_{t\to l} f (t) &  &\text{for } f\in W^{1,2}(0,l) \text{ and}\\
f' (0) &:=\lim_{t\to 0} f' (t)  &  f' (l)&:=\lim_{t\to l} f' (t) &  &\text{for }f \in W^{2,2}(0,l).
\end{align*}

\begin{defn}
By $\tr{}(f)$ we define the trace of a function $f\in W^{1,2}(X_E)$ to be the vector of all boundary values of $f$ and $\tr{v}(f)$ its restriction to all beginnings/ends of edges incident to $v$:
\begin{equation*}
\tr{}(f)=\left(\left( (f_e(t)\right)_{(e,t)\in E_v}\right)_{v\in V}, \qquad \tr{v}(f):=(f_e(t))_{(e,t)\in E_v}.
\end{equation*}
Analogue we define the signed trace
\begin{equation*}
\str{}(f)=\left(\left( (\sgn(e,t)\,f_e(t)\right)_{(e,t)\in E_v}\right)_{v\in V}, \qquad \str{v}(f):=\left(\sgn(e,t)\,f_e(t)\right)_{(e,t)\in E_v},
\end{equation*}
where $\sgn(e,t)=1$ for $t=0$ and $\sgn(e,t)=-1$ for $t=l(e)$.
\end{defn}
If we look at $\str{}(f')$ the minus sign at the derivatives of the end points give the so called ingoing derivatives (where ingoing refers to the edges). Hence the direction of the edge is neglected.
\subsection{Boundary conditions}
With the boundary values we can set up boundary conditions to find self-adjoint extensions of the symmetric Laplace operator. This was first done by Kostrykin and Schrader in \cite{KostrykinS-99b} for a star graph and then by Kuchment in \cite{Kuchment-04}. We will use boundary conditions based on the ones by Kuchment but use the notation and results from \cite{LSV}.

\begin{defn}
A metric graph $\Gamma$ with a self-adjoint differential operator is called Quantum graph.
\end{defn}

\begin{bem}
Let $\Gamma$ be a metric graph with a lower bound of the edge lengths. 
Then for a function $f\in W^{1,2}(X_E)$ 
\begin{enumerate}
\item $\tr{v}(f)\in \rrv$, $\tr{}(f)\in \bigoplus_{v\in V}\rrv$ and
\item $\str{v}(f)\in \rrv$, $\str{}(f)\in \bigoplus_{v\in V}\rrv$.
\end{enumerate}
holds.
\end{bem}
The proofs of the statements can be found in remark 1.9 and 1.10 in \cite{LSV}. 

\begin{defn}
\label{def:bc}
Let $\Gamma$ be a metric graph. A boundary condition of the form 
\hypertarget{BC:PLS}{(BC:P,L,S)} consists of a pair $(P,L)$ of families: Here $P=(P_v)_{v\in V}$ is a family of orthogonal projections $P_v:\rrv \longrightarrow \rrv$ on closed subspaces of $\rrv$ and $L= ((L_v,\dom(L_v)))_{v\in V}$ a family of self-adjoint operators
\begin{align*}
L_v :\dom(L_v)  \longrightarrow (1-P_v)\left(\rrv\right) \qquad \text{with }\dom(L_v) \subset (1-P_v)\left(\rrv\right),
\end{align*}
where $L_v$ should be uniformly bounded from below by $-S$. This means:
\begin{equation*}
\esgibt S>0 \text{ with } \langle L^-_v x,x\rangle \geq -S\langle x,x\rangle \text{ for all } x \in \dom(L_v) \text{ and }v\in V.
\end{equation*}
Here $L_v^-$ is the negative part of $L_v$ ($L_v^-:=L_v P_{(-\infty,0)}(L_v)$ with the spectral projection on the interval $(-\infty,0)$).

Then the operator $L:=\bigoplus_{v\in V} L_v$, the direct sum of the operators $L_v$, is self-adjoint and bounded from below by $-S$. We denote the associated form to $L$ by $s_L$:
\begin{align*}
\dom(s_L)&=\dom(L^\frac{1}{2})\subset \bigoplus (1-P_v)(\rrv),\\
s_L[x,y]&=\langle Lx,y \rangle = \sum\limits_{v\in V} \langle L_v x_v,y_v \rangle \qquad \text{for all } x \in \dom(L), y\in \dom(s_L).
\end{align*}
\end{defn}

\begin{satz}
\label{satz_H_sa}
Let a metric graph $\Gamma$ with a uniform lower bound of the edge lengths and  a boundary condition of the form \RBPLS\ be given. Then the operator $H^{P,L}$ is self-adjoint, lower bounded and associated to the form $\form_L$, where:
\begin{align*}
\dom(H^{P,L})&=\{ f\in W^{2,2}(X_E) \with \tr{v}(f) \in \dom(L_v), \\
 & \hspace{4.5cm} L_v \tr{v}(f)=(1-P_v) \str{v}(f') \fa v\in V \},\\
H^{P,L}f&=-f'',\\
\dom(\form_L)&=\{ f\in W^{1,2}(X_E)\with \tr{}(f)\in \dom(s_L)\},\\
\form_L[f]&=\|f'\|^2_{L^2(X_E)}+s_L[\tr{}(f)].
\end{align*}
\end{satz}

The statement directly follows from theorems 2.2 and 4.8 in \cite{LSV}. Note that these are in fact all self-adjoint, lower bounded Laplacians with vertex boundary conditions (see theorem A.6 in \cite{LSV}).

As the first part of the two restrictions in the boundary condition of the operator gives automatically $(1-P_v)\tr{v}(f) = \tr{v}(f) $, we want to read the restrictions successively and omit the projection in the second part. Thus we will consequently write $L_v \tr{v}(f)$ instead of $L_v(1-P_v)\tr{v}(f)$ which is commonly used in the literature.

\subsection{The random operator}
\label{ab:pot}
In our model the randomness will enter by an alloy-type potential.
We define coupling constants $\omega_e$ and single site measures $\nu_e$ for each edge with the following properties:

\begin{enumerate}
\item 
Let $\mu$ be a probability measure on $\rz$ with the Borel $\sigma$-algebra and
bounded support $\supp \mu =[q_-,q_+]$, where $-\infty<q_-<q_+<\infty $. 
Moreover $\mu$  possesses  a bounded density $\dichte_\mu\in L^\infty[q_-,q_+]$ with $ \| \dichte_\mu \|_{L^\infty [q_-,q_+]}=: c_\dichte$.
\begin{flushright} \hypertarget{pot:dichte}{(pot:dens)} \end{flushright}
 \item 
There are real, positive constants $c_-$ and $c_+$ with $c_-\leq c_+$, such that on each edge the single site potentials $\nu_e:I_e \to \rz$ with $\nu_e \in L^\infty(I_e)$ satisfy
\begin{equation}
\tag{pot:char}\label{pot:char}
c_-\ind_{I_e}\leq \nu_e \leq c_+\ind_{I_e}. 
\end{equation}
\end{enumerate}
Let $\Omega:=[q_-,q_+]^E $ and $\PP:=\bigotimes\limits_{e\in E} \mu $ a probability measure on $\Omega$. 
The function ${q_e:\Omega\to \rz}$ yields the coupling constants for the corresponding edge  $q_e(\omega):=\omega_e $.
For each configuration $\omega\in \Omega$ we define a random potential by
\begin{equation}
\label{pot:char,dichte}\tag{pot:char,dens}
 \pot_\omega :=\left( q_e(\omega) \nu_e\right)_{e\in E},
\end{equation}
which acts on $L^2(X_E)$ as multiplication operator.
For all functions $f\in L^2(X_E)$ we have
\begin{align*}
\| \pot_\omega f \|^2_{L^2(X_E)} &=\| \left( q_e(\omega) \nu_e f_e \right)_{e\in E}  \|_{L^2(X_E)}^2 \\
&\leq \sum\limits_{e \in E} \left(\max\{|q_-|,|q_+| \}\right)^2 c_+^2 \int\limits_0^{l(e)} |\ind_{I_e}(x) f_e(x) |^2 \dx\\
&\leq \left(\max\{|q_-|,|q_+| \}\right)^2 \cdot c_+^2 \|f\|^2_{L^2(X_E)}.
\end{align*}
By defining the constant $C_\pot:=\max\{|q_-|,|q_+| \}\cdot c_+ $ we get the estimate
\begin{equation}
\label{gl_pot_stetig} \|\pot_\omega f\|_{L^2(X_E)} \leq C_\pot \|f\|_{L^2(X_E)}.
\end{equation}
Thus the potential is a continuous operator with a uniform bound. This property will be sufficient to prove some of the needed estimates.

We could consider more general potentials, but for simplicity stay with the above definitions, which yield:
The restriction of the random potential on two different edges $\omega_{e_1}\,\nu_{e_1}$, $\omega_{e_2}\,\nu_{e_2}$ are independent.

Now we use the random potential to define a random operator family on the metric graph.
\begin{defn}
\label{def:ran_op}
Let a metric graph $\Gamma$ with \eqref{geom:u}, a random potential $\pot_\omega$  with \eqref{pot:char,dichte} and a boundary condition of the form \RBPLS\ be given.
For all $\omega\in \Omega$ we define the random operator $H^{P,L}(\omega)$ and a  sesquilineaer form $\form_\omega$ by:
\begin{align*}
 H^{P,L}(\omega) &=H^{P,L}+\pot_\omega  &\hspace{-1.2pt}\dom(H^{P,L}(\omega))=\dom(H^{P,L}),\\
\form_\omega[f,g]&=\langle f',g' \rangle +\sum\limits_{v\in V} \langle L_v \tr{v}(f),\str{v}(g)\rangle +\langle \pot_\omega f,g \rangle
&\dom(\form_\omega)=\dom(\form_L).
\end{align*}
\end{defn}
Then $H^{P,L}(\omega)$ is self-adjoint and bounded from below, $\form_\omega$ is a closed sesquilinear form which is bounded from below and by theorem \ref{satz_H_sa} both are associated to each other.

\begin{bem}
\label{bem_meas_op_fam}
We want to comment on the spectrum of the random operator family. Using a theorem of Kirsch and Martinelli \cite{KirschM} one usually concludes at this point deterministic spectrum of the random operator family from measurability and ergodicty.

As the underlying model, the metric graph, doesn't obey any symmetric relations (group structure or translation invariance), the family $(H^{P,L}(\omega))$ of random operators is not ergodic, as in the most cases presented on $\gz^d$ or $\rz^d$. 
With the form criterion of measurability (see proposition 1.2.6 in \cite{Stollmann-01}) we can easily conclude, that it is measurable (see proposition 2.2.3 in \cite{diss_11}), but in general not ergodic. Therefore we don't find a deterministic spectrum and localization statements are not as strong as in the deterministic case. 

The main result (theorem \ref{satz:hauptsatz}) is, that in a certain interval the spectrum of the operator is almost surely pure point, if there is spectrum at all.
\end{bem}
For examples, where the operator family is still ergodic, see theorem \ref{satz_CayleyG} with Cayley Graphs.
For another paper of non-ergodic models see \cite{RM}. 

\section{Induced subgraphs}
\label{ab:ITG}
One main tool in spectral theory on infinite models is the restriction to finite subsets, where the restriction of the operator possesses discrete spectrum. Under some conditions 
we can conclude properties of the unrestricted operator from properties of these restrictions.
Usually cubes (in $\gz^d$ or $\rz^d$) are chosen as domains of the restrictions.
With them a covering or even tiling of the whole space or bigger cubes is easily constructed, also using cubes of different length scales.
In the case of metric graphs there is no analog definition of cubes with similar properties.

Here we can use two different methods, which both rely on Vitali's covering lemma.
One way is using dyadic cubes, which can be found in theorem 11 in \cite{Christ-90}. There the space is tiled with some sets, which have a minimal and a maximal radius, which are relatively far apart. 

Another way is to cover the space with balls, which are not disjoint, but have almost the same radius. One advantage will be the existence of balls for all radii, in contrary to the lattice case, where only cubes with certain length scales can be used.

In this section we will give definitions of subgraphs, restrictions of operators to those and sets, which are used for the covering. The covering itself and the needed geometric properties of the metric graphs will be given in the next section.
\begin{defn}
Let $\Gamma=(E,V,l,i,j)$ be a metric graph and $E_1 \subset E$ be a subset of the edges.
Let $V_{E_1}$ be the set of all initial and end points of the edges from $E_1$:
\begin{equation*}
V_{E_1} := \{v \in V \text{ with } \esgibt e\in E_1 \text{ with } v=i(e) \text{ or }v=j(e) \}.
\end{equation*}
Let $l_{E_1} = l|_{E_1} $, $i_{E_1}=i|_{E_1}$ and $j_{E_1}=j|_{E_1}$ be the restrictions of the functions to $E_1$. 
Then $(E_1,V_{E_1},l_{E_1},i_{E_1},j_{E_1})$ is a metric graph, will be denoted by $\Gamma_{E_1}$ and referred to as the subgraph of $\Gamma$ induced by $E_1$.
\end{defn}
Analogously we define the restrictions of $X_E$ and $X_\Gamma$  to an edge set $E_1$:
\begin{equation*}
X_{E_1}=\bigcup_{e \in E_1} \{e\} \times I_e \qquad X_{\Gamma_{E_1}}=X_{E_1}\cup V_{E_1}.
\end{equation*}

The set of inner vertices of an induced subgraph $\Gamma_{E_1} \subset \Gamma $ will be denoted by $\ik{E_1}$, the set of boundary vertices by $\rk{E_1} $, where
\begin{align*}
v \in \ik{E_1} &\Leftrightarrow \text{for } v \text{ holds }\{e \in E \text{ with } e\sim v  \}\subset E_1,\\
v \in \rk{E_1} &\Leftrightarrow \esgibt e_1 \in E_1 \text{ and }e_2 \in E\setminus E_1 \text{ with } e_1 \sim v, e_2 \sim v.
\end{align*}
We want to comment that this is not the usual definition of inner or boundary points, for sets imbedded in $\rz^d$ or $\gz^d$, but rather concerning the imbedding in the whole graph $\Gamma$.
\begin{defn}
Two induced subgraphs $\Gamma_{E_1}$, $\Gamma_{E_2}$ are called disjoint, if $E_1$ and $E_2$ are disjoint.
Inclusions of subgraphs are also related to the inclusions of the edge sets, i.\,e.
\begin{equation*}
\Gamma_{E_1}\subset \Gamma_{E_2} \Leftrightarrow E_1\subset E_2.
\end{equation*}

\end{defn}
\begin{bem} Obviously it holds:
\begin{itemize}
\item 
The set of vertices $V_{E_1}$ is the disjoint union of $\rk{E_1}$ and $\ik{E_1}$.
\item 
Two disjoint subgraphs have no common edge, but might have a common vertex, which then has to be a boundary vertex of both subgraphs.
\item 
Two restrictions of the negative Laplace operator to two disjoint subgraphs (which will be defined explicitly later) are independent, as the coupling constants are defined on the edges.
\end{itemize}
\end{bem}
In the following we will define induced subgraphs which correspond to neighborhoods of a root vertex with radius $r$ and will be used extensively.
To make this definition meaningful we will make the assumption, that the edge lengths are uniformly bounded from above. As a uniform bound from below is important for the selfadjointness of $H^{P,L}$ we will use those two assumptions:
\begin{ann}
\label{ass:uU}
Let $\Gamma$ be a metric graph, such that the edge lengths are uniformly bounded from below and from above, i.\,e. there are $\ukl$, $\okl$ with $0<\ukl\leq \okl<\infty$, s.\,t.
\begin{equation*}
\label{geom:uU}\tag{geom:\ensuremath{\ukl},\ensuremath{\okl}}
\fa e\in E \text{ we have} \qquad \ukl \leq l(e) \leq \okl.
\end{equation*}
\end{ann}

\begin{defn}
\label{def:ball}
For a metric graph $\Gamma$ with \eqref{geom:uU} 
we denote by $\Lambda_r(v_0)\subset \Gamma$ the subgraph
$\Gamma_{E(v_0,r)}$ induced by  $E(v_0,r)$, where
\begin{align*}
E(v_0,r)=\{ e\in E \text{ with } \esgibt t\in I_e \text{, such that }(e,t)\in B_r(v_0)\}.
\end{align*}
\end{defn}

\begin{figure}[htb]
\centering
\begin{tikzpicture}[scale=0.9]
\xdefinecolor{tugreen}{RGB}{0, 90, 70};
\path (0:0) coordinate (v_0);
\path (0,2) coordinate (v1);\path (2.5,2.3) coordinate (v2);\path (33:4) coordinate (v3);  
\path (33:2) coordinate (v4);\path (42:5.7) coordinate (a);   
\path (40:4.7) coordinate (r1);\path (5:4) coordinate (r2);\path (-10:4) coordinate (r3); 
\draw[thick,tugreen,dashed] (-20:4) arc(-20:70:4);   
\node[tugreen] at (4.4,2.2) {$B_r(v_0)$};    
\draw[thick,blue] (v_0) to [out=190,in=25+180] (v1);\draw[thick,blue] (v1) to [out=25,in=180] (v2);\draw[thick,blue] (v2) to [out=90,in=175] (r1);
\draw[thick,blue] (v2) to [out=0,in=150] (v3);\draw[thick,blue] (v2) to [out=270,in=45] (v4);\draw[blue,thick] (v4) to [out=-35,in=180] (r2);
\draw[blue,thick] (v_0) to [out=10,in=180+45] (v4);\draw[blue,thick] (v_0) to [out=-80,in=180] (r3);
\draw[thick,blue,densely dotted] (v1) to [out=25+90,in=-25] (-0.7,2.9);\draw[thick,gray,densely dotted] (r2) to [out=89,in=180] ($(r2)+(1.3,1.2)$);
\draw[gray,thick] (r2) to [out=20,in=-20] (r3);\draw[thick,gray] (r1) to [out=5,in=265] (a);
\knotenn{v_0}; 
\knoten{(v3)};\knoten{(v1)};\knoten{(v2)};\knoten{(v4)}; 
\knotenf{(r1)}{red};\knotenf{(r2)}{red};\knotenf{(r3)}{red}; 
\knotenf{(a)}{gray}; 
\knotenf{(5.2,0.5)}{red};\node[right,red] at (5.2,0.5) {$\rk{E(v_0,r)}$};
\knoten{(5.2,-0.3)};\node[right] at (5.2,0.2-0.5) {$\ik{E(v_0,r)}$};
\node[blue] at (1,3) {$E(v_0,r)$};
\end{tikzpicture}
\caption{Induced subgraph -- inner and boundary vertices.}
\label{abb:ITG}
\end{figure}
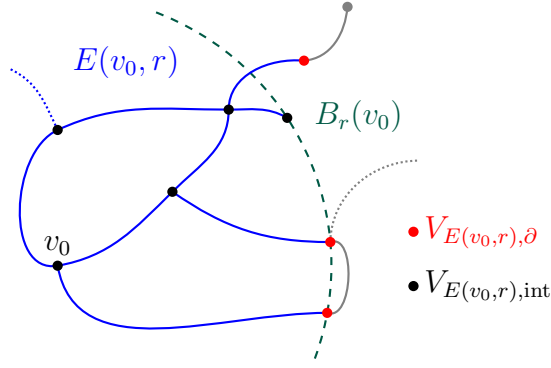

The assumption on the uniform bound of the edge lengths guarantees that all points of $\Lambda_r(v_0)$ have a distance less than $r+\okl$ to the center $v_0$.
In the following this neighborhood will be called ball with radius $r$ and center $v_0$

A central point of the multiscale analysis is estimating what the resolvent of $H^{P,L}$ transports from the interior of some neighborhoods to their boundary.
What we will understand by those sets will be defined now.
\begin{defn}
\label{def:intout}
The interior of a ball $\Lambda_r(v)$ is the ball with radius $\tfrac{r}{3}$:
\begin{equation*}
\Lint_r(v):=\Lambda_\frac{r}{3}(v),
\end{equation*}
the exterior or boundary of a ball is defined as the difference 
\begin{equation*}
\Lout_r(v):=\Lambda_r(v)\setminus \Lambda_{r-3\okl}(v).
\end{equation*}
\end{defn}
The subgraphs $\Lint_r(v)$ and $\Lout_r(v)$ are again induced subgraphs, they are induced by $E\left(v,\tfrac{r}{3}\right)$ and $E(v,r)\setminus E(v,r-3\okl)$.
\begin{defn}
The distance of two subsets $A$ and $B$ of $X_E$ is defined as
\begin{equation*}
\dist(A,B)=\inf_{a\in A}(\inf_{b\in B}\metrik(a,b)).
\end{equation*}
Thereby we can define the distance of two induced subgraphs $\Gamma_{E_1}$ and $\Gamma_{E_2}$ as
\begin{equation*}
\dist(\Gamma_{E_1},\Gamma_{E_2}):=\dist(X_{E_1},X_{E_2}).
\end{equation*}
\end{defn}
The distance of the interior to the exterior of a ball with radius $r$ can be estimated by
\begin{equation*}
\dist(\Lint_r(v),\Lout_r(v)) > (r-3\okl)-\left(\tfrac{r}{3}+\okl\right)=\tfrac{2}{3}r-4\okl,
\end{equation*}
which gives 
\begin{equation}
\label{gl_dist_Lint_Lout}
\dist(\Lint_r(v),\Lout_r(v))> \frac{r}{2}.
\end{equation}
for a big radius $r\geq 24\okl$.

Now we are able to restrict operators  and forms to induced subgraphs.

For a random self-adjoint operator $H^{P,L}(\omega)$ with 
\RBPLS\ and \eqref{pot:char,dichte}
we define the restriction to an induced subgraph $\Gamma_{E_1}=(E_1,V_{E_1},l_{E_1},i_{E_1},j_{E_1})$ as
\begin{align*}
H^{\Gamma_{E_1}}(\omega)f &:=H^{\Gamma_{E_1},P,L}(\omega) f  = (-f''_e)_{e\in E_1} +(q_e(\omega) \nu_e f_e)_{e\in E_1}, \\
\dom(H^{\Gamma_{E_1}}(\omega))&=\{f \in \bigoplus\limits_{e\in E_1} W^{2,2}(I_e) \text{ with } f\text{ satisfies } (P_v,L_v) \text{ b.\,c. of } H^{P,L} \text{ on } \ik{E_1},\phantom\}\\
&\phantom{=\{f \in \bigoplus\limits_{e\in E_1} W^{2,2}(I_e) \text{ with }} \tr{v}(f)\equiv 0 \text{ on } \rk{E_1} \}.
\end{align*}
Which means, that the restriction satisfies the primary boundary condition on the inner vertices of the subgraph and on the boundary vertices Dirichlet boundary conditions.
In particular $H^{\Gamma_{E_1}}(\omega)$ is again self-adjoint and lower bounded.
The associated quadratic form $\form^{\Gamma_{E_1}}_\omega$ is given by
\begin{align*}
 \form^{\Gamma_{E_1}}_\omega[f,f]&=\langle f',f' \rangle_{L^2(X_{E_1})} +\sum\limits_{v \in \ik{E_1}} \langle L_v \tr{v}(f),\tr{v}(f) \rangle+\langle \pot_\omega f,g\rangle_{L^2(X_{E_1})},\\
\dom\left(\form^{\Gamma_{E_1}}_\omega\right)&=\{f\in \bigoplus\limits_{e\in E_1} W^{1,2}(I_e) \text{ with } \tr{v}(f) \in \dom(L_v) \text{ on } \ik{E_1},
 \tr{v}(f)\equiv 0 \text{ on } \rk{E_1} \}.
\end{align*}

\begin{bem}
Let $\Gamma_{E_1}\subset \Gamma_{E_2}$. Then $\dom(\form^{\Gamma_{E_1}}) \subset \dom(\form^{\Gamma_{E_2}})$, as $\tr{v}(f) \equiv 0$ always lies in $\dom(L_v)$.
\end{bem}

For restricting functions to the corresponding function spaces $L^2(X_{E_1})$ we will mainly use characteristic functions $\ind_{X_{E_1}}$.
As we will mostly use induced subgraphs which are balls with radius $r$ at some vertex, we will also denote the characteristic functions in this way: 
\begin{equation*}
\ind_{\Lambda_r(v)}:= \ind_{X_{E(v,r)}}, \hspace{0.4cm} \ind_{\Lint_r(v)}:= \ind_{X_{E(v,\frac{r}{3})}},
\hspace{0.4cm}\ind_{\Lout_r(v)}:= \ind_{X_{E(v,r)\setminus E(v,r-3\okl)}}.
\end{equation*}

In the theory of square integrable functions and Sobolev-spaces smooth functions with compact support are of great importance, e.\,g. as test functions.
For metric graphs it is not clear, how compactness should be defined.
Using $\delta$-type boundary conditions, which yield continuity in the vertices, might lead to stronger definitions.
As we will treat also more general boundary conditions, we also need a more general definition of compactness:

For $\tilde{E}\subset E$ let $\mathcal K_{\tilde{E}}$ be the set of all sets, which will be regarded as substitute of compact sets in $X_{\tilde{E}}$:
\begin{equation*}
\begin{split}
\mathcal K_{\tilde{E}}= \Bigl\{ K\subset X_{\tilde{E}} \with \fa v\in \rk{\tilde{E}}\, \esgibt r>0 : B_r(v)\cap K=\varnothing,\\
K\cup\ik{\tilde{E}} \text{ is compact in }\left(X_{\Gamma_{\tilde{E}}},\metrik\right) \Bigr\}.
\end{split}
\end{equation*}
This definition guarantees, that inner vertices of induced subgraphs are not treated as \enquote{boundary}, but only the boundary vertices of the subgraph.
The set of all compact sets is defined by
\begin{equation*}
\mathcal K := \bigcup\limits_{\tilde{E}\subset E \text{ finite}} \mathcal K_{\tilde{E}}.
\end{equation*}

\begin{defn}
Let $\Gamma$ be a metric graph with \eqref{geom:u}.
We define for all $j\in \nz$ the sets of functions with compact support in $X_E$ by:
\begin{align*}
C\komp^\infty(X_E)&=\left\{ f=(f_e)_{e\in E} \with f_e \in C^\infty(I_e), \supp(f) \in \mathcal K  \right\},\\
L^2\komp(X_E)&=\left\{ f=(f_e)_{e\in E} \with f_e \in L^2(I_e), \supp(f) \in \mathcal K  \right\},\\
W^{j,2}\komp(X_E)&=\left\{ f=(f_e)_{e\in E} \with f_e \in W^{j,2}(I_e), \supp(f) \in \mathcal K  \right\}.
\end{align*}
\end{defn}

With the definition of compact sets, we can define function spaces with locally square integrable functions:
\begin{defn}
Let $\Gamma$ be a metric graph with \eqref{geom:u}
As spaces of locally square integrable functions we define for all $j\in \nz$:
\begin{align*}
L^2\lok(X_E)&=\{ (f_e)_{e\in E} \with f_e \in L^2\lok(I_e), \text{ for all } K\in \mathcal K :
 \ind_K f \in L^2(X_E)   \},\\
W^{j,2}\lok(X_E)&=\{ (f_e)_{e\in E} \with f_e \in W^{j,2}\lok(I_e), \text{ for all } K\in \mathcal K :\\
&\hspace{4.5cm} \ind_K f^{(n)} \in L^2(X_E), n=0,\ldots,j   \}.
\end{align*}
\end{defn}

The product of a compactly supported function $f\in L^2\komp(X_E)$ with a locally square integrable function $g\in L^2\lok(X_E)$ is integrable, i.\,e. in $L^1(X_E)$.
With the integral over this product we can define a mapping similar to the scalar product via
$\langle \cdot | \cdot \rangle:L^2\komp(X_E)\times L^2\lok(X_E)\to \kz$
yielding the value of the integral, the complex number
\begin{equation*}
\langle f | g \rangle := \sum\limits_{e \in E}\int\limits_{I_e} f(x)\overline{g(x)} \dx.
\end{equation*}
Obviously $\langle \cdot |\cdot \rangle$ is sesquilinear. The complex conjugate mapping, mapping $L^2\lok\times L^2\komp\to \kz $ will again be denoted by $\langle \cdot |\cdot\rangle$.
Let $f\in L^2\lok(X_E)$, $g\in L^2\komp(X_E)$. Then
\begin{equation*}
\sum\limits_{e\in E}\int\limits_{I_e} f(x)\overline{g(x)}\dx=\overline{\sum\limits_{e\in E}\int\limits_{I_e} g(x)\overline{f(x)}\dx}=\overline{\langle g|f \rangle}:=\langle f|g\rangle.
\end{equation*}
holds.
Thereby we can define mappings similar to the sesquilinear forms $(\form_\omega,\dom(\form_\omega))$ with boundary conditions \RBPLS\ and potential \eqref{pot:char,dichte}:
$\forms :\dom\komp(\form_\omega) \times \dom\lok(\form_\omega) \to \kz $ is defined by
\begin{align*}
\dom\komp(\form_\omega)&:=\{f=(f_e)_{e\in E}\with f \in W^{1,2}\komp(X_E), \fa v\in V :  \tr{v}(f)\in \dom(L_v) \},\\
\dom\lok(\form_\omega)&:=\{f=(f_e)_{e\in E}\ | \fa e\in E: f_e \in W^{1,2}(I_e), 
\fa v\in V : \tr{v}(f)\in\dom(L_v) \},\\
\forms[f,g]&:=\langle f'|g' \rangle +\sum\limits_{v\in V} \langle L_v\tr{v}(f),\tr{v}(g)\rangle +\langle\pot_\omega f|g\rangle.
\end{align*}
For $g\in\dom\lok(\form_\omega)$ and $f\in \dom\komp(\form_\omega)$ we obviously have:
\begin{equation*}
\langle g'|f' \rangle +\sum\limits_{v\in V} \langle L_v \tr{v}(g),\tr{v}(f)\rangle +\langle\pot_\omega g|f\rangle\\
=\overline{\forms[f,g]}=:\forms[g,f].
\end{equation*}

The concept of local functions will be mainly used for generalized eigenfunctions, which will be introduced in section \ref{ab:VEF}. 
With the help of the generalized scalar products and forms spectral localization will be proved in section \ref{ab:SL}.

\section{Growth, geometry and covering of metric graphs}
\label{ab:WG}
One essential tool of the multiscale analysis is the covering of subgraphs, such as cubes/boxes and an estimate on the number of needed sets. This number should be uniform and polynomially increasing with the radius or length scale of the sets.

Uniform coverings can be deduced from uniform polynomial growth of the graph---which will be presented for metric graphs in this section.

\begin{defn}
The volume of a metric graph $\Gamma=(E,V,l,i,j)$ is the sum of the lengths of all edges in $E$:
\begin{equation*}
\vol({\Gamma_{E}}) :=\sum\limits_{e\in E} l(e).
\end{equation*}
\end{defn}

\begin{defn}
\label{def:poly}
We will call a metric graph $\Gamma$ of uniform polynomial growth of degree $d$ if there is a constant $\cp \in \rz$ and a real number $d$, s.\,t.
\begin{equation}
\label{geom:poly}\tag{geom:poly}
\fa v \in V \text{ and } r\geq \ukl \qquad \vol(\Lambda_r(v)) \leq \cp  \cdot r^d.
\end{equation}
If additionally the assumption \eqref{geom:uU} is satisfied, we will denote the collection of these geometric properties by \hypertarget{geom:alles}{(geom:\ensuremath{\ukl},\ensuremath{\okl},poly)}.
\end{defn}

\begin{bem}
\begin{enumerate}
\item
The definition of uniform polynomial growth only includes a uniform bound of the growth from above. If we assumed also uniform polynomial growth from below we would get even better estimates for the covering, but have a more restrictive geometric property. For results and a multiscale analysis with uniform polynomial growth from below and from above see \cite{diss_11}.
\item
A uniform polynomial growth from below of degree one and with constant one is always given for a connected metric graph.
\item
Polynomial growth of combinatorical graphs can also be defined with the help of the counting measure. 
We want to comment that a uniform bounded vertex degree follows from uniform polynomial growth and that a metric graph with \eqref{geom:uU} is of uniform polynomial growth of degree $d$ iff the corresponding combinatoric graph is of uniform polynomial growth of degree $d$.
See Proposition 2.4.4 \cite{diss_11} for details.
\end{enumerate}
\end{bem}

\begin{bsp}
Let the metric graph, constructed by adding the graph $\nz$ in one of the points of $\gz^d$, be given (edges are defined between vertices with distance one). Then the graph will have a uniform polynomial growth of degree $d$ from above and a uniform polynomial growth of degree one from below.
\end{bsp}

Now we consider a ball with center $v_0$ and radius $R$ and want to cover it by balls of a fixed radius. To achieve this we define a set of centers inspired by Vitali's  covering lemma.
\begin{defn}
Let $\Gamma$ be  metric graph with \geomalles.
By $V_{R,r}(v_0)\subset V$ we denote a set with maximal cardinality, which satisfies the following conditions:
\begin{itemize}
 \item $\Lambda_r(v_i) \cap \Lambda_r(v_j)=\varnothing$ for $ v_i$, $ v_j \in V_{R,r}(v_0)$, $v_i\neq v_j$,
 \item $\Lambda_r(v_i) \subset \Lambda_R(v_0)$ for all $v_i\in V_{R,r}(v_0)$.
\end{itemize}
\end{defn}
The set $V_{R,r}(v_0)$ provides a maximal configuration of disjoint balls with radius $r$, which all lie inside the ball with radius $R$ and center $v_0$.
\begin{lem} \label{hsatz_ueberdeckung}
Let $\Gamma $ be a metric graph with \geomalles. Then
\begin{equation*}
\bigcup_{v\in V_{R,r}(v_0)} \Lambda_{3r+5\okl} (v) \supset \Lambda_R(v_0).
\end{equation*}
\end{lem}
\begin{proof}
We assume the contrary: Let $x\in X_{E(v_0,R)} $ lie in none of the balls with radius $3r+5\okl$. Then we have $\metrik(x,v)>3r+5\okl$ for all $ v\in V_{R,r}(v_0)$. On one of the shortest paths from $x$ to $v_0$ there exists a vertex $\tilde{v}\in V$ with distance $\metrik(x,\tilde{v})\in[r+2\okl,r+3\okl]$. By triangle inequality we get for all $y\in \Lambda_r(\tilde{v})$:
\begin{align*}
\metrik(y,v)&\geq \metrik(x,v)-\metrik(x,y) \\
\metrik(y,v)& \geq \metrik(x,v)-\metrik(x,\tilde{v})-\metrik(\tilde{v},y) \\
\metrik(y,v)& \geq 3r+5\okl-(r+3\okl)-(r+\okl)=r+\okl
\end{align*}
for all $v\in V_{R,r}(v_0)$. Which means, that $\Lambda_r(\tilde{v})$ is disjoint to all $\Lambda_r(v)$ with $v\in V_{R,r}(v_0)$. And it also lies inside $\Lambda_R(v_0)$, as
\begin{align*}
\metrik(\tilde{v},v_0)&= \metrik(x,v_0)-\metrik(x,\tilde{v}) \leq (R+\okl)-(r+2\okl)=R-r-\okl\\
\metrik(y,v_0)&\leq \metrik(y,\tilde{v})+\metrik(\tilde(v),v_0)\leq r+\okl+R-r-\okl=R,
\end{align*}
but $\tilde{v}$ is not a vertex from $V_{R,r}(v_0)$:
\begin{equation*}
\metrik(\tilde{v},v)\geq \metrik(x,v)-\metrik(x,\tilde{v})\geq 3r+5\okl-(r+3\okl)=2r+2\okl>0,
\end{equation*}
yielding a contradiction to the maximality of $V_{R,r}(v_0)$.
\end{proof}

Now we can extract the number of smaller balls needed to cover a bigger one, which is important for the multiscale analysis.
\begin{lem}
\label{hsatz_rasterzahl}
Let $\Gamma$ be a metric graph with \geomalles. Then we have
\begin{equation*}
\frac{R}{\cp\,(3r+5\okl)^d}\leq 
| V_{R,r}(v_0)| \leq \cp\, \frac{R^d}{r} \qquad \text{for all } r\geq \ukl, v_0\in V.
\end{equation*}
\end{lem}

\begin{proof}
Let $V_{R,r}(v_0)\subset V$ as defined above. Then by the last lemma and \eqref{geom:poly} we conclude
\begin{align*}
|V_{R,r}(v_0)| \cdot r \leq \vol\left(\bigcup_{v\in V_{R,r}(v_0)} \Lambda_r(v)  \right) \leq \vol(\Lambda_R(v_0)) \leq \cp \, R^d,\\
R \leq \vol(\Lambda_R(v_0)) 
\leq \vol\left(\bigcup_{v \in V_{R,r}(v_0)} \Lambda_{3r+5\okl}(v)  \right)
\leq |V_{R,r}(v_0)|\cdot\cp\cdot (3r+5\okl)^d,
\end{align*}
which easily yield the assertion.
\end{proof}
The last two lemmas show how to cover balls with radius $R$ by sets with radius $r$, e.\,g.  by choosing the vertex-raster $V_{R,\frac{r}{4}}$ for $r$ big enough ($r\geq 20\okl$).
In the induction process of the multiscale analysis we will cover $\Lambda_R(v_0)$ be interiors of balls with radius $r$, i.\,e. by $\Lint_r(v)$.
Therefore we will choose the raster $V_{R,\frac{r}{10}}$, which is a bit finer as necessary. This yields a few properties, which make calculations easier.

In the multiscale analysis we will construct so-called {\it container sets} (which will also be balls with certain centers and radii), which include all sets with certain properties. The exterior of the container sets can be covered by sets, which are no subsets of the container and thus don't posses the property. The next proposition will show, that this is possible.
\begin{hsatz}
\label{hsatz_containnerrand}
Let $\Gamma$ be a metric graph with \geomalles, $x\in V$ and $r>300\okl $. Let $\Lambda_s(v) \subset \Lambda_R(x)$ with $v\in V_{R,\frac{r}{10}}(x)$.
Then $\Lout_s(v)$ can be covered by sets with radius $\frac{r}{3}$ and centers in $V_{R,\frac{r}{10}}(x)$, which are not subsets of $\Lambda_s(v)$, i.\,e. there is a subset $W\subset V_{R,\frac{r}{10}}(x)$ with:
\begin{align*}
\bigcup_{w\in W} \Lambda_\frac{r}{3}(w)\supset \Lout_s(v) \qquad \text{and} \qquad \fa w\in W: \Lambda_\frac{r}{3}(w)\not\subset \Lambda_s(v).
\end{align*}
\end{hsatz}
\begin{proof}
Let $W$ be the subset of $V_{R,\frac{r}{10}}(x)$, such that $\Lambda_\frac{r}{3}(w)\not\subset \Lambda_s(v)$ for all $w\in W$. 

Let $x_0 \in \Lout_s(v)$. By lemma \ref{hsatz_ueberdeckung} there exists a raster point $v_0 \in V_{R,\frac{r}{10}}(x)$, such that $x_0$ lies in the ball with radius  $\left( 3\tfrac{r}{10}+5\okl\right) $ and center $v_0$. Thus we have
\begin{align}
\notag \metrik(x_0,v_0)&< \frac{3r}{10}+6\okl \qquad and \\
\label{gl_cont_rand}
\metrik(v,v_0)&\geq \metrik(v,x_0)-\metrik(x_0,v_0)\geq s-3\okl-\left(\frac{3r}{10}+6\okl\right)=s-9\okl-\frac{3r}{10}.
\end{align}
The following relations show, that $x_0\in \Lambda_\frac{r}{3}(v_0)$ and $v_0 \in W$:
\begin{itemize}
\item $x_0\in \Lambda_\frac{r}{3}(v_0) $, if $\metrik(x_0,v_0)\leq \frac{r}{3}$. Which is true for $\frac{3r}{10}+6\okl\leq \frac{r}{3}$, i.\,e. $r\geq 180\okl$,
 \item $\metrik(v,v_0)+\frac{r}{3}>s+\okl$ holds with \eqref{gl_cont_rand} if $s-9\okl-\frac{3r}{10}+\frac{r}{3}-s-\okl>0 $, i.\,e. $r>300\okl$.\qedhere
\end{itemize}
\end{proof}

The following remark is used to simplify calculations:
\begin{bem}
\label{bem_ueberdeckung}
Let $\Gamma$ be a metric graph with \geomalles, $\Lambda_R(x)$ and $\Lambda_s(v)$ two balls with $\Lambda_s(v)\subset \Lambda_R(x)$ and $v\in V_{R,\frac{r}{10}}(x)$ with 
$r\geq 180\okl$.
Then $\Lambda_s(v)$ can be covered with sets $\Lambda_{\frac{r}{3}}(v_i)$, such that 
$v_i\in V_{R,\frac{r}{10}}(x)$ and $\metrik(v,v_i)\leq s+\frac{r}{3}$.
\end{bem}
\begin{proof}
Lemma \ref{hsatz_ueberdeckung} and the fine raster show that the statement is true for $r\geq 180 \okl$.
\end{proof}

Some statements in the last two sections are only true for a minimal radius. We will gather all assumptions in one for which all lemmas, propositions and remarks are true:
\begin{equation*}
r> \rgeom=300\okl.
\end{equation*}

\section{Combes-Thomas estimate}
\label{sec:CTE}

The Combes-Thomas estimate is an important estimate and will often be used in the multiscale analysis. It provides an initial exponential decay of the local resolvent, if the energy parameter lies inside a known spectral gap.

The rate of the exponential decay depends on the size of the spectral gap, the distance of the energy parameter to the boundary of the gap and the distance of the two subgraphs between which the resolvent acts.
The constants will not depend on a special choice of the subgraphs or the concrete realization of the random operator, i.\,e. they don't depend on $\omega$.

The estimate goes back to the article \cite{CombesT} and was further improved in \cite{BCH}. The proof presented here uses associated quadratic forms and is based on the proof of theorem 2.4.1 in \cite{Stollmann-01}.
Differences arise from the inner boundary conditions \RBPLS\ in the sesquilinear form and the more restricted domain.

The estimate is proved using sectorial forms and sectorial operators, extending the theory of lower bounded forms and operators. For general theory see \cite{kato}.

The proof of the Combes-Thomas estimate uses perturbation of the original form by a weight function, where the weight function depends on the two subsets between which the resolvent acts.
The weight function will be defined in the following, where we suppress the dependence on the subgraphs, as they will be fixed in the proof.

Let $\Gamma=(E,V,l,i,j)$ be a metric graph with \eqref{geom:u}.
For given sets $Y_1$, $Y_2 \subset X_E$ we define two weight functions $\tilde{w},$ $w:X_E\to [0,\infty)$ by
\begin{align*}
\tilde{w}(x)&=\inf \{\metrik(y,x) \with y\in Y_1 \},\\
w(x)&=\min\left\{ \tilde{w}(x), \inf\{ \tilde{w}(z)\with z\in Y_2 \} \right\}.
\end{align*}
Thus $\tilde{w}(x)$ yields the minimal distance from $x$ to the set $Y_1$ in $X_E$ and $w(x)$ yields the minimal distance from $x$ to $Y_1$, bounded from above by $\dist(Y_1,Y_2)$.

\begin{bem}
The function $w$ is continuous on $X_E$, moreover $w_e \in W^{1,2}(I_e)$ with $|w'|\leq 1$.
\end{bem}

Let $\alpha \in \rz^+$. By $e^{\alpha w}$ we define the multiplication operator, acting as multiplication with $e^{\alpha w(x)} $.
The inverse operator is $e^{-\alpha w}$ and obviously $e^{\alpha w}\phi$, $e^{-\alpha w}\phi \in L^2(X_E) $ for all $\phi \in L^2(X_E) $, as $e^{\alpha w(x)}$, $e^{-\alpha w(x)}\in L^{\infty}(X_E) $.

\begin{bem}
Let $\Gamma$ be a metric graph with \eqref{geom:u}, $H^{P,L}$ an operator with boundary condition of the form \RBPLS.
If $f\in \dom(\form_L)$ we have $ (e^{\alpha w} f) \in \dom(\form_L)$, as $w$ and thus also $e^{\alpha w(x)}$ are bounded, continuous functions on $X_E$. Hence the function is continuous in the vertices.
The boundary values satisfy:
\begin{equation*}
\tr{v}(e^{\alpha w} f)=c(v)\cdot \tr{v}(f),
\end{equation*}
for a real constant $c(v)$, depending only on the vertex $v$. As the operators $P_v$ and $L_v$ are linear $(e^{\alpha w} f)$ satisfies the boundary conditions.
\end{bem}

\begin{satz}
\label{satz_CTA}
Let $\Gamma$ be a metric graph with \eqref{geom:u} and $H^{P,L}(\omega)$ a negative Laplacian of the form \RBPLS\ with a potential $\pot_\omega$ satisfying \eqref{pot:char,dichte}.
Let $\omega\in \Omega$ and $(s,t)$, with $|s|<R$ and $|t|<R$, be a spectral gap of $H^{P,L}(\omega)$.
Let $\lambda\in (s,t)$ and $\eta=\dist\left(\{\lambda\},(s,t)^c\right)$.
Let $\Gamma_{E_1}$ and $\Gamma_{E_2}$ be two induced subgraphs of $\Gamma$ with positive distance
 $\delta=\dist(\Gamma_{E_1},\Gamma_{E_2})$.

Then there are two constants $C_\mathrm{CTA}(R,S,\alpha,C_\pot)$ and $\tilde{C}(R,S,\alpha,C_\pot)$ with:
\begin{equation*}
\label{gl_CTA}
\left\|\ind_{X_{E_1}} \left(H^{P,L}(\omega)-\lambda\right)^{-1} \ind_{X_{E_2}}\right\|\leq C_\mathrm{CTA}\cdot \eta^{-1}\cdot \mathrm{e}^{-\tilde{C} \sqrt{\eta(t-s)}\delta}.
\end{equation*}
\end{satz}
\begin{proof}
We define the form $\form_\alpha $ with $\dom(\form_\alpha)=\dom(\form_\omega) $ by
\begin{equation*}
\form_\alpha[f,g]=\form_\omega\left[e^{-\alpha w}f,e^{\alpha w}g\right] \qquad \text{for all } f,g \in \dom(\form_\omega),
\end{equation*}
where the weight function $w(x)$ is constructed with the sets $Y_1=X_{E_1}$ and $Y_2=X_{E_2}$.
Using the product rule, we have
\begin{align*}
\form_\alpha[f,g]
&=\left\langle e^{-\alpha w}f',e^{\alpha w}g' \right\rangle- \alpha \left\langle e^{-\alpha w}f w',e^{\alpha w}g' \right\rangle
+\alpha \left\langle e^{-\alpha w}f', e^{\alpha w}g w'\right\rangle\\
&\phantom{=}-\alpha^2 \left\langle e^{-\alpha w}fw',e^{\alpha w}gw' \right\rangle
+\sum\limits_{v\in V} \left\langle L_v \tr{v}(e^{-\alpha w}f),\tr{v}(e^{\alpha w}g)   \right\rangle
+\left\langle \pot_\omega e^{-\alpha w}f,e^{\alpha w}g\right\rangle\\
&=\langle f',g' \rangle -\alpha^2\langle fw',gw' \rangle -\alpha \langle fw',g' \rangle +\alpha\langle f',gw' \rangle\\
&\phantom{=}\hspace{0.5cm}
+ \sum\limits_{v\in V} \langle  L_v  \tr{v}(f),\tr{v}(g) \rangle+
\langle \pot_\omega f,g\rangle.
\end{align*}
The real and imaginary (resp. symmetric and antisymmetric) parts of $\form_\alpha$ will be denoted by
\begin{align*}
\tilde{\form}[f,g] &:= \re \form_\alpha[f,g]=\left\langle f',g' \right\rangle -\alpha^2\left\langle fw',gw' \right\rangle + \sum\limits_{v\in V} \left\langle  L_v \tr{v}(f),\tr{v}(g) \right\rangle+\left\langle \pot_\omega f,g\right\rangle,\\
\formk[f,g]&:= \frac{1}{\alpha} \im \form_\alpha[f,g]= i\left\langle fw',g' \right\rangle-i\left\langle f',gw' \right\rangle,
\end{align*}
such that $\form_\alpha=\tilde{\form}+i\alpha\,\formk$.
By a Sobolev theorem and the uniform lower bound of $L_v$ (see \cite{diss_11}) we get 
\begin{equation*}
2S\ep\|f'\|^2+\sum\limits_{v \in V} \langle L_v \tr{v}(f),\tr{v}(f)\rangle\geq -\frac{4S}{\ep}\|f\|^2 \qquad \text{for all } f\in \dom(\form_L)
\end{equation*}
and all $\ep $ with $0<\ep\leq \ukl$. Choosing $\ep $ with $\ep\leq \min \left\{\ukl,\frac{1}{4S}\right\}$ we get:
\begin{align*}
\tilde{\form}[f,f]&\underset{\phantom{\eqref{gl_pot_stetig}}}{=}\|f'\|^2+\sum\limits_{v\in V} \langle L_v \tr{v}(f),\tr{v}(f)\rangle-\alpha^2 \langle fw',fw' \rangle +\langle \pot_\omega f,f\rangle\\
&\underset{\eqref{gl_pot_stetig}}{\geq} \frac{1}{2} \left\|f'\right\|^2 -\left(\frac{4S}{\ep}+\alpha^2+C_\pot\right) \|f\|^2.
\end{align*}
Thus $\tilde{\form}$ is a lower bounded quadratic form. We will denote the associated self-adjoint operator by $\tilde{H}$, for which $\dom(\tilde{H})=\dom(H^{P,L})$ holds.
For the imaginary part we have:
\begin{align}
\notag |\formk[f,f]|&=|\langle fw',f'\rangle- \langle f',fw' \rangle|\\
\notag &\leq  2 |\langle  fw',f'\rangle | \leq 2 \| fw' \| \|f'\|\leq 2 \|f\| \|f'\| \leq \frac{1}{2}\|f'\|^2 +2\|f\|^2  \\
\label{gl_K_norm_1} &\leq (\tilde{\form} +C_1)[f,f]
\end{align}
with a constant $C_1$ with $C_1\geq \frac{4S}{\ukl}+\alpha^2+C_\pot +2$, depending on $S$, $\ukl$ and $C_\pot$.
Thus $\form_\alpha$ is a sectorial form and possesses a unique associated sectorial operator $H^\alpha$.

From here on the proof loses its connection with the special choice of the model and operators and carries on with abstract properties of the forms and operators $\tilde{\form}$, $\form_\alpha$, $\tilde{H}$ and $H^\alpha$ as usual---see theorem 2.5.7 in \cite{diss_11} or \cite{ExnerHS-07} for details.
\end{proof}

For the application in the MSA we will use the following 
\begin{folg}
\label{folg:CTA}
Let $\Gamma$ be a metric graph with \eqref{geom:uU} and $H^{P,L}(\omega)$ an operator with \RBPLS\ and potential \eqref{pot:char,dichte}.
Then for all radii $r_1 \in \rz$ and vertices $x \in V$, such that $H^{\Lambda_{r_1}(x)}(\omega)$ posses a spectral gap $(s,t)$ with
$|s|$, $|t|\leq R$ and $\lambda \in (s,t)$ with $\eta:=\dist(\{\lambda \},(s,t)^c)$ and $\delta:=\dist\left(\Lout_{r_1}(x),\Lint_{r_2}(v)\right)>0$ the Combes-Thomas estimate
\begin{equation*}
\left\|  \ind_{\Lout_{r_1}(x)}\big(H^{\Lambda_{r_1}(x)}(\omega)-\lambda\big)^{-1}\ind_{\Lint_{r_2}(v)}  \right\|
\leq C_\mathrm{CTA}\, \eta^{-1} \, \exp\big(\tilde{C} \, \sqrt{(t-s)\eta} \,\delta\big),
\end{equation*}
holds, where $\tilde{C}$ and $C_\mathrm{CTA}$ don't depend on the radii $r_1$, $r_2$ or the vertices $x$ and $v$.
\end{folg}
See also the illustration in figure \ref{fig:folg:cta}.

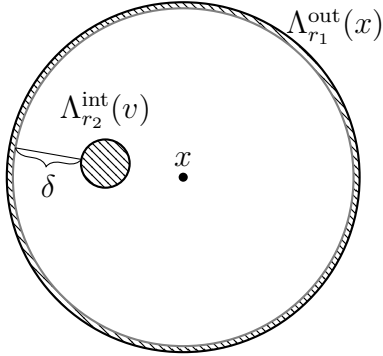
\begin{figure}[!hbt]

\centering
\begin{tikzpicture}[scale=0.8]
\coordinate (x) at (0,0);\coordinate (v) at (170:1.3);
\draw[thick,pattern=north west lines] (x) circle (2.9cm);   
\draw[thick,color=gray,fill=white] (x) circle (2.8cm);   
\draw[thick,pattern=north west lines] (v) circle (0.4cm); 
\knotenn{x}
\node at ($(x)+(2.5,2.5)$) {$\Lout_{r_1}(x)$};
\node[anchor=south] at ($(v)+(0,0.4)$) {$\Lint_{r_2}(v) $};
\draw (170:1.7) -- (170:2.8);
\draw [decorate,decoration={brace,amplitude=5pt}]
(170:1.7)  -- (170:2.8) node [black,midway,below=4pt] { $\delta$};
\end{tikzpicture}
\caption{Illustration Combes-Thomas estimate.}
\label{fig:folg:cta}
\end{figure}

\section{Geometric resolvent inequality}
\label{sec:GRI}
The multiscale analysis is an inductive procedure which gives decay of the resolvent of balls/boxes from one length scale to the next.
Therefore we need a relation of the resolvents of different length scales, which is given by the geometric resolvent inequality (GRI).

We will prove a general version of the GRI and give  a formulation, which can be used in the induction step, in corollary \ref{folg:gru}.

The statements and proofs of this section are based on the same statements in section 2.5 in \cite{Stollmann-01}.
Modifications in the statements and proofs arise from the concrete form of the sesquilinear form and its domain.
We will suppress the dependence on the randomness, as we always can use the uniform estimate $\|\pot_\omega\|\leq C_\pot$.

\begin{hsatz}[Geometric resolvent equation]
\label{satz_RGL}
Let $\Gamma$ be a metric graph satisfying \eqref{geom:uU} and $H^{P,L}(\omega)$ a random operator with boundary conditions of the form \RBPLS\ and a random potential with \eqref{pot:char,dichte}.
Let $\Gamma_{E_1}\subset \Gamma_{E_2}$ be two finite induced subgraphs of $\Gamma$,
$H^{E_1}:=H^{\Gamma_{E_1}}(\omega)$ and $H^{E_2}:=H^{\Gamma_{E_2}}(\omega)$
the restrictions of $H^{P,L}(\omega)$ on the induced subgraphs, $\psi \in C^2_\mathrm{komp}(X_E)$ with $\supp \psi \subset X_{E_1}$
a twice continuously differential, real function, which is constant on an open neighborhood of any vertex in $\Gamma_{E_1}$. Let $z\in \res(H^{E_1})\cap \res(H^{E_2})$. Then
\begin{equation*}
\big(H^{E_1}-z\big)^{-1}\psi=\psi \big(H^{E_2}-z\big)^{-1}+\big(H^{E_1}-z\big)^{-1}\left(\psi' \partial +\partial \psi' \right)\big(H^{E_2}-z\big)^{-1},
\end{equation*}
holds.
\end{hsatz} 
The operator $\partial: W^{1,2}(X_E)\to L^2(X_E)$ maps each function to its first weak derivative. 
Note that by construction of $\psi$ we have $\tr{v}(\psi)\equiv 0$ on $\rk{E_1} $ and $ \tr{v}(\psi)\equiv c_v$, $\str{v}(\psi')\equiv 0 $ on $V_{E_1}$.

\begin{proof}
We will prove the following equivalent equality ($\form^{E_1}:=\form^{\Gamma_{E_1}}$)
\begin{align}
\notag
\langle \psi g,\varphi \rangle=
(\form^{E_1}-z)\Bigl[\underbrace{ \Big( \psi \big(H^{E_2}-z\big)^{-1}+\big(H^{E_1}-z\big)^{-1}\left(\psi' \partial +\partial \psi' \right)\big(H^{E_2}-z\big)^{-1}\Big) g}_{=:h},\varphi \Bigr]
\end{align}
for all $g\in L^2(X_E)$ and all $\varphi \in \dom(\form^{E_1})$.
\begin{itemize}
\item 
Let $f:=\left(H^{E_2}-z\right)^{-1}g$. Now we get $\psi f\in W^{2,2}(X_{E_1})$ from $f\in W^{2,2}(X_{E_1})$ using the product rule and uniforms bounds of $\psi$, $\psi'$ and $\psi''$ in the $L^\infty$-norm. Moreover $\psi f \in \dom(H^{E_1})$ as $\psi$ is supported inside $X_{E_1}$, where $H_1$ and $H_2$ coincide, and the boundary conditions are satisfied as $\tr{v}(\psi f)=c_v\tr{v}(f)$, $\str{v}((\psi f)')=\tr{v}(\psi)\str{v}(f')=c_v\str{v}(f')$.
\item From $\supp \psi' \subset X_{E_1}$ we conclude $\big(H^{E_1}-z\big)^{-1}\left(\psi' \partial +\partial \psi' \right)\big(H^{E_2}-z\big)^{-1} g \in \dom(H^{E_1})$.
\end{itemize}
Thus we know $h\in \dom(H^{E_1})$ and we are able to show the above stated equation (using the above mentioned properties of $\psi$ and linearity of $L_v$, $\pot_\omega$ and partial integration)
{\allowdisplaybreaks 
\begin{flalign*}
&\big(\form^{E_1}-z\big)[h,\varphi]&\\   
&\hspace{0.8cm}=\big(\form^{E_1}-z\big)\Big[ \Big( \psi \big(H^{E_2}-z\big)^{-1}+\big(H^{E_1}
-z\big)^{-1}\left(\psi' \partial +\partial \psi' \right)\big(H^{E_2}-z\big)^{-1}\Big) g,\varphi\Big]&\\
&\hspace{0.8cm}=\left\langle \psi' \big(H^{E_2}-z\big)^{-1}g,\varphi'\right\rangle+\left\langle \psi \Big(\big(H^{E_2}-z\big)^{-1}g \Big)',\varphi'\right\rangle&\\
&\hspace{1.2cm}\phantom{=}\underbrace{+\sum\limits_{v \in V_{E_1}} \left\langle L_v \tr{v}\Big(\big(H^{E_2}-z\big)^{-1}g\Big),\tr{v}(\psi \varphi)\right\rangle +\left\langle (\pot_\omega-z)\big(H^{E_2}-z\big)^{-1}g,\psi \varphi \right\rangle}_{=:\, \mathcal C}&\\
&\hspace{1.2cm}\phantom{=}+\left\langle (\psi' \partial +\partial \psi' )\big(H^{E_2}-z\big)^{-1} g,\varphi\right\rangle&\\
&\hspace{0.8cm}= \left\langle \psi' \big(H^{E_2}-z\big)^{-1}g,\varphi'\right\rangle+\left\langle \Big(\big(H^{E_2}-z\big)^{-1}g \Big)',\psi \varphi'\right\rangle +\mathcal C+\left\langle \partial \big(H^{E_2}-z\big)^{-1}g,\psi'\varphi\right\rangle&\\
&\hspace{1.2cm}\phantom{=}  -\left\langle \psi'\big(H^{E_2}-z\big)^{-1}g,\varphi'\right\rangle+\underbrace{\sum\limits_{e\in E} (\psi' \big(H^{E_2}-z\big)^{-1}g)_e(x)\cdot \overline{\varphi_e(x)}|_0^{l(e)}}_{=0\text{, as }\psi_e'(0)=\psi_e'(l(e))= 0, \text{ for all } e\in E}&\\
&\hspace{0.8cm}=\left\langle \Big( \big(H^{E_2}-z\big)^{-1}g \Big)',(\psi \varphi)' \right\rangle+\left\langle (\pot_\omega-z)\big(H^{E_2}-z\big)^{-1}g,\psi \varphi \right\rangle+\mathcal C&\\
&\hspace{0.8cm}= \big(\form^{E_2}-z\big)\left[\big(H^{E_2}-z\big)^{-1}g,\psi \varphi\right]&\\
&\hspace{0.8cm}=\langle \psi g,\varphi \rangle.&
\end{flalign*}}
Hence proving $(H^{E_1}-z)h=\psi g $ for all $g \in L^2(X_E)$.
\end{proof}

\begin{defn} 
Let $\Gamma$ be a metric graph with \eqref{geom:uU} and $\Gamma_{\tilde{E}}$ an induced subgraph. Let $H^{P,L}(\omega)$ be a random operator with boundary conditions of the form \RBPLS\ and potential with \eqref{pot:char,dichte}. We call a function $f$ weak solution of the equation
\begin{equation*}
H^{P,L}(\omega)f=g \qquad \text{on }\Gamma_{\tilde{E}},
\end{equation*}
if $g\in L^2(X_{\Gamma_{\tilde{E}}})$, $f\in \{\psi \in W^{1,2}(X_{\Gamma_{\tilde{E}}}) \with \tr{v}(\psi)\in \dom(L_v) \text{ for all }v\in \ik{\tilde{E}}  \}$ and for all $\varphi \in \dom\left(\form^{\Gamma_{\tilde{E}}}_\omega\right)$ we have:
\begin{equation*}
\langle f',\varphi' \rangle+\sum\limits_{v\in \ik{\tilde{E}}} \langle L_v\tr{v}(f),\tr{v}(\varphi) \rangle+\langle \pot_\omega f,\varphi\rangle=\langle g,\varphi \rangle.
\end{equation*}
\end{defn}

Thus a function will be tested by the map of the associated form to be a weak solution and
has to satisfy the boundary conditions of the form at the inner vertices, but no boundary conditions at the boundary, as the test funtctions are zero there.
For weak solutions we can prove 

\begin{lem}[Caccioppoli inequality] 
\label{hsatz_ru2}
Let $\Gamma$ be a metric graph with \eqref{geom:uU} and bounded vertex degree. Let $H^{P,L}(\omega)$ be a negative Laplacian satisfying \RBPLS\ and \eqref{pot:char,dichte}.
Let $\Gamma_{E_3}\subset \Gamma_{E_4} $ be induced subgraphs of $\Gamma$ with $\rk{E_3}\subset \ik{E_4} $ and let $g\in L^2(X_{E_4})$. 
Then there exists a constant $C_\mathrm{CP}=C_\mathrm{CP}(\ukl,S,q_-,q_+,c_+)$, such that for all weak solutions $f$ of $H^{P,L}(\omega)f=g$ on $\Gamma_{E_4}$ we have
\begin{equation*}
\|f'\|_{L^2(X_{E_3})}  \leq C_\mathrm{CP}\left( \|f\|_{L^2(X_{E_4})}+ \|g\|_{L^2(X_{E_4})} \right).
\end{equation*}
\end{lem} 

\begin{proof} 
By definition of the subgraphs there exists a cut-off function $\psi$ with $0\leq \psi\leq 1$, $\psi|_{X_{E_3}} \equiv 1$, $\psi|_{X_E \setminus X_{E_4}}\equiv 0$, $\psi\in C^\infty(X_E) $ (i.\,e. also smooth in the vertices), $\tr{v}(\psi)\equiv0$ on $\rk{E_4}$, $\tr{v}(\psi)\equiv 1 $ on $\ik{E_4} $ and with a uniform bound on $\|\psi'_e\|_\infty$ for all edges $e$, depending only on $\ukl$.

Let $g\in L^2(X_{E_4})$, $f$ be as in the theorem and $\phi:=f\psi^2 $. Then $\phi\in \dom(\form^{E_4})\subset W^{1,2}(X_{E_4})$.
By the condition of the theorem and the product rule we get (with scalar products and norms taken in $X_{E_4}$ if not stated otherwise)
\begin{align*}
\|\psi f'\|^2 &\underset{\phantom{\text{Vor.}}}{=}\langle \psi f',\psi f' \rangle = \langle f',\phi' \rangle -2 \langle \psi f',f\psi' \rangle\\
&\hspace{-9pt}\underset{\text{weak sol.}}{=} \langle g, \phi \rangle -\sum\limits_{v\in\ik{E_4}}\langle L_v\tr{v}(f),\tr{v}(\phi)\rangle_{\rrv}-\langle \pot_\omega f,\phi \rangle -2\langle \psi f',f\psi' \rangle\\
&\underset{\phantom{\text{Vor.}}}{\leq}  \|g\|\|f\|+ C_\pot \|f\|^2 +2 \|\psi'\|_\infty \|f\| \|\psi f'\|-\sum\limits_{v\in \ik{E_4}}\langle L_v \tr{v}(\psi f),\tr{v}(\psi f)\rangle_{\rrv}.
\end{align*}
Using the uniform bound $S$ of the operators $L_v$ and a uniform Sobolev inequality (e.g. see equation (4) in \cite{LSV})
we get on the subgraph induced by $E_4$
\begin{align*}
\sum\limits_{v\in \ik{E_4}} \langle L_v (\psi f)(v),(\psi f)(v)\rangle_{\rrv}
\geq -2S\left( \frac{2}{\ep}\|f\|^2+2\ep \| \psi f' \|^2+2\ep\|\psi'f \|^2 \right)
\end{align*}
for all $\ep$ with $0<\ep\leq \ukl$. Choosing $\ep$, such that $1-4S\ep\geq \frac{1}{2}$, i.\,e. $\ep :=\min \{\ukl,\frac{1}{8S} \}$ yields
\begin{align*}
\|\psi f'\|^2 \leq 2\|g\|\|f\|+ \underbrace{2\left(C_\pot+\frac{4S}{\ep}+4S\ep\|\psi'\|_\infty\right)}_{=:\widehat{C}}\|f\|^2 +4 \|\psi'\|_\infty \|f\| \|\psi f'\|,
\end{align*}
which can be read as a quadratic inequality in $x=\psi f'$. It is satisfied between its two roots. Estimating the roots from above we get
\begin{align*}
\|f'\|_{L^2(X_{E_3})} &\leq \|\psi f'\|_{L^2(X_{E_4})}\leq \widetilde{C}\|f\|_{L^2(X_{E_4})} \hspace{-1.5pt}+\hspace{-1.5pt} \sqrt{\widetilde{C}^2+\widehat{C}} \|f\|_{L^2(X_{E_4})} \hspace{-1.5pt}+\hspace{-1.5pt}\frac{1}{ \sqrt{\widetilde{C}^2+\widehat{C}}} \|g\|_{L^2(X_{E_4})},
\end{align*}
where $\widetilde{C}=2\|\psi'\|_\infty$ (for details see theorem 2.6.3 in \cite{diss_11}).
\end{proof} 

\begin{satz}[Geometric resolvent inequality] 
\label{satz_GRU}
Let $\Gamma$ be a metric graph with finite vertex degree and \eqref{geom:uU}. Let $H^{P,L}(\omega)$ be a random Laplacian with \RBPLS\ and \eqref{pot:char,dichte}.
Let $\Gamma_{E_1} \subset \Gamma_{E_2}$ be finite induced subgraphs of $\Gamma$, $I$ a bounded interval and
$\lambda\in \res(H^{E_1})\cap \res(H^{E_2})\cap I $.
Let $\Gamma_{E_A}\subset \Gamma_{E_1}$ and $\Gamma_{E_B} \subset \Gamma_{E_2} \setminus \Gamma_{E_1}$ be induced subgraphs.
If there exists a function $\varphi\in C^\infty_\mathrm{comp}(X_{E_1}) $ with $\varphi $ is constant on a neighborhood of any vertex in
$V_{E_1} $, identically zero at any boundary vertex of $\Gamma_{E_1} $, $\varphi|_{X_{E_A}}\equiv 1 $ and there exist two induced subgraphs $\Gamma_{E_3}\subset \Gamma_{E_4}\subset \Gamma_{E_1}$, such that $\supp \varphi' \subset X_{E_3}$, $\rk{E_3}\subset \ik{E_4}$ and $\Gamma_{E_A}$ and $\Gamma_{E_4}$ are disjoint, then there exists a constant $C_\mathrm{GRU}$, depending only on $S$, $\ukl$, the constants of the potential and the interval $I$, such that 
\begin{flalign*}
&\left\| \ind_{X_{E_B}} \big(H^{E_2}-\lambda\big)^{-1} \ind_{X_{E_A}} \right\| 
\leq C_\mathrm{GRU} \left\| \ind_{X_{E_B}} \big(H^{E_2}-\lambda\big)^{-1} \ind_{X_{E_4}} \right\|
\left\| \ind_{X_{E_4}} \big(H^{E_1}-\lambda\big)^{-1} \ind_{X_{E_A}} \right\|&
\end{flalign*}
holds.
\end{satz}
If the six induced subgraphs are chosen first, with the necessary properties, then the cut-off function $\varphi$ exists, if the induced subgraphs are big enough and have enough distance. See also the illustration in figure \ref{abb:gru} and corollary \ref{folg:gru} for explanation. 
\begin{figure}[!htb]
\centering
\begin{tikzpicture}[scale=0.7]
\draw[thick] (-4,-3.5) -- (4,-3.5) -- (4,3.5) -- (-4,3.5) -- (-4,-3.5);\node[anchor=west] at (4,2.5) {$\Gamma_{E_2}$};
\draw[thick] (-3.5,-3) -- (2,-3) -- (2,3) -- (-3.5,3) -- (-3.5,-3);\node[anchor=west] at (2,2.5) {$\Gamma_{E_1}$};
\draw[thick,pattern=north west lines] (3,-1.2) ellipse (0.75 and 1.5);
\draw[thick,color=white,fill=white] (3,-1.2) circle (0.55cm); \node at(3,-1.2) {$\Gamma_{E_B}$}; 
\draw[thick,pattern=north west lines] (-1,-0) ellipse (2cm and 1.8cm);
\draw[thick,fill=white] (-1,-0) ellipse (1.5cm and 1.6cm);
\draw[thick] (-0.4,1.6) -- (0.6,1.6);\node[anchor=west] at (0.6,1.6) {$\Gamma_{E_4}$};
\draw[thick,pattern=north west lines] (-1,-0) circle (1);
\draw[thick,color=white,fill=white] (-1,-0) circle (0.55cm); \node at(-1,-0) {$\Gamma_{E_A}$}; 
\end{tikzpicture}
\caption{Induced subgraph for the geometric resolvent inequality.}
\label{abb:gru}
\end{figure}
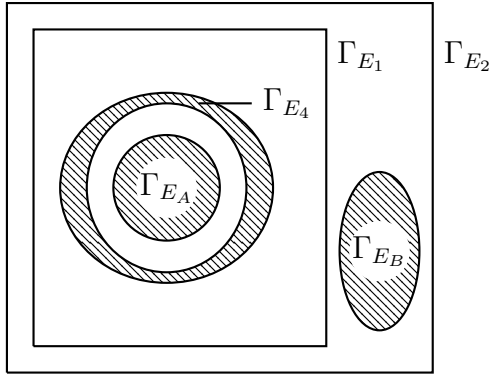

\begin{proof}
With $\varphi|_{X_{E_A}}\equiv 1, \varphi|_{X_{E_B}}\equiv 0$ we see
{\allowdisplaybreaks 
\begin{align*}
\left\| \ind_{X_{E_B}} \big(H^{E_2}-\lambda\big)^{-1}\ind_{X_{E_A}} \right\|
&=\left\| \ind_{X_{E_A}} \big(H^{E_2}-\lambda\big)^{-1}\ind_{X_{E_B}} \right\| \\
&=\left\|  \ind_{X_{E_A}} \left( \varphi \big(H^{E_2}-\lambda\big)^{-1} -\big(H^{E_1}-\lambda\big)^{-1}\varphi \right) \ind_{X_{E_B}} \right\| \\
&\hspace{-8pt}\underset{\text{pr. }\ref{satz_RGL}}{=}\left\| \ind_{X_{E_A}} \big(H^{E_1}-\lambda\big)^{-1}(\partial \varphi'+\varphi'\partial)\big(H^{E_2}-z\big)^{-1} \ind_{X_{E_B}} \right\| \\
&\leq \underbrace{\left\|\ind_{X_{E_A}} \big(H^{E_1}-\lambda\big)^{-1} \partial \varphi' \big(H^{E_2}-\lambda\big)^{-1}\ind_{X_{E_B}} \right\|}_{(i)} \\
&\phantom{=\ }+\underbrace{\left\|\ind_{X_{E_A}} \big(H^{E_1}-\lambda\big)^{-1}  \varphi' \partial \big(H^{E_2}-\lambda\big)^{-1}\ind_{X_{E_B}}\right\|}_{(ii)}.
\end{align*}}
First we want to treat $(i)$, which due to $\supp \varphi'\subset X_{E_3}$ yields
\begin{align*}
(i)
&=\left\|\ind_{X_{E_A}} \big(H^{E_1} -\lambda\big)^{-1} \partial \ind_{X_{E_3}} \ind_{X_{E_4}}  \varphi' \big(H^{E_2}-\lambda\big)^{-1}\ind_{X_{E_B}}\right\| \\
&\leq \left\|\varphi'\right\|_\infty \underbrace{\left\|\ind_{X_{E_A}} \big(H^{E_1} -\lambda\big)^{-1} \partial \ind_{X_{E_3}} \right\|}_{(iii)} \left\| \ind_{X_{E_4}}  \big(H^{E_2}-\lambda\big)^{-1}\ind_{X_{E_B}}\right\|.
\end{align*}
We have to estimate (iii). Therefore we define for $f\in L^2(X_{E_1})$, $g:=\ind_{X_{E_A}} f $ and $h:=\left(H^{E_1}-\lambda\right)^{-1}g $, such that $h\in \dom(H^{\Gamma_{E_1}})\subset\dom(\form^{E_1})$. By partial integration (see proposition 2.1 in \cite{LSV} for rearranging the sums) we get for all $w\in \dom(\form^{E_4})$
\begin{flalign*}
&\langle h',w' \rangle_{L^2(X_{E_4})}+\sum\limits_{v \in \ik{E_4}}\langle L_v \tr{v}(h),\tr{v}(w) \rangle+(\pot_\omega -\lambda)\langle h,w \rangle_{L^2(X_{E_4})}&\\
&\hspace{1.5cm}=-\langle h'',w \rangle_{L^2(X_{E_4})} -\sum\limits_{v \in \ik{E_4}} \langle \str{v}(h'),\tr{v}(w) \rangle&\\
&\hspace{1.5cm}\hspace{2em}+\sum\limits_{v \in \ik{E_4}}\langle L_v \tr{v}(h),\tr{v}(w) \rangle+(\pot_\omega -\lambda)\langle h,w \rangle_{L^2(X_{E_4})}&\\
&\hspace{1.5cm}=\left\langle -h''+(\pot_\omega-\lambda)h,w \right\rangle_{L^2(X_{E_4})}+\sum\limits_{v \in \ik{E_4}} \underbrace{\langle L_v \tr{v}(h)-\str{v}(h'),\tr{v}(w)\rangle}_{=0,\text{ by boundary conditions}} &\\
&\hspace{1.5cm}=\langle g,w \rangle_{L^2(X_{E_4})},&
\end{flalign*}
Thus we see that $h$ is a weak solution of $H^{P,L}(\omega)=g+\lambda h$ on $\Gamma_{E_4}$.
With lemma \ref{hsatz_ru2} we conclude
\begin{flalign}
\label{gl_ru_01}
\notag \|h'\|_{L^2(X_{E_3})}&=\left\|\ind_{X_{E_3}}\partial \big(H^{E_1}-\lambda\big)^{-1}\ind_{X_{E_A}} f \right\|&\\
\notag &\leq C_\mathrm{CP}\bigg((1+|\lambda|)\left\| \big(H^{E_1}-\lambda\big)^{-1}\ind_{X_{E_A}} f\right\|_{L^2(X_{E_4})}+\underbrace{\left\|\ind_{X_{E_A}} f \right\|_{L^2(X_{E_4})}}_{=0} \bigg).&
\end{flalign}
Hence $\ind_{X_{E_3}}\partial \left(H^{E_1}-\lambda\right)^{-1}\ind_{X_{E_A}}$ is a bounded operator yielding
\begin{align*}
(iii)=\left\|\ind_{X_{E_A}} \big(H^{E_1}-\lambda\big)^{-1}\partial \ind_{X_{E_3}}\right\|
&\leq C_\mathrm{CP}(1+|\lambda|)\left\| \ind_{X_{E_A}} \big(H^{E_1}-\lambda\big)^{-1} \ind_{X_{E_4}}\right\|.
\end{align*}
So we can estimate $(i)$ as claimed with a constant $C_\mathrm{CP}(1+|\lambda|)\|\varphi'\|$. Part $(ii)$ can be treated analogously, yielding the same constant and overall
\begin{flalign*}
&\left\|\ind_{X_{E_B}} \big(H^{E_2}-\lambda\big)^{-1} \ind_{X_{E_A}}\right\|&\\
 &\hspace{1cm}\leq C_\mathrm{GRU} \left\| \ind_{X_{E_A}} \big(H^{E_1}-\lambda\big)^{-1} \ind_{X_{E_4}}\right\| \left\|\ind_{X_{E_4}} \big(H^{E_2}-\lambda\big)^{-1} \ind_{X_{E_B}} \right\|,
\end{flalign*}
where $C_\mathrm{GRU}=2\, C_\mathrm{CP}(1+|\lambda|) \|\varphi'\|_\infty$, depending only on $\ukl$, $S$ the interval $I$ and $C_\pot$.
\end{proof}

\begin{folg}
\label{folg:gru}
Let $\Gamma$ be a metric graph with \geomalles\ and $H^{P,L}(\omega)$ satisfy \RBPLS\ and \eqref{pot:char,dichte}, $R$, $s$, $r$ three length scales and $x$, $v$, $v_1\in V$ three vertices, such that
\begin{itemize}
 \item $\Lambda_s(v)\subset \Lambda_R(x) $ where $\Lout_R(x)$ and $ \Lambda_s(v)$ are disjoint,
 \item $\Lambda_r(v_1)\subset \Lambda_s(v)$.
\end{itemize}
Let $I$ be a bounded interval and  $\lambda \in \res(H^{\Lambda_R(x)}(\omega))\cap \res(H^{\Lambda_s(v)}(\omega))\cap I$.
Then we have
\begin{flalign*}
&\left\| \ind_{\Lout_R(x)} \big( H^{\Lambda_R(x)}(\omega)-\lambda\big)^{-1} \ind_{\Lint_r(v_1)} \right\| &\\
&\hspace{0.8cm}\leq C_{\mathrm{GRU}} \left\| \ind_{\Lout_R(x)}\big( H^{\Lambda_R(x)}(\omega)-\lambda  \big)^{-1} \ind_{\Lout_s(v)}\right\|
\cdot \left\| \ind_{\Lout_s(v)}\big( H^{\Lambda_s(v)}(\omega)-\lambda  \big)^{-1} \ind_{\Lint_r(v_1)}\right\|.
\end{flalign*}

\end{folg} 

\begin{bem}
The sets used in the corollary in the notation of theorem \ref{satz_GRU} are:
$E_2=E(x,R)$, $E_1=E(v,s)$, $E_A=E(v_1,\frac{r}{3})$ inducing $\Lint_r(v_1)$, $E_B=E(x,R)\setminus E(x,R-3\okl)$ inducing $\Lout_R(x)$, 
$E_3=E(v,s-\okl)\setminus E(v,s-2\okl)$ and $E_4=E(v,s)\setminus E(v,s-3\okl)$ inducing $\Lout_s(v)$.
The construction of the sets $E_B$ and $E_3$ in this case are the actual reason of defining the exteriors $\Lout$ with a width of $3\okl$.

The corollary will be used with container sets, taking the role of $\Lambda_s(v)$ and in the special case $r=s$ and $v=v_1$.
\end{bem}

\begin{figure}[!htb]
\hspace*{2.2cm}
\begin{minipage}[hb]{5.35cm}
\centering
\begin{tikzpicture}[scale=0.8]
\coordinate (x) at (0,0);\coordinate (v) at (-0.7,-0.4);\coordinate (v1) at (-1.2,-1.1);
\draw[thick,pattern=north west lines] (x) circle (2.9cm);   
\draw[thick,color=gray,fill=white] (x) circle (2.8cm);   
\draw[thick,pattern=north west lines] (v) circle  (1.8cm);  
\draw[thick,color=gray,fill=white] (v) circle (1.7cm); 
\draw[thick,pattern=north west lines] (v1) circle (0.25cm);
\knotenn{v}
\knotenn{x}
\node at ($(x)+(2.5,2.5)$) {$\Lout_R(x)$};
\node at ($(v)+(1.7,1.7)$) {$\Lout_s(v) $};
\draw (v1)--($(v1)+(0.4,0)$);
\node[anchor=west] at ($(v1)+(0.24,0)$) {$\Lint_r(v_1) $};
\end{tikzpicture}
\end{minipage}
\begin{minipage}[hb]{5.35cm}
\centering
\begin{tikzpicture}[scale=0.8]
\coordinate (x) at (0,0);\coordinate (v) at (-1.3,+0.8);
\draw[thick,pattern=north west lines] (x) circle (2.9cm);   
\draw[thick,color=gray,fill=white] (x) circle (2.8cm);   
\draw[thick,pattern=north west lines] (v) circle  (0.9cm);  
\draw[thick,color=gray,fill=white] (v) circle (0.8cm); 
\draw[thick,pattern=north west lines] (v) circle (0.25cm);
\knotenn{x}
\node at ($(x)+(2.5,2.5)$) {$\Lout_R(x)$};
\draw (v)--($(v)+(1.1,0)$);
\node at ($(v)+(1.02,1.02)$) {$\Lout_r(v) $};
\node[anchor=west] at ($(v)+(1.1,0)$) {$\Lint_r(v) $};
\end{tikzpicture}
\end{minipage}
\caption{Induced subgraphs as used in corollary \ref{folg:gru} and special case $v=v_1$, $s=r$.}
\label{fig:folg:gru}
\end{figure}
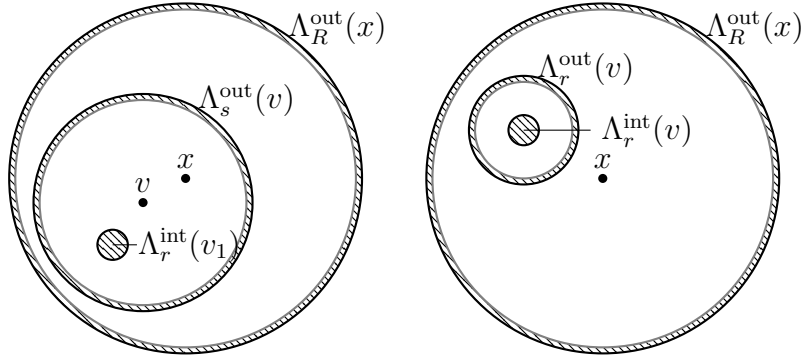

\section{Weyl asymptotics}
\label{sec:WEYL}
In this section we will discuss uniform bounds on the distribution of eigenvalues for Laplace operators on finite induced subgraphs.
We will use the spectral counting function of the Dirichlet Laplacian to estimate the perturbation by other boundary conditions and the potential.

\begin{defn}
Let $\Gamma=(E,V,l,i,j)$ be a metric graph.
The spectral counting function $n:\rz \to \nz\cup\{0,\infty\} $ of a self-adjoint operator $H$ on $\Gamma$ is defined by
\begin{equation*}
n(\lambda)=\Tr \left(\ind_{(-\infty,\lambda]}(H) \right).
\end{equation*}
\end{defn}
For discrete spectrum it yields the number of eigenvalues below $\lambda$ counted with multiplicity and in fact is an eigenvalue counting function.
\begin{hsatz} 
\label{hsatz:weyl_dir}
The following holds:
\begin{enumerate}
\item 
Let $\Gamma$ be a finite metric graph with finite edge lengths and $H^{P,L}$ a self-adjoint Laplacian with boundary conditions \RBPLS\ on $\Gamma$. Then the spectrum of $H^{P,L}$ is discrete.
\item 
Let $H$ be a self-adjoint operator with compact resolvent and $V$ a bounded potential on a Hilbert space $\hil$. Then $H+V$ has compact resolvent.
\item 
The negative Laplace operator with Dirichlet boundary conditions on an interval with length $l$ possesses the eigenvalues $\frac{n^2
 \pi^2}{l^2}$ for $n\in \nz$.
\item 
The eigenvalue counting function of the operator from point 3 has the form $n(\lambda)=\left\lfloor \frac{l}{\pi}\sqrt{\lambda} \right\rfloor $ for $\lambda\geq 0$ and $n(\lambda)=0$ for $\lambda <0$.
\end{enumerate}
\end{hsatz} 
\begin{proof}
\begin{enumerate}
 \item This was proven in theorem 18 in \cite{Kuchment-04}.
 \item Easily follows from the second resolvent identity.
 \item Of course, this is known. The solution of the differential equation $-f''=\lambda f$ with the boundary values $f(0)=f(l)=0$ yields  
the above eigenvalues with the eigenfunctions $c\cdot \sin(\sqrt{\lambda}x) $ for $\lambda >0 $ and only the trivial solution for $\lambda \leq 0$.
 \item Follows from 3. \qedhere
\end{enumerate}
\end{proof}

Now we will estimate the eigenvalue counting function of restrictions of \RBPLS\ Laplacians to finite subgraphs.
To achieve this we will use the corresponding function for the Dirichlet boundary condition.
The idea is the same as in \cite{GruberLV-07}, where this was proven for metric graphs with constant edge lengths.
Point 4 of the last proposition yields
\begin{bem}
\label{bem:EZf:Dir}
Let $\Gamma$ be a finite metric graph with finite edge lengths and $\okl$ be the maximal edge length. Then we have the following estimate for the eigenvalue counting function of the Dirchlet Laplace $H_\mathrm{D}$ on $\Gamma$
\begin{equation*}
n_{H_\mathrm{D}}(\lambda) \leq \begin{cases} \frac{\okl}{\pi}\sqrt{\lambda}\cdot |E|, & \lambda \geq 0 \\ 0, & \lambda<0, \end{cases}
\end{equation*}
where $|E|$ stands for the number of edges in $E$.
\end{bem}

\begin{hsatz} 
\label{hsatz:EZf:diff}
Let $\Gamma$ be a finite metric graph with finite edge lengths and $H_1$ and $H_2$ be two self-adjoint realizations of the negative Laplacian of the form \RBPLS\ on $\Gamma$. Then we have
\begin{equation*}
|n_{H_1}(\lambda)-n_{H_2}(\lambda)|\leq 2\,|E|.
\end{equation*}
\end{hsatz}
\begin{proof} 
We have $\dom(H_i)/D_0=2|E|$, where  $D_0:=W^{2,2}_0(X_E)$ is the domain of the minimal Laplace operator. 
Stating $n_H(\lambda)$ in terms of dimensions of eigenfunction spaces, we get
\begin{align*}
n_{H_2}(\lambda)&=\max\{ \dim Y \with Y\subset \dom(H_2), H_2\big|_Y \leq \lambda \}\\
&\leq \max\{ \dim Y \with Y\subset D_0, H_2\big|_Y \leq \lambda \}+2\,|E|,
\intertext{as we cut a subspace of dimension $2|E|$; moreover $H_1\big|_{D_0}=H_2\big|_{D_0}$ holds}
&=\max\{ \dim Y \with Y\subset D_0, H_1\big|_Y \leq \lambda \}+2|E|\\
&\leq \max\{ \dim Y \with Y\subset \dom(H_1), H_1\big|_Y \leq \lambda \}+2|E|\\
&=n_{H_1}(\lambda)+2|E|.\qedhere
\end{align*}
\end{proof}

\begin{hsatz}
Let $\Gamma$ be a finite metric graph with finite edge lengths and $H$ a negative Laplacian with boundary conditions \RBPLS. Let $W$ be a bounded self-adjoint operator on $L^2(X_E)$. Then the eigenvalue counting function of $H+W$ satisfies
\begin{equation*}
n_{H+W}(\lambda) \leq
\begin{cases}
|E|\left[2+\frac{\okl}{\pi} \left(\sqrt{\lambda}+\|W\|\right)\right] , & \lambda\geq -\|W\| \\
2\,|E|, & \lambda < -\|W\|,
\end{cases}
\end{equation*}
where $\okl$ ist the length of one of the longest edges in $\Gamma$.
\end{hsatz}
\begin{proof} 
By min-max principle and proposition \ref{hsatz:EZf:diff} we get with the Dirichlet Laplacian $H_\mathrm{D}$:
\begin{align*}
n_{H+W}(\lambda)&\leq n_H(\lambda+\|W\|)\\
&\leq 2\,|E|+ n_{H_\mathrm{D}}(\lambda +\|W\|).
\end{align*}
Inserting the eigenvalue counting function of $H_\mathrm{D}$ from proposition \ref{bem:EZf:Dir} we conclude the assertion.
\end{proof} 
We finish this section by stating a corollary for special induced subgraphs of infinite metric graphs, namely balls with radius $r$.
\begin{folg} 
\label{folg:weyl}
Let $\Gamma$  be a metric graph with \geomalles\ of polynomial growth with degree $d$. Let $(H^{P,L}(\omega))$ be a random negative Laplacian with \RBPLS\ and \eqref{pot:char,dichte}.
Then for the restriction of $H^{P,L}$ to balls $\Lambda_r(v)$ with radius $r$ we have the estimate:
\begin{equation*}
n_{H^{\Lambda_r(v)}(\omega)}(\lambda) \leq \begin{cases}\left( 2+\frac{(\sqrt{\lambda}+\sqrt{C_\pot})\okl  }{\pi} \right)\frac{\cp \cdot r^d}{\ukl},& \lambda \geq -C_\pot\\\frac{2\,\cp\cdot r^d}{\ukl}, & \lambda<-C_\pot.    \end{cases}
\end{equation*}
Thus for each bounded interval $I$ there exists a constant $C_\mathrm{Weyl}$, which depends only on $\ukl$, $\okl$, $C_\pot$, $\cp$ and $I$, but not on $v$ and $\omega$, with
\begin{equation*}
\fa \lambda \in I: \qquad
n_{H^{\Lambda_r(v)}(\omega)}(\lambda)\leq C_\mathrm{Weyl}\cdot r^d.
\end{equation*}
\end{folg}

\section{Initial length scale and Wegner estimate}
\label{ab:ALAW}

In the last sections we proved estimates uniform in $\omega$. Now we will state two results which depend on $\omega$, namely the initial length scale estimate and the Wegner estimate.

With the initial length scale estimate we prove the spectrum of restrictions to finite subgraphs to lie at the lower bound of the spectrum with very low probability only. This will be concluded from an assumption on the single site measure, not to be concentrated at the lower end of its support, which is usually called \enquote{high disorder}:
\begin{ann}
\hypertarget{pot:alles}{}
\label{ann_unordnung}
Let $\Gamma$ be a metric graph with \eqref{geom:uU} and uniform polynomial growth of degree $d$. We assume the existence of a constant $\tau > \frac{3d}{2}-1$, such that
\begin{equation}
\label{pot:unord}\tag{pot:disord}
\mu([q_-,q_-+h])\leq h^\tau \qquad \text{for small } h.
\end{equation}
\end{ann}
For the proof $\tau > \tfrac{d}{2}$ would be sufficient, but for later reasons we choose another bound.
The proof of the initial length scale estimate is done as usual---see Theorem 3.2 in \cite{ExnerHS-07} for the Kirchhoff Laplace on $\gz^d$. Here the dimension $d$ is taken by the degree of the polynomial growth. 
By $\leos$ we denote the lower bound of $H^{P,L}+(q_- \nu_e)_{e \in E}$, which is the lower bound of the spectrum of all random operators $H^{P,L}(\omega)$.
\begin{satz}[Initial length scale estimate]
\label{satz:ALA}
Let $\Gamma$ be a metric graph which satisfies \geomalles\ and $(H^{P,L}(\omega))$  satisfy \RBPLS\ and \potalles.
Then for each $\xi\in (0,2\tau-d) $ there exists a $\beta>0$ (depending on $\tau$ and $\xi$) and a radius $\rone \geq \rgeom$, such that for all $v\in V$ and $r\geq \rone$ we have
\begin{equation*}
\PP\left\{ \dist\left(\sigma\left(H^{\Lambda_r(v)}(\omega)\right),\{\leos\}\right) \leq r^{\beta-2} \right\} \leq r^{-\xi}.
\end{equation*}
\end{satz}
\begin{proof} 
For a radius $r$ and a distance $h$ we define the set $\Omega_{r,h}$ as
\begin{equation*}
\Omega_{r,h}:=\{\omega \in\Omega \with \fa v\in V,~\fa e\in E(v,r) :  q_e(\omega)\geq q_-+h  \}.
\end{equation*}
With \eqref{pot:char} we conclude for all $f \in \dom(H^{P,L})$ with $\|f\|=1$:
\begin{equation*}
\left\langle H^{\Lambda_r(v)}(\omega) f,f \right\rangle \geq \leos  +\frac{h}{c_-}c_-=\leos  +h \qquad \text{for all } \omega \in \Omega_{r,\frac{h}{c_-}}.
\end{equation*}
Thus the probability that $H^{\Lambda_r(v)}(\omega)$ has spectrum in $[\leos,\leos+h]$ is greater or equal to the probability of the complementary event to $ \Omega_{r,\frac{h}{c-}}$
\begin{align*}
\PP\left(\Omega_{r,\frac{h}{c_-}}\right)&=\left(1- \mu\left(\left[q_-,q_-+\frac{h}{c_-} \right]\right) \right)^{|E(v,r)|}\\
&\geq 1-|E(v,r)| \cdot \mu\left(\left[q_-,q_-+\frac{h}{c_-} \right]\right) \\
&\hspace{-1.4cm} \underset{ \eqref{geom:poly},\eqref{pot:unord}}{\geq} 1-\frac{\cp\cdot r^d}{\ukl}\left(\frac{h}{c_-}\right)^\tau.
\end{align*}
For $\xi\in (0,2\tau-d)$ we choose $\beta$  with $\beta<\frac{2\tau-d-\xi}{\tau}$, such that $\xi < \tau(2-\beta)-d $ holds. Moreover we set $h:=r^{\beta-2}$. Then
\begin{equation*}
\PP\left(\Omega_{r,\frac{r^{\beta-2}}{c_-}}\right) \geq 1- \underbrace{\frac{\cp\cdot r^{\xi-\tau(2-\beta)+d}}{\ukl\cdot c_-^\tau}}_{(i)}\cdot  r^{-\xi},
\end{equation*}
where $(i) \leq 1$ for $r \geq \rone$ (where $\rone$ depends on $\ukl,\cp,c_-,\tau,d,\xi$ and $\beta$), as the exponent of $r$ in $(i)$ is smaller as zero.
\end{proof} 
Note that there is a proof of Lifshitz-type asymptotics for the metric graph $\gz^d$ and Kirchhoff boundary conditions in \cite{Sabri}, yielding localization for more general potentials.


The Wegner estimate gives an estimate over the mean value of the number of eigenvalues of a restriction to an induced subgraph in a small energy interval. 
\begin{defn}
For $\ep>0$  we define the modulus of continuity of a probability measure $\mu$ by
\begin{equation*}
s(\mu,\ep):=\sup\{ \mu([\lambda-\ep,\lambda+\ep]) \with\lambda \in \rz \}.
\end{equation*}
\end{defn}

In the next theorem we will state the Wegner estimate from theorem 6 in \cite{GruberHV-08}, adapted to our model. The estimates presented there are much more general, but cover the case presented here---see the remark after definition 3 in the second section in \cite{GruberHV-08}.
\begin{satz} 
\label{satz_wegner}
Let $\Gamma$ be a metric graph with \geomalles\ and $(H^{P,L}(\omega))$ a random operator with \RBPLS\ and \eqref{pot:char,dichte} on $\Gamma$.
Let $\lambda_0 \in \rz$. Then there exists a constant $C_\mathrm{W}=C_\mathrm{W}(\lambda_0)$, such that for all $\lambda\leq \lambda_0$, all finite sets of edges $\tilde{E}\subset E$ and all $\ep\leq \frac{1}{2}$ we have
\begin{equation*}
\label{gl_WA}
\mathbb E\left\{ \Tr\left( \ind_{[\lambda-\ep,\lambda+\ep]} \left(H^{\Gamma_{\tilde{E}}}(\omega)\right)  \right)  \right\}
\leq C_\mathrm{W}\cdot  s(\mu,\ep) \cdot|\tilde{E}|.
\end{equation*}
\end{satz} 
For induced subgraphs corresponding to balls with radius $r$, induced by edge sets $E(v,r)$, we get by the uniform polynomial growth
\begin{equation*}
|E(v,r)|\leq \frac{\cp \cdot r^d}{\ukl}.
\end{equation*}

\section{Existence of generalized eigenfunctions}
\label{ab:VEF}

The existence of generalized eigenfunctions (GEF) is an essential tool of the multiscale analysis. In this section we state the existence of GEF spectrally almost everywhere, which are of polynomial growth. Combined with the decay of the \enquote{local resolvent}---the result of the MSA---this yields localization.
Furthermore there is an expansion in generalized eigenfunctions, see \cite{PSW} for general theory and \cite{LSS-08} for results on metric graphs.

For the definition of GEF we will use the locally square integrable function spaces from section \ref{ab:ITG}.
\begin{defn}
\label{def:VEF}
Let $X$ be a topological space with measure $\dx$. Let $H$ be a local operator on $L^2(X)$, which means: with $f\in L^2_\mathrm{komp}(X)\cap \dom(H)$ also $Hf \in  L^2_\mathrm{komp}(X)$ holds and $\mathcal D :=\dom(H)\cap L^2_\mathrm{komp}(X)$ is a core for $H$, i.\,e. $H=\overline{H|_\mathcal D} $. 
We call a nontrivial function $f\in L^2_\mathrm{lok}(X)$ generalized eigenfunction for $H$ to $\lambda \in \kz$, if
\begin{equation*}
\langle H \phi | f \rangle=\lambda \langle \phi|f\rangle \qquad \text{for all }\phi\in \mathcal D.
\end{equation*}
\end{defn}
Next we will show, that this definition is applicable for $H^{P,L}(\omega)$.
\begin{hsatz} 
\label{hsatz:VEF_RB}
Let $\Gamma$ be a metric graph with \eqref{geom:u}, bounded vertex degree and $H^{P,L}(\omega)$ with \RBPLS\ and \eqref{pot:char,dichte}.
Then for $H^{P,L}(\omega)$ on $L^2(X_E)$ we have:
\begin{enumerate}
\item The operator $H^{P,L}(\omega)$ is local.
\item $\dom(H^{P,L}(\omega))\cap L^2\komp(X_E)$ is a core for $H^{P,L}(\omega)$.
\item Let $f$ be a generalized eigenfunction for $H^{P,L}(\omega)$. 
Then we have $f\in W^{2,2}\lok(X_E)$, $f$ satisfies the boundary condition of the form \RBPLS.
\end{enumerate}
\end{hsatz}
\begin{proof}
{\it 1.} is trivial for the negative Laplacian and potentials acting as multiplication operators.
{\it 2.} and {\it 3.} can be found in \cite{LSS-08} in Proposition 5.3. (where compactly supported cut-off functions, supported in a small neighborhood of a vertex are used; note $f''=\pot_\omega f-\lambda f$ holds for the second weak derivative).
\end{proof}

Thus a GEF satisfies the boundary conditions and an eigenvalue equation---where the operator is replaced by its map:
\begin{equation*}
-\Delta f +\pot_\omega f=\lambda f.
\end{equation*}
Thereby we can construct an estimate of the norm of the first derivative, by the norm of the GEF.
\begin{hsatz} 
\label{hsatz:VEF:kegel}
Let $\Gamma$ be a metric graph with \eqref{geom:uU}, $(H^{P,L}(\omega))$ a random operator with \RBPLS\ and \eqref{pot:char,dichte}.
Let $f$ be a generalized eigenfunction for $\lambda \in \sigma(H^{P,L}(\omega))$. 
Then there is a constant $C_\mathrm{cone}$, depending only on $|\lambda|$, $\ukl$ and $C_\pot$, such that for all edges $e\in E$ we have the estimate
\begin{equation*}
\|f'\|^2_{L^2(I_e)}\leq C_\mathrm{cone} \|f\|^2_{L^2(I_e)}.
\end{equation*}
\end{hsatz}
\begin{proof}
If a domain $\Omega$ satisfies the cone condition---which is obviously true for edges---the first derivative of a $W^{2,2}(\Omega)$-function can be bounded by the second derivative and the function itself in the following way
\begin{equation*}
\|f'_e\|_{L^2(I_e)} \leq C_\mathrm{K} \left(\ep \|f''\|_{L^2(I_e)}+\ep^{-1}\|f\|_{L^2(I_e)}  \right),
\end{equation*}
for all $\ep\leq \ukl$ and a constant $C_\mathrm{K}$, depending only on the relevant cone (see lemma 5.5 in \cite{AdamsF}).
With $f''=\pot_\omega f -\lambda f$ we get
\begin{equation*}
\|f'_e\|_{L^2(I_e)} \leq C_\mathrm{K} \left(\ep (C_\pot+|\lambda|)\|f\|_{L^2(I_e)}+\ep^{-1}\|f\|_{L^2(I_e)}  \right). \qedhere
\end{equation*}
\end{proof}

An expansion in generalized eigenfunctions for metric graphs was proven in \cite{LSS-08}, using general theory from \cite{PSW}.
From corollary 5.4 in \cite{LSS-08} we deduce:
\begin{folg} 
\label{cor:VEF}
Let $\Gamma$ be a metric graph with \eqref{geom:u} and $H^{P,L}(\omega)$ an operator with \RBPLS\ and \eqref{pot:char,dichte}.
Let $\varrho$ be a spectral measure for $H^{P,L}(\omega)$ and $w:X_E\to [1,\infty)$ a weight function with $w^{-1}\in L^2(X_E)$. Then 
there is a generalized eigenfunction $f_{\lambda}$ for $H^{P,L}(\omega)$ for $\varrho$-almost every spectral value $\lambda \in \sigma(H^{P,L}(\omega))$ with the property $w^{-1}f_\lambda \in L^2(X_E)$.
\end{folg}
For the definition of a weight function we have to fix a root vertex $0\in V$. This vertex is an arbitrary one, but has to stay fixed in the following.

\begin{bem}{(a weight function)} 
\label{bem:VEF:gewfkt}
Let $\Gamma$ be a metric graph with \geomalles.
Let $m>\frac{d+1}{2}$ and $\we:X_E \to [1,\infty)$ be defined by
\begin{equation*}
\we(x):=(1+\metrik(x,0))^m.
\end{equation*}
Then $\we$ is a weight function in the sense of the last corollary, i.\,e. $\we^{-1}\in L^2(X_E)$.
\end{bem}
\begin{proof} 
We will estimate the $L^2$-norms of $\we^{-1}$ over the annuli with radius $n$. The uniform polynomial growth of the graph yields a uniform growth of the volumes of the balls:
\begin{align*}
\int\limits_{X_E}\| \we^{-1}(x) \|^2 \dx&=\sum\limits_{n=1}^\infty \int\limits_{B_n(0)\setminus B_{n-1}(0)} \we^{-2}(x)\dx 
\leq \sum\limits_{n=1}^\infty \int\limits_{B_n(0)\setminus B_{n-1}(0)} n^{-2m}\dx\\
&\leq\sum\limits_{n=1}^\infty \vol(B_n(0)) n^{-2m}
\leq  \cp \sum\limits_{n=1}^\infty n^{d-2m}<\infty,
\end{align*}
as $m>\frac{d+1}{2}$.
\end{proof}
The weight function and the last corollary yield polynomial growth of the generalized eigenfunctions. We have to choose $m$ small, to get a low growth rate:
\begin{hsatz}
\label{hsatz:VEF:poly}
Let $\Gamma$ be a metric graph with \geomalles\ and $(H^{P,L}(\omega))$ a random operator with \RBPLS\ and \eqref{pot:char,dichte}.
Then for each generalized eigenfunction for $H^{P,L}(\omega)$, for all verteices $v\in V$ and all radii $R>5\okl $ 
\begin{align*}
\left\|\ind_{\Lambda_R(v)} f \right\|\leq C_{\mathrm{poly}\hspace{-1pt}\nearrow} \cdot  R^d\cdot(1+\metrik(v,0) +R+\okl)^{\frac{d+2}{2}}
\end{align*}
holds, where $C_{\mathrm{poly}\hspace{-1pt}\nearrow}$ depends only on $\cp$ and $f$.
\end{hsatz}
\begin{proof}
By remark \ref{bem:VEF:gewfkt} we see, that $w(x)=(1+\metrik(x,0))^\frac{d+2}{2}$ is a weight function, such that $w^{-1} f \in L^2(X_E)$ for all generalized eigenfunctions $f$ for $H^{P,L}(\omega)$.
For all $r>5\okl$ we have
\begin{equation*}
\sup\limits_{v\in V} \left\| \ind_{\Lambda_r(v)} w^{-1}f\right\|\leq C_r < \infty.
\end{equation*}
Otherwise a contradiction to $ w^{-1}f $ lying in $L^2(X_E)$ could be found, using the polynomial growth of the graph.
For a point $x \in X_{E(v_0,R)} $ we find $\metrik(x,0)\leq \metrik(v_0,0)+R+\okl$.
Inserting this point in the weight function we get
\begin{equation*}
w^{-1}(x)\geq \left( 1+\metrik(v_0,0)+R+\okl\right)^{-\frac{d+2}{2}}.
\end{equation*}
The ball $\Lambda_R(v_0) $ can be covered with balls with radius $r$, if we choose a vertex-raster $V_{R,s}(v_0)$, with $s=\frac{r-5\okl}{3}$, by lemma \ref{hsatz_ueberdeckung}. By lemma \ref{hsatz_rasterzahl} this raster has at most $\cp\,\frac{R^d}{s} $ elements. This yields
\begin{align*}
\left( 1+\metrik(v_0,0)+R+\okl\right)^{-\frac{d+2}{2}} \cdot \left\| \ind_{\Lambda_R(v_0)}f \right\|_{L^2(X_E)} &\leq  \left\| \ind_{\Lambda_R(v_0)} f w^{-1}\right\| \\
&\leq  \sum\limits_{v\in V_{R,s}(v_0)} \left\|\ind_{\Lambda_r(v)} f w^{-1}\right\|\\
&\leq \frac{3\,\cp\cdot R^d}{r-5\okl}\, C_r
\end{align*}
and hence the assertion with a constant depending only on $\cp$, $f$ and $r$, where $r$ is a free parameter of no importance.
\end{proof}


\section{Multiscale analysis}
\label{ab:msa}
The aim of the multiscale analysis is to provide exponential or polynomial decay of the form
\begin{equation*}
\left\| \ind_{\Lout_r(v)} \big(H^{\Lambda_r(v)}(\omega)-\lambda\big)^{-1}\ind_{\Lint_r(v)} \right\| \leq e^{-\gamma\cdot r}, \qquad\text{or } \leq r^{-n}
\end{equation*}
for different length scales $r$. The estimate will be proven inductively and won't be true for all $\omega$. The probability of those events has to go to one with the length scales going to infinity.

To apply the MSA to metric graphs we modify the MSA from \cite{Stollmann-01}, mainly by using the subgraphs from section \ref{ab:ITG} and coverings from section \ref{ab:WG}. As the covering is not as precise as in the $\gz^d$-case, we will get polynomial decay only. This will be explained in detail later. We have to adjust definitions and induction parameters. For another MSA with polynomial decay see \cite{FischerML-00}.

\begin{defn}
Let $\Gamma$ be a metric graph with \geomalles\ and $(H^{P,L}(\omega))$ a random operator with \RBPLS\ and \potalles.
Let $n>0$, $r>0$ and $v\in V$. The induced subgraph $\Lambda_r(v) \subset \Gamma $ is called $(n,\lambda,\omega)$-good, if
$\lambda \in \res(H^{\Lambda_r(v)}(\omega))$ and
\begin{equation*}
\left\| \ind_{\Lout_r(v)} \big(H^{\Lambda_r(v)}(\omega)-\lambda\big)^{-1}\ind_{\Lint_r(v)} \right\| \leq r^{-n}.
\end{equation*}
Otherwise $\Lambda_r(v)$ is called $(n,\lambda,\omega)$-bad.
\end{defn}

\begin{defn}
Let $I\subset \rz$ be an interval, $r>0$ a radius and $n,\xi,\theta,q>0$. We define the following logic assertions:
\begin{enumerate}
 \item $G(I,r,n,\xi)$:
	For all $v_1$, $v_2\in V$, with $\Lambda_r(v_1)$ and $\Lambda_r(v_2)$ are disjoint, we have the estimate
	\begin{equation*}
	  \PP \left\{\omega\in \Omega \text{ with } \fa \lambda\in I :\Lambda_r(v_1) \text{ or }\Lambda_r(v_2)  \text{ is }(n,\lambda,\omega)\text{-good}\right\} \geq 1-r^{-2\xi}.
	\end{equation*}
 \item $W(I,r,\theta,n,q)$:
	For all $\lambda\in I$ and all balls $\Lambda_r(v) \subset \Gamma$ we have
	\begin{equation*}
	 \PP \left\{\omega\in \Omega \text{ with } \dist\left(\sigma(H^{\Lambda_r(v)}(\omega)),\{\lambda\}\right) \leq r^{-\theta n} \right\} \leq r^{-q}.
	\end{equation*}
\end{enumerate}
\end{defn}
The aim of the multiscale analysis is to prove $G(I,r,n,\xi)$ for a sequence of radii $(r_k)$ with $r_k\to \infty$, while fixing the other three parameters.

The assertion $W(I,r,\theta,n,q)$ is called weak Wegner estimate and follows from the Wegner estimate. A proof for a metric graph over $\gz^d$ can be found in lemma 13 in \cite{GruberHV-08} and adapted to the general metric graph case:
\begin{bem} 
\label{bem:W}
Let $\Gamma$ be a metric graph with \geomalles\ and $(H^{P,L}(\omega))$ a random operator with \RBPLS\ and \potalles.
Let $\theta$, $q>0$ with $q<\theta n-d$ and $I$ a bounded interval in $\rz$. Here $d$ stands for the degree of polynomial growth of the graph.
Then there is a radius $\rtwo\in(0,\infty)$, such that for all $r\geq \rtwo$ assertion $W(I,r,\theta,n,q)$ holds.
\end{bem} 
\begin{proof} 
With the density of the single site potential we get
\begin{align*}
s(\mu,\ep)
\leq 2\ep \| \dichte_\mu \|_\infty\leq 2\cdot \ep \cdot c_\dichte.
\end{align*}
By the Wegner estimate in theorem \ref{satz_wegner} we get for $\lambda \in I $ and $v\in V$
\begin{flalign*}
&\PP \left\{\omega\in \Omega\with  \dist\left(\sigma(H^{\Lambda_r(v)}(\omega)),\{\lambda\}\right) \leq r^{-\theta n}\right\}&\\
&\hspace{3cm}\leq \mathbb E \left\{ \Tr \left[ \ind_{\left(\lambda-r^{-\theta n},\lambda+r^{-\theta n}\right)}\left(H^{\Lambda_r(v)}(\omega)\right)  \right] \right\}&\\
&\hspace{3cm}\leq C_\mathrm{W}s(\mu,r^{-\theta n}) |\Lambda_r(v)|&\\
&\hspace{3cm}\hspace{-0.55cm}\underset{\eqref{geom:poly}}{\leq} 2 \cdot C_\mathrm{W}\cdot c_\dichte r^{-\theta n} \cdot\frac{ \cp \cdot r^d}{\ukl}   \hspace{2.5cm} \leq  r^{-q} &
\end{flalign*}
for $r\geq \rtwo(C_\mathrm{W},\cp,c_\dichte,\ukl) $.
\end{proof} 

With the initial length scale estimate and the Combes-Thomas estimate we can start the induction by
\begin{satz}[Induction start]
Let $(H^{P,L}(\omega))$ be a Laplacian on $\Gamma$ with \RBPLS, \potalles\ and \geomalles.
Let $\xi \in (0,2\tau-d)$. Then there is a $\beta\in (0,2)$ and a radius $\rthree\geq \rone$, such that for $r\geq \rthree$, $n\in \rz^+$ and $I=\left[\leos,\leos+ \tfrac{1}{2} r^{\beta-2} \right]$ assertion $G(I,r,n,\xi)$ holds.
\end{satz}
Here $\sigma_0$ again denotes the lower bound of all spectra of $(H^{P,L}(\omega))$.
\begin{proof} 
By theorem \ref{satz:ALA} we find a radius $\rone$ and $\beta\in (0,2)$, such that for all $r\geq \rone$ we have
\begin{equation*}
\PP\{ \omega\in \Omega\with\dist(\sigma(H^{\Lambda_r(v)}(\omega)),\{\leos\})\leq r^{\beta-2} \}\leq r^{-\xi}.
\end{equation*}
For all $\omega$ with $\dist(\sigma(H^{\Lambda_r(v)}(\omega)),\{\leos\}) > r^{\beta-2}$ we can use the Combes-Thomas estimate from Corollary \ref{folg:CTA} with the constants $\lambda\in I $, $s:=\leos-1$, $t:=\leos+r^{\beta-2}$ and $\eta \geq \tfrac{1}{2} r^{\beta-2} $, yielding
\begin{flalign*}
&\left\| \ind_{\Lint_r(v)} \big(H^{\Lambda_r(v)}(\omega)-\lambda\big)^{-1}\ind_{\Lout_r(v)} \right\| &\\
& \hspace{2cm} \leq C_\mathrm{CTA}\cdot\eta^{-1}\cdot \exp\left( -\tilde{C}\sqrt{\eta(s-r)}\dist(\Lint_r(v),\Lout_r(v)) \right)&\\
&\hspace{1.98cm} \underset{\eqref{gl_dist_Lint_Lout}}{\leq}
C_\mathrm{CTA}\cdot 2 r^{2-\beta} \cdot\exp\left( -\tilde{C}\cdot\sqrt{ \frac{r^{\beta-2}}{2}(r^{\beta-2}+1) } \cdot  \frac{r}{2}\right)&\\
&\hspace{2cm} \leq 2\,C_\mathrm{CTA}\cdot  r^{2-\beta}\cdot\exp\left( -\tilde{C} \cdot r^{\frac{\beta}{2}} \right)&\\
& \hspace{2cm} \leq r^{-n}
\end{flalign*}
for $r \geq \rthree(C_\mathrm{CTA},\beta,n,\rone)$. Thus all balls $\Lambda_r(v)$ for the above mentioned configurations $\omega$ are  $(n,\lambda,\omega)$-good.
Hence the probability that two independent balls (i.\,e. $E_{\Lambda_r(v_1)}\cap E_{\Lambda_r(v_2)}=\varnothing $) are $(n,\lambda,\omega)$-bad is less or equal to $\left(r^{-\xi}\right)^2$.
\end{proof}
Thereby we proved $G(I,r,n,\xi)$ for fixed $n$ and $\xi$ and all radii $r \geq \rthree$. But the interval depends on $r$ and its length goes to zero for $r \to \infty$.
Now we need to prove an induction step from $G(I,r,n,\xi)$ to a greater length scale, but keeping the interval $I$.
\begin{satz}[Induction] 
\label{satz:indsatz}
Let $(H^{P,L}(\omega))$ be a Laplacian on $\Gamma$ with \geomalles, \RBPLS\ and \potalles.
Let the induction parameters be given by
\begin{equation}
\label{IP}\tag{IP}
\left.
\begin{aligned}
&q\in \left(7d-6,7d\right), & \\
& \tau > \frac{3d}{2}-1,&\\
&\xi\in\left(2d-2,\min\left\{ 2\tau-d,\frac{q-3d+2}{2}  \right\}\right),&\\
&\alpha\in\left(  1,\min \left\{ \frac{2+2\xi}{2d+\xi},\frac{2+q}{3d+2\xi} \right\} \right) ,&\\
&\theta\in\left(\frac{q+d}{n},\frac{n+2-d-\alpha d}{\alpha n}\right),&\\
&n>9\alpha d+d-2. &  & 
\end{aligned}
\right\}
\end{equation}
Then there is a radius $\rnine$, such that if $G(I,r,n,\xi)$ holds with $r \geq \rnine$ for an open, bounded interval $I\subset \rz$  also $G(I,R,n,\xi)$ holds for $R=r^\alpha$.
\end{satz}

The induction start can be used to start the induction with the interval $I=[\leos,\leos+\frac{1}{2}{\rnine}^{\beta -2}]$.
But the induction step works also for any other interval satisfying an induction start.
In the proof we will need a slightly larger Interval $I_0:=I+\left(-\frac{1}{2},\frac{1}{2}\right)$ for the use of Wegner estimate and Weyl asymptotic.

Some choices of the induction parameters will be made clear in the proof. Most of them guarantee, that others can be chosen, i.\,e. the intervals for other parameters are nonempty.
To demonstrate this and show the dependencies we gathered all parameters in the theorem.
For further details see the appendix \ref{app}.\medskip

The proof will be divided in four steps.
The first three will provide the event in $\Omega$ with probability of at least $1-R^{-2\xi}$
for the assertion $G(I,R,n,\xi)$.
The last step will show how to find a $(n,\lambda,\omega)$-good ball $\Lambda_R(v)$.

The remainder of this section is devoted to a proof of the theorem:
\medskip

First we will define a (good) event in $\Omega$, such that for a vertex $v\in V$, there are maximal three disjoint $(n,\lambda,\omega)$-bad balls, with center in the vertex-raster $\raster(v)\subset V$ and radius $r$, in $\Lambda_R(v)$.
\begin{align*}
\Omega_{\mathrm{G}}(v)&:=\{ \omega \in \Omega \with \forall\  \lambda\in I \ \nexists \  4 \text{ disjoint balls }\Lambda_r(b_i)\subset\Lambda_R(v)\\
&\hspace{0.95cm} \text{ with } b_i\in\raster(v),i=1,\ldots,4, \text{ which are all }(n,\lambda,\omega)\text{-bad}\}.
\end{align*}

\begin{schritt}
\label{schritt1}
With the assumptions of the induction theorem we conclude
\begin{enumerate}
 \item There exists a radius $\rfour=\rfour(d,\alpha)$, such that for $r\geq \rfour$ we have:
 	\begin{equation*}
 	 \PP (\Omega_{\mathrm{G}}(v))\geq 1-\tfrac{1}{3}R^{-2\xi} \qquad \text{for all }v\in V.
 	\end{equation*}
  \item For $\omega \in \Omega_{\mathrm{G}}(v)$, $\lambda\in I$, $x\in V$, $r$ big enough, there exist three induced subgraphs $\Cont_i$ $(i=1,2,3)$, with $\Cont_i:=\Lambda_{r_i}(v_i)$ or $\Cont_i=\Gamma_\varnothing$, which will be called container sets, with the following properties
 	\begin{enumerate}
  		\item $r_i\in \mathcal R=\left\{ 3r+2\okl,\tfrac{63}{10}r+11\okl,\tfrac{48}{5}r+\tfrac{41}{2}\okl\right\}$, $\sum\limits_{i=1}^3 r_i\leq \tfrac{48}{5}r+\tfrac{41}{2}\okl$, $v_i\in \raster(x)$.
 		\item If $\Lambda_r(b)\subset \Lambda_R(x)$ with $b\in \raster(x)$ is $(n,\lambda,\omega)$-bad, then
 			\begin{equation*}
 		 		\Lambda_r(b)\subset \bigcup_{i=1}^3 \Cont_i.
 			\end{equation*}
 		\item Each two containers are disjoint.
 	\end{enumerate}
\end{enumerate}
\end{schritt}

\begin{bem} The statements of step 1 yield the following
\begin{itemize}
 \item
For each tuple of an energy $\lambda \in I$ and an $\omega \in \Omega_{\mathrm{G}}$ we find $(n,\lambda,\omega)$-bad balls with radius $r$ and centers in the raster, such that there are maximally three disjoint ones. With 2.b) they all lie in one of the containers.
\item 
In the following we have to estimate the number of all possible containers. To fix the centers on the raster and the number of possible radii reduces the number largely. The different radii are necessary to satisfy 2.c) and their maximal sum will be important in step \ref{schritt4} of the iteration.
\item 
Covering the boundary of the container as in proposition \ref{hsatz_containnerrand} with balls, which don't lie completely in the container, we obtain with 2.c) that these balls are $(n,\lambda,\omega)$-good.
\end{itemize}
\end{bem}

\begin{proof}[Proof of step 1]
\begin{enumerate}
\item
Let $\omega \not\in \Omega_\mathrm{G}$. Then there exists a $\lambda\in I$, such that there are four disjoint $(n,\lambda,\omega)$-bad balls with radius $r$ and centers in $\raster(x)$. 

By lemma \ref{hsatz_rasterzahl} we know that there are maximal $\cp \frac{10\,R^d}{r} $ centers in
$\raster(x)$ and thus at most $\left(\frac{10\, \cp}{r}\right)^4 R^{4d}$ quadruple of disjoint balls with radius $r$ and center in $\raster(x)$ in $\Lambda_R(x)$. By $G(I,r,n,\xi)$ we know the probability of one pair of $(n,\lambda,\omega)$-bad balls is bounded by $r^{-2\xi}$. Hence we get

\begin{equation*}
\PP(\Omega_{\mathrm{G}}^C)
\leq \left(\frac{10\,\cp}{r}\right)^4 R^{4d} r^{-4\xi}.
\end{equation*}
Multiplied with $ R^{2\xi}=r^{2\alpha \xi}$ we get as exponent of $r$:
\begin{equation*}
-4+4\alpha d-4\xi+2\alpha\xi<0
\Leftrightarrow \alpha <\frac{2+2\xi}{2d+\xi},
\end{equation*}
which is satisfied due to \eqref{IP}.
Thus we find $\PP(\Omega_{\mathrm{G}}^C)\leq \frac{1}{3}R^{-2\xi}$ for $r\geq \rfour(\cp,d,\alpha,\xi)$, where the parameters $\alpha$ and $\xi$ depend only on $d$ and $\tau$.

\item 
Picking $\omega \in \Omega_\mathrm{G}$ there are maximally three disjoint $(n,\lambda,\omega)$-bad balls with centers in the vertex-raster $\raster(x)$.
Let $\Lambda_r(b_i)$, $b_i \in \raster(x)$, $i=1,2,3$ be three of such balls (if there are less, the proof gets easier).

Then all $(n,\lambda,\omega)$-bad balls with center in the raster $\raster(x)$, which are not disjoint to $\Lambda_r(b_i)$, are subsets of
$\Lambda_{3r+2\okl}(b_i)$:
For all those $(n,\lambda,\omega)$-bad balls $\Lambda_r(b)$ there is an edge $e_b$ with $e_b\in E(b,r)$ and $e_b\in E(b_i,r)$.
The distance of any point $y\in\Lambda_r(b)$ to $b_i$ is less or equal to $\metrik(y,b)+\metrik(b,b_i)$, where $\metrik(y,b)<r+\okl$ and $\metrik(b,b_i)\leq r+r+\okl$, as $e_b$ begins with maximal distance $r$ from $b_i$ and $b$ and has length smaller or equal to $\okl$.

If two of the sets $\Lambda_{3r+2\okl}(b_i)$ are not disjoint, we have to join them to one bigger neighborhood.
Without loss of generality let $\Lambda_{3r+2\okl}(b_1)$ and $\Lambda_{3r+2\okl}(b_2)$ be not disjoint. Then there is an edge
$e_b$, contained in both graphs and a point $\tilde{x}$ in $I_{e_b}$.
With lemma \ref{hsatz_ueberdeckung} we can find a point $\tilde{v}\in \raster(x)$, such that
$\tilde{x}\in \Lambda_{3\cdot \tfrac{r}{10}+5\okl}(\tilde{v})$. 
For all points $y$ lying in the container $\Lambda_{3r+2\okl}(b_i)$, ($i=1,2$) we have
\begin{align*}
 \metrik(y,\tilde{v})&\leq \metrik(y,b_i)+\metrik(b_i,\tilde{v})\\
&< (3r+2\okl+\okl)+\left(3r+2\okl+\frac{3r}{10}+5\okl+\okl\right)\\
&\leq \frac{63}{10}r+11\okl.
\end{align*}
Thus the new, bigger container is $\Lambda_{\frac{63}{10}r+11\okl}(\tilde{v})$.

Now it might happen, that $\Lambda_{\frac{63}{10}r+11\okl}(\tilde{v})$ and $\Lambda_{3r+2\okl}(b_3)$ are not disjoint.
With the arguments shown above we can find a vertex $\overline{v}$ in the raster $\raster(x)$, such that both containers are contained in 
$\Lambda_{\frac{48r}{5}+\frac{41\okl}{2}}(\overline{v})$. Here the radius is the minimal radius with the claimed properties.
(See \cite{diss_11} for more details.)

Appropriately rearranging we find one of the following configurations for containers for the three bad balls, which satisfy all stated properties:
\begin{enumerate}[label=(\roman*)]
\item $\Cont_i=\Lambda_{3r+2\okl}(b_i)$, $i=1,2,3$
\item $\Cont_1=\Lambda_{\frac{63}{10}r+11\okl}(\tilde{v})$, $\Cont_2=\Lambda_{3r+2\okl}(b_3)$, $\Cont_3=\Gamma_\varnothing$
\item $\Cont_1=\Lambda_{\frac{48r}{5}+\frac{41\okl}{2}}(\overline{v}) $, $\Cont_2=\Cont_3=\Gamma_\varnothing$.
\end{enumerate}
\end{enumerate}\vspace{0.1cm}
This ends the proof of step 1.
\end{proof}
The concrete setting of containers depends on the parameters $\omega \in \Omega$ and $\lambda \in I$. Thus for each tuple $(\omega,\lambda)$ there are three sets $\Cont_i(\omega,\lambda)$.

\begin{defn} 
Let $\Lambda_s(z)$ be a ball with $s>0$. We define
\begin{equation*}
\sigma_1\left(H^{\Lambda_s(z)}(\omega)\right):=\sigma\left(H^{\Lambda_s(z)}(\omega)\right)\cap (I+(-\tfrac{1}{2}s^{-\theta n},\tfrac{1}{2}s^{-\theta n})),
\end{equation*}
where $A+B$ stands for the Minkowsi sum.
\end{defn} 
From the definition $I_0:=I+\left(-\frac{1}{2},\frac{1}{2}\right)$ we automatically get for $s\geq 1$
\begin{equation*}
\left( I+(-\tfrac{1}{2}s^{-\theta n},\tfrac{1}{2}s^{-\theta n}) \right)\subset I_0.
\end{equation*}
\begin{schritt} 
\label{schritt2}
Let $\Lambda_{r_1}(v_1)$ and $ \Lambda_{r_2}(v_2)$ with $r_1$, $r_2\geq \rtwo$ be two disjoint and thus independent balls. We have
\begin{align*}
&\PP \left\{\omega\in \Omega \with \dist\left(\sigma_1\left(H^{\Lambda_{r_1}(v_1)}(\omega)\right), \sigma_1\left(H^{\Lambda_{r_2}(v_2)}(\omega)\right) \right) \leq  {\min\{r_1,r_2\}}^{-\theta n} \right\}\\
&\hspace{5cm}\leq C_\mathrm{Weyl}\cdot\frac{{\max\{r_1,r_2 \}}^d }{{\min\{r_1,r_2\}}^q}.
\end{align*}
\end{schritt} 

\begin{proof} 
Without loss of generality let $r_1=\min\{r_1,r_2\}$.
Let $E_1$, $E_2\subset E$ be finite, disjoint subsets and $\Omega_0\subset \Omega$ an event with
\begin{equation*}
\prod\limits_{e\in E_1\cup E_2} q_e(\Omega_0)\times {[q_-,q_+]}^{E\setminus (E_1\cup E_2)}=\Omega_0,
\end{equation*}
thus being a cylinder set, depending only on the edges in $E_1$ and $E_2$. We have
\begin{align}
 \label{gl:schritt2_1}
\PP(\Omega_0)&=
\mathbb E_{E\setminus E_1}\left(\PP_{E_1}\left(\Omega_0\right) \right)
=\mathbb E_{E_2}\left( \PP_{E_1}(\Omega_0) \right).
\end{align}
Further we get
\begin{align*}
& \PP_{E(v_1,r_1)}  \left\{\omega\in\Omega : \dist\left(\sigma_1\left(H^{\Lambda_{r_1}(v_1)}(\omega)\right),\sigma_1\left(H^{\Lambda_{r_2}(v_2)}(\omega)\right)\right)\leq {r_1}^{-\theta n}\right\}\\
&\hspace{1cm}=\PP_{E(v_1,r_1)}  \Bigg\{\omega\in\Omega : \min\limits_{\lambda\in\sigma_1\left(H^{\Lambda_{r_2}(v_2)}(\omega)\right)} \dist\left(\sigma_1\left(H^{\Lambda_{r_1}(v_1)}(\omega)\right),\{\lambda\}\right)\leq {r_1}^{-\theta n}\Bigg\}\\
&\hspace{1cm}=\PP_{E(v_1,r_1)}\text{\large{$\Bigg\{$}} \bigcup\limits_{\lambda\in \sigma_1\left(H^{\Lambda_{r_2}(v_2)}(\omega)\right)}  \left\{ \omega\in\Omega : \dist\left(\sigma_1\left(H^{\Lambda_{r_1}(v_1)}(\omega)\right),\{\lambda\}\right)  \leq {r_1}^{-\theta n} \right\} \text{\large{$\Bigg\}$}} \\
&\hspace{1cm}\leq \sum\limits_{\lambda\in \sigma_1\left(H^{\Lambda_{r_2}(v_2)}(\omega)\right)}  \PP_{E(v_1,r_1)} 
\left\{   \omega\in\Omega : \dist\left(\sigma_1(H^{\Lambda_{r_1}(v_1)}(\omega)\right),\{\lambda\})  \leq {r_1}^{-\theta n}\right\}.
\intertext{
Replacing $\sigma_1\left( H^{\Lambda_{r_1}(v_1)}\right)$ with $\sigma\left(H^{\Lambda_{r_1}(v_1)}\right)$ the probabilities only can get larger. They can be estimated using the Wegner estimate $W(I,r,\theta,n,q)$ from remark \ref{bem:W} by ${r_1}^{-q}$.
The Weyl asymptotic in theorem \ref{folg:weyl} yields a bound on the number of summands, being the number of eigenvalues  of $H^{\Lambda_{r_2}(v_2)}(\omega)$ in $I_0$. Altogether we obtain:
}
&\hspace{0.45cm}\ldots\leq \left|\sigma_1\left(H^{\Lambda_{r_2}(v_2)}(\omega)\right)\right|\cdot  {r_1}^{-q}
\leq C_\mathrm{Weyl}\cdot \frac{ {r_2}^d}{{r_1}^{q}}.
\end{align*}
In total this gives
\begin{align*}
&\PP  \bigg\{
\underbrace{
\omega\in\Omega : \dist\left(\sigma_1\left(H^{\Lambda_{r_1}(v_1)}(\omega)\right),\sigma_1\left(H^{\Lambda_{r_2}(v_2)}(\omega)\right)\right)\leq {r_1}^{-\theta n}}_{=:\Omega_0}
\bigg\}\\
&\hspace{1cm}\underset{\eqref{gl:schritt2_1}}{=}\mathbb E_{E(v_2,r_2)}\left( 
\PP_{E(v_1,r_1)}(\Omega_0)
\right)
\leq \mathbb E_{E(v_2,r_2)} \left( C_\mathrm{Weyl}\cdot\frac{ {r_2}^d}{{r_1}^{q}} \right) =C_\mathrm{Weyl}\cdot\frac{ {r_2}^d}{ {r_1}^{q}}.
\end{align*}
This ends the proof of step 2.
\end{proof}

\begin{schritt}
\label{schritt3}
There exists a radius $\rfive$, such that for $r\geq \rfive$ we have: For $x$, $y\in V$ with $\Lambda_R(x)$ and $\Lambda_R(y)$
are disjoint
\begin{align*}
\PP(\Omega_\mathrm{W})&\leq \frac{1}{3} R^{-2\xi}
\end{align*}
holds. Here
\begin{align*}
\Omega_\mathrm{W}&=\left\{ \omega\in \Omega \with \esgibt \Lambda_{r_i}(v_i), i=1,2 \text{ with } \right.\\
&\phantom{=\ \ \ }\Lambda_{r_1}(v_1)=\Lambda_R(x) \text{ or } v_1\in \raster(x) \text{ and }r_1\in \mathcal R,\\
&\phantom{=\ \ \ }\Lambda_{r_2}(v_2)=\Lambda_R(y) \text{ or } v_2\in \raster(y) \text{ and }r_2\in \mathcal R,\\
&\phantom{=\ \ \ }\left.
\dist\left(\sigma_1\left(H^{\Lambda_{r_1}(v_1)}(\omega)\right),\sigma_1\left(H^{\Lambda_{r_2}(v_2)}(\omega)\right)\right) \leq \min\{r_1,r_2\}^{-\theta n} \right\}.
\end{align*}
\end{schritt}

\begin{proof}
We will estimate the number of all possible pairs of balls $\Lambda_{r_i}(v_i)$.
The number of centers is bounded by the number of vertices in the rasters $\raster(x)$ and $\raster(y)$. Both are bounded by $10\,\cp R^d/r$ 
by lemma \ref{hsatz_rasterzahl}.
As there are maximal four different radii, we get with step \ref{schritt2} for all $r\geq \rtwo$
\begin{align*}
\PP(\Omega_\mathrm{W}) &\leq \left( 4\cdot 10 \cp \frac{R^d}{r}  \right)^2\cdot C_\mathrm{Weyl} \cdot \frac{(\max\{r_1,r_2\})^d }{(\min\{r_1,r_2\})^q }\qquad (r<r_1,r_2\leq R)\\
&\leq C_\mathrm{Weyl} \left( 40 \cp  \right)^2 \cdot\frac{R^{3d}}{r^{2+q}}
=C_\mathrm{Weyl} \left( 40\cp \right)^2 r^{3d\alpha-2-q}.
\end{align*}
Multiplied with $R^{2\xi}=r^{2\alpha\xi}$ we get $3d\alpha -2-q+2\alpha\xi $ as exponent of $r$. This is smaller than zero, if 
\begin{equation*}
\alpha < \frac{2+q}{3d+2\xi}
\end{equation*}
which is satisfied by \eqref{IP}.
Hence there is a radius $\rfive(C_\mathrm{Weyl},\cp,d,q,\alpha,r_2)$, such that the assertion is true for all $r\geq \rfive$.
This ends the proof of step 3.
\end{proof}

\begin{schritt}
\label{schritt4}
Let $x$, $y\in V$, such that $\Lambda_R(x)$ and $\Lambda_R(y)$ are disjoint, $\Omega_\mathrm{G}$ and $\Omega_\mathrm{W}$ as in step \ref{schritt1} and \ref{schritt3}.
Let $\omega\in \Omega_\mathrm{G}(x)\cap \Omega_\mathrm{G}(y)\cap \Omega_\mathrm{W}^c$ and $\lambda\in I$. 
Then there exists a vertx $z\in\{x,y\}$ and a radius $\reight$, such that $\Lambda_R(z) $ is $(n,\lambda,\omega)$-good for $r\geq \reight$.
\end{schritt}

\begin{defn}
Let $\lambda \in I$. We call an induced subgraph $\Lambda_r(v)$ $\lambda$-resonant to $\omega\in \Omega$, if
\begin{equation*}
\dist\left(\sigma(H^{\Lambda_r(v)}(\omega)),\{\lambda\}\right)\leq \tfrac{1}{2}r^{-\theta n},
\end{equation*}
otherwise we call it $\lambda$-dissonant to $\omega$.
\end{defn}

\begin{proof}
Let $r_1$ be the maximal radius of all balls $\Lambda_s(v)$ with $s\in \mathcal R$ and $v\in\raster(x)\cup\raster(y)$ or $s=R$ and $v\in \{x,y\}$ being $\lambda$-resonant to $\omega\in \Omega_\mathrm{W}^c$.
(If there is no such ball, then they all are $\lambda$-dissonant.) 
Without loss of generality let $\Lambda_{r_1}(v_1)$ be such a ball with $v_1\in \raster(y)$ or $v_1=y$. 
Then all balls $\Lambda_{r_2}(v_2)$ with $v_2\in \raster(x) $ and $r_2\in \mathcal R $ and also $\Lambda_R(x)$ are $\lambda$-dissonant:

We assume the contrary: Without loss of generality let $\Lambda_{r_2}(v_2)$ be $\lambda$-resonant.
Then 
\begin{flalign*}
&\dist\left(\sigma_1\left(H^{\Lambda_{r_1}(v_1)}(\omega)\right), \sigma_1\left(H^{\Lambda_{r_2}(v_2)}(\omega)\right)\right) &\\
&\hspace{1cm}\leq \dist\left(\sigma_1\left(H^{\Lambda_{r_1}(v_1)}(\omega)\right),\{\lambda\}\right)+\dist\left(\sigma_1\left(H^{\Lambda_{r_2}(v_2)}(\omega)\right),\{\lambda\}\right)&\\
&\hspace{1cm}\leq \frac{1}{2}r_1^{-\theta n}+\frac{1}{2}r_2^{-\theta n}  \qquad \left(\text{as } \Lambda_{r_1}(v_1),\Lambda_{r_2}(v_2) \text{ are } \lambda \ \text{-resonant}\right)&\\
&\hspace{1cm}\leq \min\{r_1,r_2 \}^{-\theta n},
\end{flalign*}
yielding $\omega\in \Omega_\mathrm{W}$, which is a contradiction.
Hence all container in $\Lambda_R(x)$ as described in step \ref{schritt1} and $\Lambda_R(x) $ itself are $\lambda$-dissonant.

Now we will show that $\Lambda_R(x)$ is $(n,\lambda,\omega)$-good.
This will be done using a covering of $\Lint_R(x)$ with balls $\Lint_r(v)$ and an iteration of the covering which is the essential part of the multiscale analysis.

Let $x_0$ be a center needed to cover $\Lint_R(x)$, i.\,e. $x_0\in \raster(x)\cap \Lambda_{\frac{R+r}{3}}(x)$. If we have chosen the centers $x_0,\ldots,x_m$ $\in\raster(x)$ with $m\geq 0$ we will do a case study depending on $\omega$ how to choose the next center $x_{m+1}$. See also the illustration in figure \ref{abb:msa}.
\begin{itemize} 
 \item[(+):] {\it $\Lambda_r(x_m)$ is $(n,\lambda,\omega)$-good}:\\
Then by the geometric resolvent inequality in corollary \ref{folg:gru} we have
\begin{flalign}
\label{gl_+_1}
\notag &\left\| \ind_{\Lout_R(x)}\left(H^{\Lambda_R(x)}(\omega)-\lambda \right)^{-1} \ind_{\Lint_r(x_m)} \right\|&\\
\notag  & \hspace{1cm}\leq C_\mathrm{GRU}\left\| \ind_{\Lout_R(x)}\left(H^{\Lambda_R(x)}(\omega)-\lambda \right)^{-1} \ind_{\Lout_r(x_m)} \right\|&\\
\notag & \hspace{2.5cm} \cdot \left\| \ind_{\Lout_r(x_m)}\left(H^{\Lambda_r(x_m)}(\omega)-\lambda \right)^{-1} \ind_{\Lint_r(x_m)}\right\|&\\
&\hspace{1cm}\leq C_\mathrm{GRU} \left\| \ind_{\Lout_R(x)}\left(H^{\Lambda_R(x)}(\omega)-\lambda \right)^{-1} \ind_{\Lout_r(x_m)} \right\| \cdot r^{-n}.&
\end{flalign}
Now we cover $\Lout_r(x_m)$ by balls with centers in $\raster(x)$ and radius $\tfrac{r}{3}$, meaning with the interiors of balls with radius $r$. By lemma \ref{hsatz_ueberdeckung} and remark \ref{bem_ueberdeckung} this is done in the following way:
\begin{equation*}
\Lout_r(x_m) \subset \bigcup\limits_{v \in \raster(x)\cap V_{E\left(x_m,\frac{4}{3}r\right)}} \Lint_r(v).
\end{equation*}
Lemma \ref{hsatz_rasterzahl} yields a bound on the number of sets/centers needed for the covering
\big(the centers all lie in $E\left(x_m,\frac{4r}{3}\right)$; thus the associated disjoint sets with radius $\frac{r}{10}$ all lie in $E\left(x_m,\frac{4r}{3}+\frac{r}{10}+\okl\right)$\big):
\begin{align*}
\left|\raster(x)\cap V_{E\left(x_m,\frac{4r}{3}\right)}\right| &\leq \left|V_{\frac{4r}{3}+\frac{r}{10}+\okl,\frac{r}{10}}(x_m)\right|
\leq \cp \frac{10\left(\tfrac{43}{30}r+\okl\right)^d}{r}\\ &\hspace{-0.17cm}\underset{r>\rgeom}{<} \cp\cdot \left(\tfrac{3}{2}\right)^d\cdot r^{d-1}.
\end{align*}
With relation \eqref{gl_+_1} we conclude
\begin{flalign}
\label{gl_MSA_+}
\notag &\left\| \ind_{\Lout_R(x)}\left(H^{\Lambda_R(x)}(\omega)-\lambda \right)^{-1} \ind_{\Lint_r(x_m)} \right\|&\\
\notag &\hspace{1cm}\leq C_\mathrm{GRU} \, r^{-n}\cdot \sum\limits_{v\in\raster(x)\cap V_{E\left(x_m,\frac{4}{3}r\right)}}
\left\| \ind_{\Lout_R(x)}\left(H^{\Lambda_R(x)}(\omega)\nm{-1pt}\lambda \right)^{-1} \ind_{\Lint_r(v)} \right\|&\\
&\hspace{1cm}\leq\underbrace{ C_\mathrm{GRU} \cdot r^{-n-1+d} \cdot \cp \cdot \left(\tfrac{3}{2}\right)^d}_{=:\delta_+} \cdot
\left\| \ind_{\Lout_R(x)}\left(H^{\Lambda_R(x)}(\omega)-\lambda \right)^{-1} \ind_{\Lint_r(x_{m+1})} \right\|,&
\end{flalign}
where $x_{m+1}$ is a vertex from $\raster(x)\cap V_{E\left(x_m,\frac{4}{3}r\right)}$, for which the above norm is maximal.
The distance $\metrik(x_m,x_{m+1})$ is less or equal to $\tfrac{4r}{3}$.
\item[($-$):] {\it $\Lambda_r(x_m)$ is $(n,\lambda,\omega)$-bad}:

Then by step \ref{schritt2} we find a container $\Cont_i=\Lambda_{r_i}(v_i)$ ($i\in \{1,2,3\}$) with $r_i\in \mathcal R$ and $v_i\in \raster(x)$, such that $\Lambda_r(x_m)\subset \Cont_i$.
By \eqref{gl_dist_Lint_Lout} we have
\begin{equation*}
\dist\left(\Cont_i^\mathrm{out},\Lint_r(x_m)\right)\geq \dist\left(\Lout_r(x_m),\Lint_r(x_m)\right)\geq \frac{r}{2}.
\end{equation*}
As all containers in $\Lambda_R(x)$ are $\lambda$-dissonant we know that $(s,t)$ is a gap in the spectrum of $H^{\Cont_i} $, where 
$s:=\lambda-\tfrac{1}{2}{r_i}^{-\theta n}$ and $t:=\lambda+\tfrac{1}{2}{r_i}^{-\theta n}$.
With $\eta=\frac{1}{2}{r_i}^{-\theta n}$ we conclude from the Combes-Thomas estiamte (theorem \ref{satz_CTA}):
\begin{flalign}
\label{gl_it_-1}
\notag\left\| \ind_{\Cont_i^{\mathrm{out}}} \left(H^{\Cont_i}(\omega)-\lambda \right)^{-1}\ind_{\Lint_r(x_m)} \right\|
&\leq C_\mathrm{CTA}\cdot 2\cdot {r_i}^{\theta n}\cdot
\exp\left(-\frac{\tilde{C}}{\sqrt{2}}\cdot {r_i}^{-\theta n}\cdot \frac{r}{2} \right)&\\
& \leq  2 C_\mathrm{CTA} \cdot{r_i}^{\theta n}.
\end{flalign}

If $\Cont_i^\mathrm{out}$ and $\Lout_R(x)$ are disjoint, we get from the corollary of the geometric resolvent inequality \ref{folg:gru}
\begin{flalign}
\label{gl_MSA_-1}
\notag &\left\|\ind_{\Lout_R(x)} \left(H^{\Lambda_R(x)}(\omega)-\lambda\right)^{-1} \ind_{\Lint_r(x_m)}  \right\|
&\\
\notag &\hspace{0.75cm}
\leq C_\mathrm{GRU} \left\| \ind_{\Lout_R(x)} \left(H^{\Lambda_R(x)}(\omega)-\lambda\right)^{-1}\ind_{\Cont_i^\mathrm{out}}  \right\|
\cdot \left\| \ind_{\Cont_i^\mathrm{out}} \left(H^{\Cont_i}(\omega)-\lambda\right)^{-1}\ind_{\Lint_r(x_m)}  \right\|&\\
&\hspace{0.715cm}\underset{\eqref{gl_it_-1}}{\leq} 2\,C_\mathrm{GRU} \, C_\mathrm{CTA} \, {r_i}^{\theta n}\cdot
\left\| \ind_{\Lout_R(x)} \left(H^{\Lambda_R(x)}(\omega)-\lambda\right)^{-1}\ind_{\Cont_i^\mathrm{out}}  \right\|.&
\end{flalign}
Now, using proposition \ref{hsatz_containnerrand}, we can cover $\Cont_i^\mathrm{out}$ with balls of radius $\tfrac{r}{3}$ and centers in $\raster(x)$, such that these balls don't lie completely in any container, i.\,e. they are $(n,\lambda,\omega)$-good
\begin{equation*}
\exists\  W \subset \raster(x) \text{ with } \bigcup\limits_{w \in W} \Lint_r(w) \supset \Cont_i^\mathrm{out}  \text{ and } \Lint_r(w)\not\subset\Cont_i \text{ for }w\in W.
\end{equation*}

All the disjoint sets with radius $\frac{r}{10}$ of the raster are contained in a ball with radius $r_i+\frac{r}{3}+\frac{r}{10}+\okl$ and with the center of the container.
By lemma \ref{hsatz_rasterzahl} we can bound the number of elements in $W$ needed to cover the container
\begin{equation*}
| W|\leq \cp \frac{10(r_i+\tfrac{r}{3}+\frac{r}{10}+\okl)^d}{r},
\end{equation*}
yielding
\begin{flalign}
\label{gl_MSA_-2}
&
\begin{aligned}
&\left\| \ind_{\Lout_R(x)} \left(H^{\Lambda_R(x)}(\omega)-\lambda\right)^{-1}\ind_{\Cont_i^\mathrm{out}}  \right\|&\\
&\hspace{1cm}\leq 
\cp \frac{10(r_i+\tfrac{r}{3}+\frac{r}{10}+\okl)^d}{r} \left\| \ind_{\Lout_R(x)}\left(H^{\Lambda_R(x)}(\omega)-\lambda\right)^{-1} \ind_{\Lint_r(\widehat{x})}\right\|,&
\end{aligned} &
\end{flalign}
where $\widehat{x}$ is an element of $W$ maximizing the last norm.
As $\Lint_r(\widehat{x})$ is by construction $(n,\lambda,\omega)$-good we can proceed by a step (+).
Overall, with the relations\eqref{gl_MSA_+}, \eqref{gl_MSA_-1}  and \eqref{gl_MSA_-2}, we have
\begin{flalign}
\label{gl:($-$)}
&\begin{aligned}
&\left\|\ind_{\Lout_R(x)} \left(H^{\Lambda_R(x)}(\omega)-\lambda\right)^{-1} \ind_{\Lint_r(x_m)}  \right\|&\\
&\hspace{1cm}\leq  \underbrace{20\left(\tfrac{3}{2}\right)^dC_\mathrm{CTA} \left(C_\mathrm{GRU}\cdot \cp\right)^2 (r_i+\tfrac{13}{30}r+\okl)^d \cdot r^{-n-2+d}\cdot{r_i}^{\theta n}}_{=:\delta_-}&\\
&\hspace{2.5cm} \cdot \left\| \ind_{\Lout_R(x)}\left(H^{\Lambda_R(x)}(\omega)-\lambda \right)^{-1} \ind_{\Lint_r(x_{m+1})} \right\|.&
\end{aligned}&
\end{flalign}
For the centers we know
\begin{align*}
\metrik(x_m,v_i) &\leq r_i+\okl-r \\
\metrik(\widehat{x},v_i)&\leq \tfrac{r}{3}+r_i+\okl
\intertext{with step (+) resulting in}
\metrik(x_m,x_{m+1}) &\leq 2r_i +\tfrac{2}{3}r+2\okl.
\end{align*}
This ends the case study.
\end{itemize}

The prefactors $\delta_+$ and $\delta_-$ can be estimated in the following way:
\begin{equation*}
\delta_+ = C_\mathrm{GRU}\, \cp \left(\tfrac{3}{2}\right)^d\cdot r^{-n-1+d}
\end{equation*}
gets arbitrarily small for $r\geq \rsix$, where $\rsix(C_{\mathrm{GRU}},\cp,d,n)$.
With $r_i<R$, respectively $r_i+\frac{13}{30}r+\okl < R$ \big( which is satisfied for $r\geq 11^{\frac{1}{\alpha-1}}$\big) we get:
\begin{align*}
\delta_- &\leq 20(\tfrac{3}{2})^d \,C_\mathrm{CTA} \left(C_\mathrm{GRU}\,\cp\right)^2 \cdot
r^{-n-2+d+\alpha\theta n+\alpha d}
\end{align*}
To guaranty polynomial decay of the prefactor, the exponent needs to be smaller than zero.
In terms of $\theta$ this means $\theta< \tfrac{n+2-d-\alpha d}{\alpha n}$, which is satisfied by \eqref{IP}.
Thus there exists a radius $\rseven$, depending on $C_\mathrm{GRU}$, $C_\mathrm{CTA}$, $\cp$, $d$, $\alpha$ and $n$, such that
\begin{equation*}
\delta_- \leq \frac{1}{2} \qquad \text{for all } r\geq \rseven.
\end{equation*}

\begin{figure}[htb]
\centering
\begin{tikzpicture}[scale=0.8]
\draw[thick,pattern=north west lines] (190:6) arc(190:270:6);   
\draw[thick,color=gray,fill=white] (190:5.85) arc(190:270:5.85);   
\node[anchor=west] at (270:5.9) {$\Lout_R(x)$};
\draw[thick] (190:1.7) arc(190:270:1.7); \node[anchor=west] at (270:1.7) {$\Lint_R(x)$};
\draw[thick,color=green] (230:1.75) circle (0.2cm);
\draw[thick,color=red] (234:1.9) circle (0.2cm); 
\draw[thick,color=blue] (238:2.25) circle (0.6cm); 
\draw[thick,color=green] (235:2.95) circle (0.2cm);\draw[thick,color=green] (237:3.2) circle (0.2cm);
\draw[thick,color=green] (236:3.5) circle (0.2cm);\draw[thick,color=green] (233:3.6) circle (0.2cm);
\draw[thick,color=red] (231:3.85) circle (0.2cm); 
\draw[thick,color=blue] (231:4.15) circle (0.6cm); 
\draw[thick,color=green] (238:4.5) circle (0.2cm);\draw[thick,color=green] (239:4.75) circle (0.2cm);
\draw[thick,fill=black] (239:5.1) circle (0.02cm); 
\draw[thick,fill=black] (239.5:5.25) circle (0.02cm); 
\draw[thick,fill=black] (240:5.4) circle (0.02cm); 
\draw[thick,color=green] (1.7,-2.5) circle (0.2cm);
\node[anchor=west] at (2.1,-2.5) {\textcolor{green}{$\Lambda_r(x_m)\text{ is }(n,\lambda,\omega)\text{-good }$}};
\draw[thick,color=red] (1.7,-3.3) circle (0.2cm);
\node[anchor=west] at (2.1,-3.3) {\textcolor{red}{$\Lambda_r(x_m) \text{ is }(n,\lambda,\omega)\text{-bad }$}};
\draw[thick,color=blue] (1.7,-4.1) circle (0.4cm); 
\node[anchor=west] at (2.1,-4.1) {\textcolor{blue}{container $\Cont_i$}};
\end{tikzpicture}
\caption{Illustration of the iteration in the multiscale analysis.}
\label{abb:msa}
\end{figure}
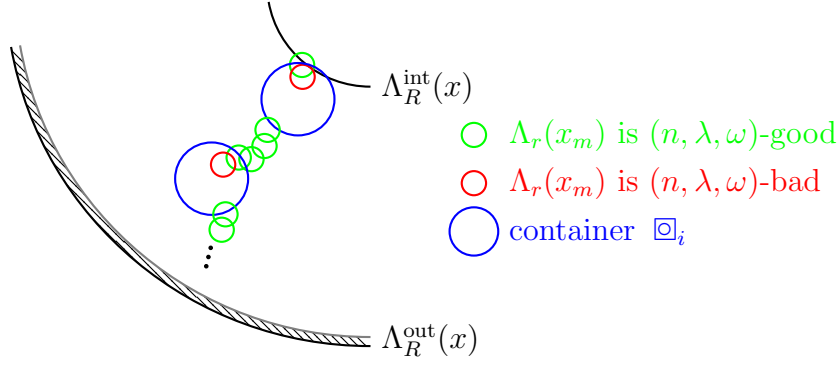

As the prefactors are smaller than one---at least for big radii---the estimate gets better with each step of iteration.
Thus we want to have as many iterations as possible. We have to obey the following restrictions: For case (+) we need $\Lambda_r(x_m)$ and $\Lout_R(x)$ to be disjoint, to be able to apply the geometric resolvent inequality.
For case ($-$) $\Cont_i^\mathrm{out}$ and $\Lout_R(x)$ have to be disjoint and a step (+) has to be added after choosing $\widehat{x}$.

We will denote the center, at which we have to stop the iteration by $x_k$. To be able to perform a Combes-Thomas estimate with 
$\Lint_r(x_k)$ and $\Lout_R(x)$, this two sets have to have a positive distance.

Now we want to estimate the minimal number of steps corresponding to case (+), denoted by $k_+$. 
The worst case is if we start as far as possible from $x$ and go with each step straight forward to the boundary $\Lout_R(x)$ using the maximally possible step width and thereby run through all containers.
See the illustration in figure \ref{abb:msa}.
The number of steps by ($-$) will be denoted by $k_-$. Hence we get
\begin{align*}
&k_+\cdot  \underbrace{\left( \frac{4}{3}r\right)}_{\text{step width }  (+)} +\sum\limits_{i=1}^3\underbrace{\left(2r_i + \frac{2}{3}r + 2\okl\right)}_{\text{step width case ($-$)}}\geq
\underbrace{R - 3\okl}_{\Lout_R(x)}-\underbrace{\left( \frac{R}{3}  + \frac{r}{3} + \okl \right)}_{\Lint_r(x)\cap\Lint_R(x_0)\neq \varnothing} - \underbrace{\left(\frac{r}{3} + 2\okl\right)}_{x_k \text{-relation}}.
\end{align*} 
With part 2.a) of step \ref{schritt2} we conclude
\begin{align}
\label{gl:msa:k+}
k_+ &\geq \frac{1}{2}\frac{R}{r}-\frac{82}{5}-\frac{159}{4} \frac{\okl}{r} \geq \frac{1}{2}\frac{R}{r}-17
\end{align}
for $r\geq 300 \okl=\rgeom$.
In total we get by the relations \eqref{gl_MSA_+} for case (+) and \eqref{gl:($-$)} in case ($-$), using $\delta_- <1$ and $k=k_++k_-$
\begin{flalign*}
&\hspace{0.5cm}\left\| \ind_{\Lout_R(x)}\left(H^{\Lambda_R(x)}(\omega) -\lambda\right)^{-1} \ind_{\Lint_r(x_0)} \right\|&\\
&\hspace{2cm}\leq\delta_+^{k_+} \cdot \underbrace{{\delta_-}^{k_-}}_{< 1}\cdot \left\| \ind_{\Lout_R(x)}\left(H^{\Lambda_R(x)}(\omega) -\lambda\right)^{-1} \ind_{\Lint_r(x_k)} \right\|.
\end{flalign*}
We can estimate the last norm analogue as in \eqref{gl_it_-1} with the Combes-Thomas estimate ($s=\lambda-\frac{1}{2}R^{-\theta n}$, $t=\lambda+\frac{1}{2}R^{-\theta n}$) as $\Lambda_R(x)$ is $\lambda$-dissonant
\begin{flalign}
\label{gl:it_schritt}
&\hspace{0.5cm}\left\| \ind_{\Lout_R(x)}\left(H^{\Lambda_R(x)}(\omega) -\lambda\right)^{-1} \ind_{\Lint_r(x_0)} \right\|\leq \delta_+^{k_+}\cdot
2 \,C_\mathrm{CTA}\cdot R^{ \theta n}.&
\end{flalign}
With this relation we are able to prove the induction step. First we cover $\Lint_R(x)$ with balls of the form $\Lint_r(y)$, where the distance of the used centers to $x$ is bounded by $\frac{R+r}{3}$. Hence we can use relation \eqref{gl:it_schritt} with $x_0=y$.
Estimating the number of centers we have
\begin{flalign*}
&\hspace{0.5cm}\left\| \ind_{\Lout_R(x)} \left(H^{\Lambda_R(x)}(\omega)-\lambda\right)^{-1} \ind_{\Lint_R(x)}\right\|&\\
&\hspace{2cm}\leq
\cp \frac{10\left(\frac{R+r}{3}+\frac{r}{10}+\okl\right)^d}{r}   \cdot {\delta_+}^{k_+}\cdot 2\,C_\mathrm{CTA}\cdot R^{\theta n}&\\
&\hspace{1.49cm} \underset{\delta_+< 1,\eqref{gl:msa:k+}}{\leq} 2\, C_\mathrm{CTA} \cdot \cp \cdot
\left(\delta_+\right)^{\frac{R}{2r}-17}\cdot
\frac{10\left(\frac{R}{3}+\frac{13}{30}r+\okl\right)^d}{r}    \cdot R^{\theta n}&\\
&\hspace{2cm}\leq \tilde{C} \cdot\left( \delta_+ \right)^{\frac{R}{2r}-17 } \cdot r^{\alpha d-1+\alpha \theta n}.&
\end{flalign*}
By \eqref{IP} we get
\begin{flalign*}
\hspace{2cm}&\leq \tilde{C} \cdot (\delta_+)^{\frac{R}{2r} -17} \cdot r^{n}.&
\end{flalign*}
For $r\geq \rsix$ we have $\delta_+ \leq \frac{1}{2}$. Combined  with $r\geq (4\cdot 17)^\frac{1}{\alpha-1}$ 
we conclude for the exponential function $\frac{r^{\alpha-1}}{2}-17\leq \frac{r^{\alpha-1}}{4}$, yielding
\begin{flalign*}
\hspace{2cm}&\leq \tilde{C} \cdot \left(\frac{1}{2}\right)^{\frac{r^{\alpha-1}}{4}} \cdot r^{n}\leq \tilde{C}\cdot \left(\frac{\sqrt[4]{8}}{2} \right)^{r^{\alpha-1}} \cdot r^{n}. &
\end{flalign*}
Hence there exists a radius $\reight$ with $\reight \geq \max\{\rsix,\rseven \}$, such that for all radii $r\geq \reight(C_\mathrm{CTA},\cp,\rsix,\rseven)$ the exponential part is decaying faster as the polynomial $\tilde{C}\cdot r^{n+\alpha n}\leq \tilde{C}\cdot r^{4n} $ is growing:
\begin{flalign*}
&\hspace{0.5cm}\left\| \ind_{\Lout_R(x)}\left(H^{\Lambda_R(x)}(\omega) -\lambda\right)^{-1} \ind_{\Lint_r(x)} \right\| \leq r^{-\alpha n}=R^{-n}. &
\end{flalign*}This ends step 4.
\end{proof}
Due to step \ref{schritt4} we know: If $G(I,r,n,\xi)$ holds for all $\omega \in \Omega_\mathrm{G}(x)\cap\Omega_\mathrm{G}(y)\cap\Omega_\mathrm{W}^c$ and $r\geq \reight$ also $G(I,R,n,\xi)$ holds .
For the used events we know from part 1 of step \ref{schritt1} and step \ref{schritt3}
\begin{align*}
\PP(\Omega_\mathrm{G})\geq 1-\frac{1}{3}R^{-2\xi}, \qquad \PP(\Omega_\mathrm{W})\leq \frac{1}{3}R^{-2\xi},
\end{align*}
 for all $r\geq \max\{\rfour,\rfive\} $. This yields
\begin{align*}
\PP(\Omega_\mathrm{G}(x)\cap\Omega_\mathrm{G}(y)\cap\Omega_\mathrm{W}^c) &=1-\PP(\Omega_\mathrm{G}(x)^c\cup\Omega_\mathrm{G}^c(y)\cup\Omega_\mathrm{W})\\
&\geq 1-\left(\PP(\Omega_\mathrm{G}(x)^c) +\PP(\Omega_\mathrm{G}(y)^c)+\PP(\Omega_\mathrm{W})  \right)\\
&=1-R^{-2\xi},
\end{align*}
and for all radii $r\geq \rnine:=\max\{\rfour,\rfive,\reight\}$ the assertion stated in the theorem. This finishes the proof of theorem \ref{satz:indsatz}.

\begin{bem}
\label{bem:pol:decay}
We are not able to prove exponential decay of the local resolvent. The reason is the covering of an arbitrary metric graph, which is not as precise as in the $\gz^d$ case. It results in the estimate
$k_+ \geq \frac{1}{2}\cdot \frac{R}{r}-c$, where the prefactor $\tfrac{1}{2}$ plays an essential role. If we want to prove
\begin{align*}
\left\| \ind_{\Lout_r(v)} \left(H^{\Lambda_r(v)} (\omega)-\lambda  \right)^{-1} \ind_{\Lint_r(v)}  \right\| &\leq e^{-\gamma\cdot r} \qquad \text{and}\\
\left\| \ind_{\Lout_R(v)} \left(H^{\Lambda_R(v)}( \omega)-\lambda  \right)^{-1} \ind_{\Lint_R(v)}  \right\| &\leq e^{-\gamma_R\cdot R}
\end{align*}
we would end in $\gamma_R \sim \tfrac{1}{2}\gamma$, leaving us with no positive prefactor for the exponential decay for $r \to \infty$.
\end{bem}

\section{Spectral localization} 
\label{ab:SL}
From section \ref{ab:VEF} we know the existence of generalized eigenfunctions with a maximal polynomial growth for operators $(H^{P,L}(\omega))$. Together with the induction theorem and the polynomial decay we can prove polynomial decay of the generalized eigenfunctions. If the decay is fast enough we can conclude the eigenfunctions lying in $L^2(X_E)$ and thus being real eigenfunctions. Hence we get almost surely discrete spectrum in the interval $I$ (from the previous section).

Therefore we need the following proposition to estimate the norms of the generalized eigenfunctions.
\begin{hsatz}
\label{hsatz:AL1}
Let $\Gamma$ be a metric graph with \geomalles\ and $(H^{P,L}(\omega))$ a random operator with \RBPLS\ and \potalles.
Let a ball $\Lambda_r(v)$ with $v\in V$ and $r\geq 6\okl$ be given.
Then for each bounded interval $I_0 \in \rz$ there exists a constant $C_\mathrm{VEF}(I_0,\ukl,S,C_\pot)$, not depending on $r$ and $v$, such that for any generalized eigenfunction $f$ for $\lambda\in I_0$ and any $\omega \in \Omega$ with $\lambda \in \res(H^{\Lambda_r(v)}(\omega))$ we have
\begin{equation*}
\left\|\ind_{\Lint_r(v) }f \right\| \leq C_\mathrm{VEF}\left\| \ind_{\Lout_r(v)} \left(H^{\Lambda_r(v)}(\omega)-\lambda\right)^{-1}\ind_{\Lint_r(v) } \right\|
\left\| \ind_{\Lout_r(v)} f \right\|.
\end{equation*}
\end{hsatz}
We choose $r\geq 6\okl$ to obtain that $\Lint_r(v)$ and $\Lout_r(v)$ are disjoint induced subgraphs.
This relation is called eigenfunction decay inequality and 
the proof follows the steps of the proof of lemma 3.2.2 (b) in \cite{Stollmann-01}.
\begin{proof} 
Let $\psi:\left[0,\frac{\ukl}{2}\right]\to [0,1]$ be a smooth function with $\psi(x)=0$ in an open neighborhood of zero and $\psi(x)=1$ in an open neighborhood of $\frac{\ukl}{2}$. With the help of $\psi$ we can construct a cut-off function $\phi$ with support in $X_{E(v,r-\okl)}$.

On the boundary edges, i.\,e. edges having one vertex in $\rk{E(v,r-\okl)}$, we set $\phi$ equal to $\psi$ starting from the boundary vertex and continue on the rest of $X_{E(v,r-\okl)}$ by setting $\phi$ identically one. Outside of $X_{E(v,r-\okl)}$ we choose zero. Then we have 
$\tr{v}(\phi)\equiv 1$ on $\ik{E(v,r-\okl)}$, $\tr{v}(\phi)\equiv 0$ in all other vertices, $\supp \phi \subset X_{E(v,r-\okl)}$, $\supp \phi'\subset \Lout_r(v)$ and the boundary vertices of $\Gamma_{E(v,r-\okl)}$
are inner vertices of $\Lout_r(v)$.

Let $f$, $g\in \dom\lok(\form_\omega)$. Then with the notation of  section \ref{ab:ITG} we have
\begin{align*}
\forms[\phi f,g]-\forms[f,\phi g]&=\langle(\phi f)'|g'\rangle-\langle f'|(\phi g)'\rangle+\sum\limits_{v \in V}\skp{L_v \tr{v}(\phi f)}{\tr{v}(g)}{}\\
&\hspace{0.5cm}-\sum\limits_{v \in V}\skp{L_v \tr{v}(f)}{\tr{v}(\phi g)}{}+\langle \pot_\omega(\phi f)|g \rangle -  \langle \pot_\omega f|\phi g\rangle.
\end{align*}
Since $\phi$ has compact support both sums are finite. Moreover they are equal as of the construction of $\phi$ and the linearity of $L_v$.
As $\phi$ is real, the same holds for the scalar products with the potentials, yielding with the product rule 
\begin{align}
\label{gl_AL_Lemma2_1}
\forms[\phi f,g]-\forms[f,\phi g]&=\langle(\phi f)'|g'\rangle-\langle f'|(\phi g)'\rangle 
=\langle \phi'f|g'\rangle-\langle f'|\phi'g\rangle.
\end{align}
Let $f$ be a generalized eigenfunction for $\lambda \in I_0$ (in the sense of definition \ref{def:VEF}) and $g:=(H^{\Lambda_r(v)}(\omega)-\lambda)^{-1}\ind_{\Lint_r(v)}f $.
By proposition \ref{hsatz:VEF_RB} part 3 we know, that $f$ satisfies the boundary conditions of the operator, thus $f\in \dom\lok(\form_\omega)$ and $\phi f \in \dom(H^{\Lambda_l(v)}(\omega))$
\begin{align}
\label{gl_AL_Lemma2_2}
\notag (\form_\omega-\lambda)[\phi f,g]&=(\form^{\Lambda_r(v)}_\omega-\lambda)[\phi f,g]\\
\notag&=\skp{(H^{\Lambda_r(v)}(\omega)-\lambda)\phi f}{(H^{\Lambda_r(v)}(\omega)-\lambda)^{-1}\ind_{\Lint_r(v)} f}{}\\
 &=\skp{\phi f}{\ind_{\Lint_r(v)} f}{} \\
&=\left\|\ind_{\Lint_r(v)} f\right\|^2
\end{align}
as $\phi\big|_{\ind_{\Lint_r(v)}}\equiv 1 $.
From $\forms[f,\phi g]$ we get by partial integration
\begin{align}
\notag\overline{(\forms-\lambda)[f,\phi g]}&=(\forms-\lambda)[\phi g,f] \qquad (\tr{v}(\phi)\equiv 0 \text{ on } V\setminus \ik{E(v,r-\okl)} )\\
\notag&=\langle(\phi g)'|f'\rangle+\sum\limits_{v\in \ik{E(v,r-\okl)}} \skp{L_v \tr{v}(\phi g)}{\tr{v}(f)}{}+\langle(\pot_\omega)(\phi g)|f\rangle\\
\notag&=\langle -(\phi g)''| f\rangle+\langle(\pot_\omega)(\phi g)|f\rangle\\
\notag&\hspace{0.5cm}\underbrace{+\sum\limits_{v \in \ik{E(v,r-\okl)}}  \langle \str{v}((\phi g)'),\tr{v}(f)  \notag\rangle+\sum\limits_{v\in \ik{E(v,r-\okl)}}  \skp{L_v \tr{v}(\phi g)}{\tr{v}(f)}{}}_{=0,\text{ by boundary conditions}}\\
\notag&=\langle\left(H^{P,L}(\omega)-\lambda\right)\phi g|f\rangle\\
\label{gl:AL:par:in}&=0,
\end{align}
as $(\phi g)\in \dom(H^{P,L}(\omega)) $ and $f$ is a generalized eigenfunction of $H^{P,L}(\omega)$ for $\lambda$.
With \eqref{gl_AL_Lemma2_1} and\eqref{gl_AL_Lemma2_2} we have
\begin{align*}
\left\|\ind_{\Lint_r(v)} f\right\|^2 &=\skp{f\phi'}{g'}{} -\langle f'|\phi' g\rangle\\
& \leq\left\|\psi'\right\|_\infty \left\| \ind_{\Lout_r(v)} f \right\| \left\| \ind_{\supp \phi'}\left(\left(H^{\Lambda_r(v)}(\omega)-\lambda\right)^{-1}\ind_{\Lint_r(v)}f\right)'  \right\| \\
&\hspace{1cm}+ \left\|\psi'\right\|_\infty \left\| \ind_{\supp \phi'}f' \right\| \left\|\ind_{\Lout_r(v)}\left(H^{\Lambda_r(v)}(\omega)-\lambda\right)^{-1}\ind_{\Lint_r(v)} f \right\|.
\end{align*}
As the edges satisfy the cone condition we can estimate $f'$ using proposition \ref{hsatz:VEF:kegel}.
Together with $\supp \phi'\subset \Lout_r(v)$ 
\begin{equation*}
\left\| \ind_{\supp \phi'} f'\right\|\leq \left\| \ind_{\Lout_r(v)} f' \right\| \leq C_\mathrm{cone} \left\| \ind_{\Lout_r(v)} f \right\|
\end{equation*}
holds. We will denote the annulus of width $\okl$ inside of $\Lout_r(v)$ with $\Lambda_r^\delta(v):=\Lambda_{r-\okl}(v)\setminus \Lambda_{r-2\okl}(v) $. By construction we have $\ind_{\Lambda_l^{\delta}(v)}\geq \ind_{\supp \phi'}$.
Using partial integration---in analogy to relation \eqref{gl:AL:par:in}---we
see, that $g$ is a weak solution of $H^{P,L}(\omega)g=\lambda g+\ind_{\Lint_r(v)}f$ on $\Lambda^\delta_r(v)$ (i.\,e. $g$ is a weak solution of $(H^{P,L}(\omega)-\lambda)g=0 $ on $\Lambda_r^{\delta}(v)$). Using the Caccioppoli inequality from \ref{hsatz_ru2} we see
\begin{align*}
\left\| \ind_{\supp \phi'}\left(\left(H^{\Lambda_r(v)}(\omega)-\lambda\right)^{-1}\ind_{\Lint_r(v)}f\right)'  \right\| 
&\leq
\left\| \ind_{\Lambda_r^\delta(v)}\left(\left(H^{\Lambda_r(v)}(\omega)\right)^{-1}\ind_{\Lint_r(v)}f\right)'\right\|\\
& \hspace{-2cm}=C_\mathrm{CP} (1+|\lambda|)\left\| \ind_{\Lout_r(v)} \left( H^{\Lambda_r(v)}(\omega)-\lambda\right)^{-1} \ind_{\Lint_r(v)}f  \right\|.
\end{align*}
Overall we get
\begin{align*}
\left\|\ind_{\Lint_r(v)} f\right\|^2 &\leq
\|\psi'\|_\infty \big((1+|\lambda|)C_\mathrm{CP}+C_\mathrm{cone} \big) \\
& \hspace{1cm}\cdot \left\|\ind_{\Lout_r(v)}\left(H^{\Lambda_r(v)}(\omega)-\lambda\right)^{-1}\ind_{\Lint_r(v)}  \right\| \left\|\ind_{\Lint_r(v)} f \right\| \left\| \ind_{\Lout_r(v)} f \right\|. \qedhere
\end{align*}
\end{proof} 
Combining the last proposition, the induction theorem and the knowledge of generalized eigenfunctions from section \ref{ab:VEF} we can prove the polynomial decay of the generalized eigenfunctions.
\begin{hsatz} 
\label{hsatz:SL:poly}
Let $\Gamma$ be a metric graph and $(H^{P,L}(\omega))$ a random operator with \geomalles, \RBPLS\ and \potalles.
Let the induction parameters of the induction theorem \ref{satz:indsatz} be chosen according to \eqref{IP} and
 $I=[\leos,\leos+\frac{1}{2}{\rzero}^{\beta-2}]$
with $\rzero\geq \max\{\rthree,\rnine\}$.

Then there exists a $n_0\in \rz^+$ and a set $\Omega_0\subset \Omega$ with $\PP(\Omega_0)=1$, such that for all generalized eigenfunctions $f_{\lambda}$ of $H^{P,L}(\omega)$ for $\lambda \in I$ and $\omega\in \Omega_0$ we have:
There  exists a constant $C_{\mathrm{poly}\hspace{-1pt}\searrow}=C_{\mathrm{poly}\hspace{-1pt}\searrow}(C_\mathrm{VEF},C_{\mathrm{poly}\hspace{-1pt}\nearrow},d)$ and a radius $\rten$---depending on the eigenfunction---with
\begin{equation*}
\|\ind_{\Lambda_{20\okl}(x)} f_\lambda \|\leq C_{\mathrm{poly}\hspace{-1pt}\searrow} \metrik(x,0)^{-n_0}
\end{equation*}
for all $x\in X_E$ with $\metrik (x,0)\geq \rten$.
\end{hsatz} 
A statement like this was already proven in \cite{DreifusK}. As we only can get polynomial decay, the proof can be simplified.

 \begin{proof}
The choice of radius $\rzero$ guaranties, that the induction start and the induction theorem are valid and that $G(I,r_k,n,\xi)$ holds true for the sequence of radii $r_k={\rzero}^{\alpha^k}$ with $k\in \nz\cup\{0\}$.

For $v_0\in V$ we look at the following induced subgraphs
\begin{equation*}
A_{k+1}(v_0):=\Lambda_{2r_{k+1}+2\okl}(v_0) \setminus \Lambda_{2r_k+\okl}(v_0), \qquad k\in \nz\cup \{0\}
\end{equation*}
with the cover-raster $V_{k+1}(v_0):=V_{2r_{k+1}+2\okl,\frac{r_k}{10}}(v_0) $.
By $\Omega_k(v_0)$ we denote the event
\begin{align*}
\Omega_k(v_0):&=\left\{\omega\in \Omega\with \esgibt \lambda\in I, x\in A_{k+1}(v_0)\cap V_{k+1}(v_0) : \right.\\
&\hspace{2.93cm} \Lambda_{r_k}(x) \text{ and }\Lambda_{r_k}(v_0) \text{ are } (n,\lambda,\omega)\text{-bad}  \Big\}.
\end{align*}
If we use $G(I,r_k,n,\xi)$ to estimate the probability of two disjoint $(n,\lambda,\omega)$-bad balls and lemma \ref{hsatz_rasterzahl} to estimate the maximal number of raster points we have
\begin{align*}
\PP(\Omega_k(v_0)) 
&\leq \cp \,\frac{10\cdot( 2r_{k+1} +2\okl)^d}{r_k}\,  r_k^{-2\xi}\\
&\leq 10\cdot 4^d\cdot\cp\cdot r_k^{\alpha d-1-2\xi}\\
&\hspace{-0.12cm}\underset{\eqref{IP}}{\leq} r_k^{-\delta}
\end{align*}
for a $\delta<0$ and after $k\geq k_0$, depending on $\cp$, $\xi$, $d$ and $\alpha$ (see remark A.1 point (vii) for details).
Hence the sum over the probabilities of the events $\Omega_k$ is convergent
\begin{align*}
\sum\limits_{k\geq k_0}^\infty \PP(\Omega_k(v_0)) \leq\sum\limits_{k\geq k_0}^\infty \left( r_0^{\alpha^k} \right)^{-\delta}=\sum\limits_{k\geq k_0}^\infty \left(r_0^{-\delta}\right)^{\alpha^k}<\infty,
\end{align*}
as $r_0^{-\delta}<1$ and $\alpha^k >k$ after a certain  $k$.
The Borel-Cantelli-lemma yields
\begin{equation*}
\PP\left\{ \omega\in\Omega \text{ with } \esgibt k_1\in \nz\text{, s.\,t.}\fa k\geq k_1: \omega\not\in \Omega_k(v_0) \right\}=1.
\end{equation*}
The intersection of all those events over $v\in V$ still has measure one
and gives the claimed event
\begin{equation*}
\Omega_0:=\left\{ \omega\in \Omega\text{ with } \fa v\in V \ \esgibt k_v\in\nz, \text{s.\,t.} \fa k\geq k_v: \omega\not\in\Omega_k(v) \right\}.
\end{equation*}
Let $\omega\in \Omega_0$ and $f$ be a generalized eigenfunction of $H^{P,L}(\omega)$ to $\lambda\in I$. Let $v_0 \in V$ with $\left\| \ind_{\Lambda_{20\okl}(v_0)} f  \right\|>0 $ (otherwise $f\equiv 0$ holds). 
Then $\Lambda_{r_k}(v_0)$ is not $(n,\lambda,\omega)$-good for infinitely many $k$, since otherwise propositions \ref{hsatz:AL1} and 
\ref{hsatz:VEF:poly} provide
\begin{flalign*}
\left\| \ind_{\Lambda_{20\okl}(v_0)}f \right\|
&\leq \left\| \ind_{\Lint_{r_k}(v_0)}f \right\|&\\
&\leq C_\mathrm{VEF}
\left\| \ind_{\Lout_{r_k}(v_0)}\left( H^{\Lambda_{r_k}(v_0)}(\omega)-\lambda\right)^{-1}\ind_{\Lint_{r_k}(v_0)}  \right\|
\left\| \ind_{\Lout_{r_k}(v_0)} f\right\|\\
& \leq C_\mathrm{VEF} \cdot {r_k}^{-n} \cdot C_{\mathrm{poly}\hspace{-1pt}\nearrow} \cdot {r_k}^d \cdot \left( 1+\metrik(v_0,0)+r_k+\okl \right)^{\frac{d+2}{2}}&\\
&\leq C_{v_0} \cdot {r_k}^{\frac{3d+2}{2}-n},
\end{flalign*}
which converges to zero for any infinite subsequence of $(r_k)$ with $(n,\lambda,\omega)$-good balls---yielding a contradiction.
Thus there exists a $k_2 \in \nz$, such that for all $k\geq k_2$ the balls $\Lambda_{r_k}(v_0)$ are 
$(n,\lambda,\omega)$-bad.
Hence we conclude for $k\geq \max\{k_2,k_{v_0}\}$ and $v\in A_{k+1}(v_0)\cap V_{k+1}(v_0) $, that $\Lambda_{r_k}(v)$ is $(n,\lambda,\omega)$-good.
With the annuli $A_{k+1}(v_0)$ we can cover the whole graph $\Gamma$ and for $k_3:=\max\{k_2,k_{v_0}\}$ we have
\begin{equation*}
\bigcup\limits_{k\geq k_3} A_{k+1}(v_0)=\Gamma \setminus \Lambda_{2r_{k_3}+\okl}(v_0).
\end{equation*}
For any $y\in A_{k+1}(v_0)$, with $k\geq k_3$, there is such a center $y_1\in A_{k+1}(v_0)\cap V_{k+1}(v_0)$, such that $\Lambda_{20\okl}(y)\subset \Lint_{r_k}(y_1)$ (which follows from lemma \ref{hsatz_ueberdeckung} for $r_k\geq 780 \okl$).
With proposition \ref{hsatz:AL1} we estimate
\begin{align*}
\left\| \ind_{\Lambda_{20\okl}(y)} f \right\| &\leq \left\| \ind_{\Lint_{r_k}(y_1)} f  \right\|\\
& \leq C_\mathrm{VEF}\cdot {r_k}^{-n} \cdot \left\| \ind_{\Lout_{r_k}(y_1)} f \right\|,
\intertext{with the polynomial growth from proposition \ref{hsatz:VEF:poly} we get}
&\leq C_\mathrm{VEF}\cdot {r_k}^{-n} \cdot C_{\mathrm{poly}\hspace{-1pt}\nearrow} \cdot {r_k}^d \cdot \left(1+\metrik(y_1,0)+r_k+\okl \right)^{\frac{d+2}{2}}\\
&\leq \tilde{C} \cdot {r_k}^{d-n}\cdot \left(1+\metrik(y_1,y)+\metrik(y,v_0)+\metrik(v_0,0)+r_k+\okl\right)^{\frac{d+2}{2}}.
\intertext{With $\metrik(y,v_0)\in [2r_k+\okl,2r_{k+1}+3\okl ]$, i.\,e. $\left(\frac{\metrik(y,v_0)}{3}\right)^{\frac{1}{\alpha}}\leq r_k\leq \frac{\metrik(y,v_0)}{2} $
and $\metrik(y,y_1)\leq \frac{r_k}{3}+\okl-20\okl $ we continue by}
& \leq \tilde{C} \cdot {r_k}^{d-n} \cdot \left(1+ \frac{4 r_k}{3}-19\okl +\metrik(y,v_0)+\metrik(v_0,0)\right)^{\frac{d+2}{2}}\\
& \leq \tilde{C} \cdot \left(\frac{\metrik(y,v_0)}{3}\right)^{\frac{d-n}{\alpha}}\cdot\left(2\metrik(y,v_0)+\metrik(v_0,0)\right)^{\frac{d+2}{2}}\\
&\leq C_{\mathrm{poly}\hspace{-1pt}\searrow} \cdot {\metrik(y,v_0)}^{\frac{d}{\alpha}+\frac{d+2}{2}-\frac{n}{\alpha}}.
\end{align*}
Since \eqref{IP} we have that the exponent is negative.
The decay-constant $C_{\mathrm{poly}\hspace{-1pt}\searrow}$ does not depend on the index of the annulus $k$, but only on  $d$, $\alpha$, $C_\mathrm{VEF}$ and $C_{\mathrm{poly}\hspace{-1pt}\nearrow}$.
Thereby we proved polynomial decay around $v_0$ beginning from the radius $\rten:=2r_{k_3}+2\okl$.
\end{proof} 

For a polynomially decaying function---in the sense of the last proposition---we need the same degree of decay as for a pointwise decaying functions, to prove, that the function actually  lies in $L^2$.
\begin{bem} 
\label{bem:SL:fallen_L2}
Let $\Gamma$ be a metric graph with \geomalles. Then $f=(f_e)_{e\in E}$ with $f_e \in L^2(I_e)$ lies in $L^2(X_E)$, if
\begin{align*}
\fa v \in V: \qquad \left\| \ind_{\Lambda_{20\okl}(v)} f \right\| \leq \metrik(v,0)^{-m} \quad \mathrm{with}\ m> \frac{d+1}{2}.
\end{align*}
\end{bem}
\begin{proof} 
Let $\Lambda_k:=\Lambda_{20k\okl}(0)\setminus \Lambda_{20(k-1)\okl-\okl}(0)$. 
Then $\Gamma$ can be covered by the annuli $\Lambda_k$.
For any point $x \in X_{\Lambda_k} $ we have $\metrik(x,0)\geq 20k\okl-21 \okl$. 
Moreover we can cover any annulus using lemma \ref{hsatz_ueberdeckung} and the cover-raster $V_{20k\okl,5\okl}(0)$ with balls of radius $20\okl$. 
\begin{align*}
\left\| \ind_{\Lambda_k}f \right\|^2 &\leq \sum\limits_{v\in \Lambda_k \cap V_{20k\okl,5\okl}(0)} \left\| \ind_{\Lambda_{20\okl}(v)} f\right\|^2\\
&\hspace{-0.075cm}\underset{\text{ass.}}{\leq} \sum\limits_{v\in \Lambda_k \cap V_{20k\okl,5\okl}(0)} {\metrik(v,0)}^{-2m}\\
&\hspace{-0.5cm}\underset{\text{lemma }\ref{hsatz_rasterzahl}}{\leq} \cp \,\frac{(20k\okl)^d}{5\okl}\, (20k\okl-21\okl)^{-2m}\\
&\leq C \cdot k^d \cdot (k-2)^{-2m} \qquad (k\geq 3).
\end{align*}
Now we have
\begin{align*}
\sum\limits_{k=4}^\infty \left\| \ind_{\Lambda_k}f \right\|^2 \leq \sum\limits_{k=4}^\infty C\frac{k^d}{(k-2)^2m}
\leq C\cdot 2^d \sum\limits_{k=4}^\infty \frac{(k-2)^d}{(k-2)^{2m}}<\infty,
\end{align*}
if $d-2m<-1$, which was the assumption.
\end{proof}
From the polynomial decay of the generalized eigenfunctions we are now able to conclude pure point spectrum and prove our main theorem:  
\begin{satz} 
\label{satz:hauptsatz}
Let $\Gamma$ be a metric graph with \geomalles\ and $H^{P,L}(\omega)$ a random operator with \RBPLS\ and \potalles.
Let $\rzero$ be the radius from the last proposition and $n\geq 9\alpha d +d-2$.
Then $H^{P,L}(\omega)$ has $\omega$-almost surely pure point spectrum in $I=\left[\leos,\leos+\frac{1}{2}{\rzero}^{\beta-2}\right]$
with polynomially decaying eigenfunctions.
\end{satz}

\begin{proof} 
By assumption we can apply proposition \ref{hsatz:SL:poly}. Thus we get a set $\Omega_0 \subset \Omega$ of measure one, such that
generalized eigenfunctions to operators corresponding to those $\omega$ decay polynomially of degree $\frac{n}{\alpha}-\frac{d}{\alpha}-\frac{d+2}{2} $.
By remark \ref{bem:SL:fallen_L2} those functions lie in $L^2(X_E)$, if the degree satisfies
\begin{equation*}
\frac{n}{\alpha}-\frac{d}{\alpha}-\frac{d+2}{2} > \frac{d+1}{2}.
\end{equation*}
This is true by $n\geq 9\alpha d +d-2$.

As $L^2(X_E)$ is separable, $H(\omega):=H^{P,L}(\omega)$ can have at most countably many different eigenvalues.
Let $\rho_{H(\omega)}$ be the spectral measure of $H(\omega)$. Then---by corollary \ref{cor:VEF}---there exist generalized eigenfunctions for $\rho_{H(\omega)}$-almost all $\lambda \in I\cap \sigma(H(\omega))$, which by the above calculations lie in $L^2(X_E)$.
We denote the set of corresponding $\lambda$ with $A_0$. Thus $A_0$  has to be countable and we know that the spectral measure restricted to $I$ is supported on $A_0$ and is discrete. Hence $H^{P,L}(\omega)$ has pure point spectrum in $I$.
\end{proof}

\begin{bem}
\begin{enumerate}
 \item 
The spectrum of the operator family $(H^{P,L}(\omega))$ is in general not deterministic.
If $H^{P,L}(\omega)$ has spectrum in $I$, it is pure point spectrum. But it doesn't have to be spread out over the whole interval. Moreover the measure of $\omega$'s, corresponding to operators $H^{P,L}(\omega)$ not having spectrum in $I$, might have measure greater than zero.
See example \ref{bsp:nichtdet} for an illustration.

\item 
With all the necessary estimates in sections \ref{sec:CTE} to \ref{ab:VEF} we can use the multiscale analysis from \cite{ExnerHS-07}
to conclude spectral localization with exponential decaying eigenfunctions and dynamical localization for all Laplacians with boundary conditions of the form \RBPLS\ on the metric graph $\gz^d$ with random potentials with \potalles\ and parameters $c_-=c_+=1$.

\item 
This is the first localization proof for metric graphs, leaving $\gz^d$ or special metric trees.  Also the possible boundary conditions were extended from $\delta$-boundary conditions to all local boundary conditions, which yield a lower bounded self-adjoint operator.

\item
We only considered a single particle model. 
For boundary conditions and spectral properties of singular two-particle Laplacians on finite, compact metric graphs see \cite{BolteK}.
There are is also a multiscale analysis for multi-particle models developed in \cite{Chulaevsky} and applied to a many particle quantum graph over $\gz^d$ with Kirchhoff boundary conditions in \cite{Sabri}.
\end{enumerate}
\end{bem}

\section{Explanations and Examples}
\label{ab:el_va}
In this section we give explanations and applications of the obtained localization theorem. In particular we will analyze the localization theorem in the case when the considered operator family has no deterministic spectrum and only a few realizations have spectrum at the lower bound of the spectra.
\begin{bsp}
\label{bsp:nichtdet}
We state an example of a non-deterministic model, where the lower end of the spectrum is known, but changes dramatically with $\omega$.
Let $\Gamma=(E,V,l,i,j)$ be a metric graph with \geomalles\ and a uniform polynomial growth of degree $d$.
Let $H^{P,L}$ be a Laplacian with \RBPLS.
We denote the lower bound of the operator $H^{P,L}$ by $\leos$ and choose the random potential according to the following
\begin{align*}
q_-&:=1, \qquad q_+:=2,\\
\nu_e&:=\begin{cases} (-\leos+3) \cdot \ind_{I_e},& \leos < 3\\ \ind_{I_e},& \leos \geq 3 \end{cases}\\
\dichte_\mu&:=\begin{cases} (2d)(x-1)^{2d-1}, & 1\leq x\leq 1+2^{-\frac{1}{2d}} \\ \left(2-2^{1-\frac{1}{2d}}\right)^{-1}, & 1+2^{-\frac{1}{2d}}< x\leq 2. \end{cases}
\end{align*}
The choice of $\dichte_\mu$ guaranties \eqref{pot:unord}. 
By construction $\pot_\omega=(\omega_e \nu_e)$ obviously satisfies \potalles\ and we have $H^{P,L}(\omega)\geq 3$.

Now we modify the given graph and operator by adding an additional edge $\tilde{e}$ to an arbitrary vertex. We define $l(\tilde{e})=\pi$  as its length and set Dirichlet boundary conditions at both end points.
Furthermore we set $\nu_{\tilde{e}}=\ind_{I_{\tilde{e}}}$.

The modified graph and operator will be denoted by $\tilde{\Gamma}$ and $\tilde{H}^{P,L}$. note that the requirements for the localization are still satisfied.
With the decoupling of the Dirichlet boundary conditions and proposition \ref{hsatz:weyl_dir} part 3 we see
\begin{align*}
\sigma\left(\tilde{H}^{P,L}\right)=\sigma\left(H^{P,L}\right)\cup \{n^2 \text{ with } n\in \nz\}.
\end{align*}
Thus the spectrum of $\tilde{H}^{P,L}(\omega)$ starts at the lowest Dirichlet eigenvalue coming from $\tilde{e}$, which is equal to $1+\omega_{\tilde{e}}$.
The localization theorem yields pure point spectrum of $\tilde{H}^{P,L}(\omega)$ in the interval $(2,\ep)$, for some small $\ep$.
Here $\tilde{H}^{P,L}(\omega)$ has at most one eigenvalue in this interval and the measure of all realizations having no eigenvalue in the interval at all can be calculated using the density $\dichte_\mu$.
\end{bsp}

This is an pathologic example, but with  a more complicated setup it might be totally unclear how the lower end of the spectrum of the random operator behaves.

The localization theorem is still strong, if the considered operator family has deterministic spectrum. A big group of such operators can be found on Cayley graphs, which will be explained in the following.
\begin{defn}
Let $G$ be a finitely generated group and $S$ its generating set. Let $l:S\to \rz^+$ be a given function. We define the metric Cayley graph
$\Gamma(G,S)=(E,V,l,i,j)$ by 
\begin{align*}
V&=G & E&=\{ (g,h) \with g^{-1}h\in S\}\\
l(e)&=l((g,h)):=l(g^{-1}h)\\
i(e)&=i((g,h)):=g & j(e)&=j(g,h):=h
\end{align*}
\end{defn}
We will give the following remarks on loops and multiple edges.
\begin{itemize}
\item Loops correspond to unities in the generating set.
\item If $s\in S$ and $s^{-1}$, which are no unit, then there are to edges $(g,h)$ and $(h,g)$ 
between two vertices, but with different direction. In the study of undirected graphs, both edges will be identified with each other.
\item Multiple edges are excluded in this notation. They might be added by defining the generating set as finite subset of $G\times \nz_0$.
\end{itemize}
The growth of Cayley graphs obeys
\begin{itemize}
\item The growth of a metric Cayley graph corresponds to the combinatoric Cayley graph and is equal for each vertex, as the graph is translation invariant.
\item Moreover the growth equals the growth of the group $G$, where neighborhoods are defined by
\begin{equation*}
V_n(G):=\left\{ g\in G \with g=\prod\limits_{k=1}^n s_k : s_k \vee s_k^{-1}\in S \right\}.
\end{equation*}
\end{itemize}
We know for the growth of groups from \cite{Gromov}:
\begin{satz}
Each finitely generated group has polynomial growth, iff the group is virtually nilpotent.
\end{satz}
Thus there exists a characterization of polynomial growing Cayley graphs.
\begin{defn}
Let $\Gamma(G,S)$ be a metric Cayley graph.
The group operation defines an operation $\circ:E\times G \to E$ in the following way: 
For an edge $e=(g,h)$ corresponding to the generator $s=g^{-1}h$, i.\,e. $(g,h)=(g,gs)$, it is defined by
\begin{equation*}
e\circ k :=(k\cdot i(e),k\cdot j(e)) = (k g,k h)=(k g,k g g^{-1} h)=(k g,k g s)
\end{equation*}
for all $e\in E$ and $k\in G=V$.
\end{defn}
This mapping preserves the group structure, as an edge corresponding to a generator $s$  is mapped to an edge generated by $s$.
\begin{satz}
\label{satz_CayleyG}
Let $\Gamma(G,S)$ be a metric Cayley graph of polynomial growth. Let $(P,L)$ be one parametrization of boundary conditions \RBPLS\ for a vertex with degree $2|S|$. 
Moreover let for each generator $s\in S$ be a potential $\nu_s$ be given, which satisfies \potalles\ on the edge $(1,1 s)$.
Then the operator $H^{P,L}(\omega)$ with
\begin{align*}
H^{P,L}(\omega)f&:=-f'' + \left(\sum\limits_{s\in S}\nu_{s}f_{(g,gs)}\right)_{g\in G},\\
\dom(H^{P,L}(\omega))&:=\{ f\in W^{2,2}(X_E) \with \fa g\in G : P \tr{g}(f)=0,\\
&\hspace{4.96cm} L\tr{g}(f)=(1-P)\str{g}(f') \}
\end{align*}
is ergodic and has deterministic spectrum.
\end{satz}
\begin{proof}
We define an ergodic operator family $T_k:\Omega\to \Omega$ and a family of unitary operators  $U_k$ on $L^2(X_E)$ by
\begin{align*}
q_e(T_k(\omega))&:= q_{e \circ k^{-1}}(\omega)  \qquad \text{for all } k\in G,\\
 (U_k f)_e(t)&:=f_{(e\circ k)}(t), \qquad (U_k^* f)_e(t)=f_{e\circ k^{-1}}(t)\qquad \text{for all } k\in G.
\end{align*}
Using $(T_k)$ and $(U_k)$ we can calculate the covariance condition and conclude ergodicity of the random operator, see section 1.2 in \cite{Stollmann-01} for more information and general theory.
Together with measurability (see remark \ref{bem_meas_op_fam}) we conclude deterministic spectrum by theorem 1 in \cite{KirschM}.
\end{proof}
Thus this model has spectrum at the lower edge of all spectra with measure one and we can apply the localization theorem.

\begin{appendix}
\section{Induction parameter}
\label{app}
In this section we want to demonstrate that the induction parameters are well defined and the stated relations between them are satisfied.
We made the following assumptions in (IP):
\begin{equation}
\tag{IP}
\left.
\begin{aligned}
&q\in \left(7d-6,7d\right), & \\
& \tau > \frac{3d}{2}-1,&\\
&\xi\in\left(2d-2,\min\left\{ 2\tau-d,\frac{q-3d+2}{2}  \right\}\right),&\\
&\alpha\in\left(  1,\min \left\{ \frac{2+2\xi}{2d+\xi},\frac{2+q}{3d+2\xi} \right\} \right) ,&\\
&\theta\in\left(\frac{q+d}{n},\frac{n+2-d-\alpha d}{\alpha n}\right),&\\
&n>9\alpha d+d-2. &  & 
\end{aligned}
\right\}
\end{equation}

\begin{bem}
If the assumptions in \ref{IP} are satisfied, the following relations used in the localization-proofs are valid:
\begin{enumerate}[label=(\roman{*})]
\item $\tau>\frac{d}{2} $, needed for the initial length scale estimate in \ref{satz:ALA}, disorder assumption, 
\item $\xi \in (0,2\tau-d) $, needed for the initial length scale estimate in \ref{satz:ALA},
\item $q<\theta n-d$, needed for the weak Wegner-estimate in \ref{bem:W},
\item $\alpha<\frac{2+2\xi}{2d+\xi}$,	needed for step 1 of the induction theorem \ref{satz:indsatz},
\item $\alpha<\frac{2+q}{3d+2\xi}$,	needed for step 3 of the induction theorem  \ref{satz:indsatz},
\item $\theta<\frac{n+2-d-\alpha d}{\alpha n} $,  needed for  step 4 of the induction theorem \ref{satz:indsatz},
\item $\alpha d-1-2\xi <0$, needed for proposition \ref{hsatz:SL:poly},
\item $\frac{d}{\alpha}+\frac{d+2}{2}-\frac{n}{\alpha}<0 $, needed for proposition \ref{hsatz:SL:poly},
\item $\frac{n}{\alpha} -\frac{d}{\alpha}-\frac{d+2}{2}>\frac{d+1}{2} $, needed for the main theorem \ref{satz:hauptsatz}.
\end{enumerate}
Moreover we show that the relations in \ref{IP} are well defined, i.\,e. the stated intervals are nonempty.
\end{bem}
\begin{proof}
\begin{enumerate}[label=(\roman{*})]
\item 
As $\tau > \frac{3}{2}d-1$ we have $\tau >\frac{d}{2}$, as $d\geq 1$.
\item The relation $2d-2\geq 0$ is clear. The interval for $\xi$ is nonempty since: 
\begin{itemize}
\item $2\tau-d>2d-2$ $\Longleftrightarrow$ $ \tau>\frac{3}{2}d-1$
\item using $q>7d-6$ we have 
\begin{align*}
\frac{q-3d+2}{2} > \frac{7d-6-3d+2}{2} =2d-2
\end{align*}
\end{itemize}
\item
Is given by $\theta >\frac{q+d}{n}$.
\item, (v) and (vi) are clear.
\item[(vii)] We have
\begin{align*}
\alpha d -1-2\xi &< \frac{2+2\xi}{2d+\xi}d -1-2\xi=\frac{2d+2\xi d-(2d+\xi)-2\xi(2d+\xi)}{2d+\xi}\\
&=\frac{-\xi-2\xi d-2\xi^2}{2d+\xi}<0
\end{align*}
\item[(ix)] The choice of $n$ in \eqref{IP} implies
\begin{align*}
&0< n-d-\frac{5}{2}\alpha d < n-d-\alpha d -\frac{3}{2}\alpha = n-d-\alpha\left(\frac{d+2}{2} \right)-\alpha\left(\frac{d+1}{2} \right)\\
&\Rightarrow \frac{n}{\alpha} -\frac{d}{\alpha}-\frac{d+2}{2}>\frac{d+1}{2}
\end{align*}
also implying (viii).
\item[(x)] We show nonempty interval for $\alpha$:
\begin{itemize}
\item 
$\frac{2+2\xi}{2d+\xi}>1 $ $\Longleftrightarrow$ $\xi >2d-2 $,
\item
$\frac{2+q}{3d+2\xi}>1\Longleftrightarrow \xi<\frac{q-3d+2}{2} $ 
\end{itemize}
\item[(xi)] Finally we show nonempty interval for $\theta$:
\begin{align*}
n&>9\alpha d+d-2, \qquad \text{with }q<7d \text{ we get}\\
n&>\alpha q +2\alpha d+d-2 \\
\Longrightarrow \frac{n+2-d-\alpha d}{\alpha n}&>\frac{q+d}{n}.\qedhere
\end{align*}
\end{enumerate}
\end{proof}

Moreover $n>19d+16$ is sufficient for $n$:
\begin{align*}
9\alpha d&<9d\cdot \frac{2+2\xi}{2d+\xi}<9d\cdot \frac{2+2\frac{q-3d+2}{2}}{2d+(2d-2)}<  9d\cdot \frac{2+2\frac{7d-3d+2}{2}}{2d+(2d-2)}\\
&\leq 9d\cdot \frac{2d+2}{2d-1}\leq 9d\cdot \frac{2d+2}{d}=18d+18\\
9\alpha d +d-2 &< 18d+18+d-2=19d+16.
\end{align*}

Also $\alpha$ can be estimated:
\begin{align*}
\alpha < \frac{2+q}{3d+2\xi}\leq \frac{2+7d}{3d+2(2d-2)}= \frac{7d-4+6}{7d-4}=1+\frac{6}{7d-4}\leq 1+\frac{6}{3}=3.
\end{align*}\end{appendix}

\printbibliography

\bigskip
\noindent
\begin{tabular}{lll}
Carsten Schubert & \hspace*{0.5cm}& and\\
Fakult\"at Mathematik & &  Fakult\"at f\"ur Mathematik und Informatik \\
Technische Universit\"at Chemnitz& &  Friedrich-Schiller-Universit\"at Jena \\
09107 Chemnitz, Germany  & & Ernst-Abbe-Platz 2, 07743 Jena, Germany\\
\end{tabular}\\
\begin{tabular}{l}
{\tt carsten.schubert@mathematik.tu-chemnitz.de}
\end{tabular}

\end{document}